\documentclass[12pt]{article}
\RequirePackage[OT1]{fontenc}
\usepackage{amsthm,amsmath}
\usepackage[numbers]{natbib}
\usepackage[colorlinks,citecolor=blue,urlcolor=blue]{hyperref}

\usepackage{amssymb}
\usepackage{color}
\usepackage{mathrsfs,graphicx}

\newtheorem{atheorem}{Theorem}

\newtheorem{maintheorem}{Theorem}
\newtheorem{thm}{Theorem}[section]
\newtheorem{ex}[thm]{Example}
\newtheorem{fact}[thm]{Fact}
\newtheorem{cor}[thm]{Corollary}
\newtheorem{con}[thm]{Conjecture}
\newtheorem{lem}[thm]{Lemma}
\newtheorem{defn}[thm]{Definition}
\newtheorem{prop}[thm]{Proposition}
\newtheorem{rem}[thm]{Remark}

\allowdisplaybreaks

\numberwithin{equation}{section}

\newcommand{\Ind}[1]{\mathbf{1}\{#1\}}

\newcommand{\pd}{\partial}

\newcommand{\eps}{\epsilon}
\newcommand{\Cov}{{\mathrm{Cov}}}

\newcommand{\ssfrac}[2]{\mbox{\footnotesize $\frac{#1}{#2}$}}
\newcommand{\half}{\ssfrac{1}{2}}

\DeclareMathSymbol{\leqslant}{\mathalpha}{AMSa}{"36} 
\DeclareMathSymbol{\geqslant}{\mathalpha}{AMSa}{"3E} 
\DeclareMathSymbol{\eset}{\mathalpha}{AMSb}{"3F}     
\renewcommand{\le}{\;\leqslant\;}                   
\renewcommand{\ge}{\;\geqslant\;}                   
\DeclareMathOperator*{\inter}{\bigcap}

\renewcommand{\epsilon}{\varepsilon}

\def\RR{{\mathcal R}}

\def\TT{{\mathbb T_d}}
\def\Tk+{{\mathbb T_d^+}}

\def\SS{{\mathcal S}}

\def\AA{{\mathcal A}}

\def\GG{{\mathcal G}}

\def\VV{{\cal V}}

\def\XX{{\cal X}}

\newcommand{\la}{{\lambda}}
\newcommand{\CT}{\mathrm{CT}}

\newcommand{\PP}{{\mathbb P}}

\def\eps{\varepsilon}

\newcommand{\oo}{\mathbf{o}}
\renewcommand{\Pr}{ \mathrm P}
\newcommand{\E}{ \mathbb E}

\newcommand{\Pois}{\mathrm{Pois}}
\newcommand{\as}{\mathrm{a.s.}}
\newcommand{\whp}{\mathrm{w.h.p.}}

\newcommand{\gd}{\delta}

\newcommand{ \cT}{ \mathcal T }

\newcommand{\SC}{ \mathrm{SC} }

\let\phi=\varphi

\def\qed{\hfill $\square$}

\def\dist{\mathop{\rm dist}}

\newcommand{\N}{\mathbb N}
\newcommand{\R}{\mathbb R}
\newcommand{\Z}{\mathbb Z}

\newcommand{\W}{\mathcal W}
\newcommand{\plant}{w_{\mathrm{plant}}}

\newcommand{\F}{{\cal F}}

\newcommand{\s}{{\cal S}}

\usepackage[normalem]{ulem}

\usepackage{color}

\begin{document}

\title{On an epidemic model on finite graphs \thanks{The authors
are thankful to Microsoft and Fapesp, grant 2017/10555-0 for financial support. The third author is thankful for EPSRC for financial support, grant EP/L018896/1.
}}

\author{I.~Benjamini$^{~1}$, L.R.~Fontes$^{~2}$, J. Hermon$^{~3}$, F.P.~Machado$^{~2}$}

\maketitle {\footnotesize {\noindent $^{~1}$Department of Mathematics, Weizmann Institute
of Science, Rehovot 76100,
Israel. \\ e-mail: {\tt itai.benjamini@weizmann.ac.il}
\smallskip

\noindent $^{~2}$Instituto de Matem\'atica e Estat\'\i stica,
Universidade de S\~ao Paulo, Rua do Mat\~ao 1010, CEP 05508--900,
S\~ao Paulo SP, Brasil.

\noindent e-mail: \texttt{lrenato@ime.usp.br, fmachado@ime.usp.br} }

\noindent ${~3}$ University of Cambridge, Cambridge, UK. E-mail: {\tt jonathan.hermon@statslab.cam.ac.uk}. 



\begin{abstract}
We study a system of random walks, known as the frog model, starting from a profile of
independent Poisson($\lambda$) particles per site, with one additional active particle planted at some vertex $\mathbf{o}$ of a finite connected simple graph $G=(V,E)$. Initially, only the particles occupying $\mathbf{o}$ are active.  Active particles perform $t \in \N \cup \{\infty \}$ steps of the walk they picked before vanishing and activate all inactive particles they hit. This system is often taken as a model for the spread of an epidemic over a population. Let $\mathcal{R}_t$ be the set of vertices which are visited by the process, when active particles vanish after $t$ steps. We study the susceptibility of the process on the underlying graph, defined as the random quantity $\SS(G):=\inf \{t:\mathcal{R}_t=V \}$ (essentially, the shortest particles' lifespan required for the entire population to get infected). We consider the cases that the underlying graph is either a regular expander or a $d$-dimensional torus of side length $n$ (for all $d \ge 1$) $\mathbb{T}_d(n)$ and determine the asymptotic behavior of $\SS $ up to a constant factor. In fact, throughout we allow the particle density $\la$ to depend on $n$ and for $d \ge 2$ we determine the asymptotic behavior of $\s(\mathbb{T}_d(n))$ up to smaller order terms for a wide range of $\la=\la_n$.
\\[.3cm]
\end{abstract}
\tableofcontents

\section{Introduction}
\label{intro} We study a system of branching random walks known
as the \emph{frog model}. The model is often interpreted as a model for a spread of an epidemic or a rumor. The frog model on infinite graphs received much attention, e.g.~\cite{telcs1999branching,alves2002phase,alves2002shape,popov2001frogs,hoffman,hoffman2,hoffman3,01}. 
As we soon explain in more detail, the focus of this work is to study a natural quantity associated with the frog model on finite graphs, called the susceptibility, which in the aforementioned interpretation of the model is meant to capture  ``how interesting should a rumor be, so that eventually everybody will hear it".

\medskip

Most of the existing literature on the model is focused on the case that the underlying graph on which the particles perform their random walks is $
\Z^d$ for some $d \ge1$, e.g.\ \cite{alves2002phase,alves2002shape,popov2001frogs,RS} (in \cite{GS,dobler,dobler2017recurrence,ghosh,rosen} the case that the particles preform walks with a drift is considered). Beyond the Euclidean setup, there has been much interest in understanding the behavior of the model in the case that the underlying graph is an infinite $d$-ary tree,  \cite{hoffman,hoffman2,hoffman3}. To the best of the authors' knowledge, the only existing papers concerning the model on finite graphs are \cite{PopovFrogReview,hermonfrog,hoffman3}, 
whose main concerns are the frog model on cycle graphs, complete graphs and regular trees. 

This paper is closely related to \cite{hermonfrog} (see \S\ref{s:fot}) and also to a paper by the first and third authors about the intimately related random walks social network model (see \S\ref{s:related}). In this paper we study the model in the case that the underlying graph $G=(V,E)$ is some finite connected simple undirected graph. More specifically, we focus mainly on the cases that $G$ is a $d$-dimensional torus ($d \ge 1$) of side length $n$ or a regular expander.

\medskip

 The frog model on $G$ with density $\la$  can be described as
follows.  Initially there are Pois($\lambda$)   particles at
each vertex of  $
G $, independently (where Pois($\lambda$) is the Poisson distribution of mean $\la$). A site of $ G $ is singled out and called its \emph{origin}, denoted by $\oo$. An additional particle, denoted by $\plant$, is planted at $\oo$. This is done in order to ensure that the process does not instantly die out.  Initially, each particle independently ``picks" an infinite trajectory, which is distributed as a discrete-time simple random walk (\emph{SRW}) on $G$ started at the particle's initial position. All
particles are inactive (sleeping) at time zero, except for those occupying the origin. Each active particle performs the first $\tau $ steps of the walk it picked (for some $\tau \in \N \cup \{\infty \} $) on the vertices of $ G $ (i.e.~for $\tau$ steps, at each step it moves to a random neighbor of its current position, chosen from the uniform distribution over the neighbor set) after which it cannot become reactivated (one may consider that they vanish).
We refer to $\tau$ as the particles' \emph{lifetime}. Up to the time a particle dies (i.e.~during the $\tau$ steps of its walk), it activates all
sleeping particles it hits along its way. From the moment an
inactive particle is activated, it performs
the same dynamics over its lifespan $ \tau $, independently of
everything else (i.e.~there is no interaction between active particles). We denote the corresponding probability measure by $\PP_{\la}$.  

\medskip

Note that (in contrast to the setup in which $G$ is infinite) $\as$ there exists a finite minimal lifespan $\tau$ (which is a function of the initial configuration of the particles and the walks they pick) for which every vertex is visited by an active particle before the process ``dies out". We define this lifespan as 
$\mathcal{S}(G)$, the \emph{susceptibility} of $G$. A more explicit definition of the susceptibility  is given in \eqref{e:SandCT}.   

\medskip

The name \emph{frog model} was coined in 1996 by
Rick Durrett. It is a particular case of the  $A+B \to 2B$ family of models (see \S\ref{s:related}). Like other models in this family (e.g.~\cite{KurkovaPopovVachkovskaia,kesten2005spread,kesten2008shape}),     it  is often
motivated as a model for the spread of a rumor or infection. Keeping this interpretation in mind, the susceptibility is indeed a natural quantity. It is essentially the minimal lifespan $\tau$ of a virus (more precisely, of an individual infected by a virus),   sufficient for wiping out the entire population.
In this interpretation, the more likely $\SS(G)$ is to be large, the less susceptible the population is.

\subsection{Organization of the paper}

The paper is organized as follows. In \S\ref{s:results} we present our main results and some
conjectures that we believe may drive future research in this subject. In \S\ref{s:pro} we introduce
some notation necessary for a formal construction of the model, present a concise introduction to the
topic of the frog model on finite graphs, introduce some related models and examples and state some additional conjectures. In \S\ref{s:cycle} we prove Theorem \ref{thm: cycle}. In \S\ref{s:Toril} we prove results concerning the cover time by multiple walks and explain how they imply lower bounds on the susceptibility.
In \S\ref{s:aux} we present some auxiliary results concerning percolation and simple random walks which are handy for the proofs of the Theorems \ref{thm: tori} and \ref{thm: et1}, which are 
delivered in \S\ref{s:Tori}-\ref{s:expanders}. 
\section{Main results and conjectures}
\label{s:results}
Below we list our main results. In Theorems \ref{thm: cycle} and \ref{thm: et1} we present  bounds on $\s$ with explicit estimates for the probability the bounds fail, which are valid for fixed graph size $n$ and particle density $\la$. In particular, we allow both $\la=\la_n$ and $\la^{-1}$ to diverge as $n \to \infty$. It is natural to  allow the particle density to vary for multiple reasons. One reason is that  a-priori it is plausible that the susceptibility exhibits a phase transition when $\lambda$ is scaled in some appropriate manner. Another, is that in the our Theorem \ref{thm: tori} we relate the susceptibility of the frog model with density $\la_n$ to the cover-time of the graph by $\lceil \lambda_nn^{d}\rceil  $ independent particles, each starting at a vertex chosen uniformly at random, independently. In the setup of the cover-time by multiple independent random walks there is no particular regime (of number of walks) that appears more interesting than others, and there is no reason to restrict to the case the number of walks is comparable to the number of vertices.      

\medskip

It is interesting to note that for each family of graphs considered in this paper, the susceptibility exhibits some fixed scaling as a function of the graph size and particle density (with a polynomial dependence on $\log n$ and $\la^{-1}$) throughout the considered regimes (i.e.~the sparse, dense and ``hyper-dense" regimes: $\la_n=o(1), \la:=\Theta(1)$ and $\la_n \gg 1 $, respectively; see \S\ref{s:notation} for our (standard) usage of asymptotic notation). In particular, the susceptibility does not exhibit a phase transition. This is in sharp contrast with the notion of the cover time for the frog model, for which recently Hoffman, Johnson and Junge established a phase transition for finite $d$-ary trees \cite{hoffman3}.  
\subsection{Tori}

We denote the $n$-cycle graph by $\mathrm{C}_n$. This is a graph on $n$ vertices containing a single cycle through all vertices. The next theorem asserts that as long as $\la_n n \gg 1$ the susceptibility of $\mathrm{C}_n $ corresponding to particle density $\la_n$ is $\whp$ (i.e., with probability tending to 1 as $n \to \infty$, see \S\ref{s:notation} for a precise definition) $\Theta(\la_n^{-2}\log^2 (\la_n n) )$ (equation \eqref{eq: cyclelower} covers only some of this range, but when valid offers a better bound than \eqref{eq: cyclelower2} on the probability that $\s(\mathrm{C}_n)$ is unusually small).
\begin{maintheorem}
\label{thm: cycle}
Let $t_{\la,n}:=( \frac{1}{\la} \log (\la n))^2 $. There exist some positive absolute constants $c,c_0,c_1,c_2,C_1,C_{2}$ such that the following hold.
\begin{align}
\forall\, \la \ge 2/n, \quad & \PP_{\la}[\SS(\mathrm{C}_n) > C_1 t_{\la,n}] \le C_2 \exp[-c_0  \log^{2/3} (\la n) ], \label{eq: cycleupper}
\\ \forall\, \epsilon \in (0,1],\,  \la \ge n^{-\frac{1-\eps}{2} }, \quad & \PP_{\la}[\SS(\mathrm{C}_n)< c_1 \eps^{2} \la^{-2}  \log^2 n] \le \exp[-c_2 \eps^{-2}  n^{\eps/3}]. \label{eq: cyclelower}
\\
\forall \,   \la >0 , \quad & \PP_{\la}[\SS(\mathrm{C}_n)< ct_{\la,n} ] \le 4 \exp \left[-\frac{n^2}{32 t_{\la,n}}\right]+ (\la n)^{-1/4} . \label{eq: cyclelower2}
\end{align} 
\end{maintheorem}

We denote the $d$-dimensional torus of side length $n$ by $\mathbb{T}_d(n) $. This is the Cayley graph of $(\Z/n \Z)^d$ obtained by connecting each $x,y \in (\Z/n \Z)^d $ which disagree only in one coordinate, by $\pm 1$ mod $n$. 
Let $(S_n)_{n=0}^{\infty}$ be SRW on $\Z^d$. Let 
\begin{equation}
\label{rhodintro}
\rho(d):=\Pr_x[T_x^+=\infty], \quad \text{where} \quad T_x^+:=\inf\{n>0:S_n=x \}.
\end{equation}

The following theorem essentially asserts that for particle density $\la$ which does not vanish nor diverge too rapidly as a function of $n$, $\whp$ we have that $ \SS(\mathbb{T}_2(n))= (1 \pm o(1) )f(n,\la) $ and  that $ \SS(\mathbb{T}_d(n))=(1 \pm o(1) )\frac{d}{ \la \rho(d) } \log n $ for $d \ge 3$, where  $f(n,\la):=\frac{2}{\pi} \la^{-1} \log n \log(\la^{-1} \log n )$.
\begin{maintheorem}
\label{thm: tori}  
\begin{itemize}
\item[(i)] Let $\la=\la_n$.  Assume that $ n^{- \delta_n }  \ll \la \ll \log n $ for some $\delta_n =o(1)$ such that $\delta_n \gg \frac{1}{\log n} $. Let $f(n,\la):=\frac{2}{\pi} \la^{-1} \log n \log(\la^{-1} \log n ). $ Then for every  $\eps>0$ 

\begin{equation}
\label{T2lower1}
\lim_{n \to \infty} \PP_{\la}\left[\left|\frac{ \SS(\mathbb{T}_2(n))}{\mathbb{E}_{\la}[\SS(\mathbb{T}_2(n))]} -  1\right| >\eps \right]=0 .
\end{equation}
\begin{equation}
\label{T2upper}
\mathbb{E}_{\la}[\SS(\mathbb{T}_2(n))]=(1 \pm o(1) )f(n,\la) .
\end{equation}
Moreover, if $ \log n \ll \la n^{2} \ll n^{2} \log n $ then for every fixed $\eps>0$ 
\begin{equation}
\label{T2lower1'}
\lim_{n \to \infty} \PP_{\la}\left[\frac{1}{16} \le \frac{ \SS(\mathbb{T}_2(n))}{f(n,\la)}< 1+ \eps \right]=1 .
\end{equation}
\item[(ii)] Let $d \ge 3$. Let $\la=\la_n$. Assume that $\log n \ll \la n^{d} \ll n^{d} \log n $. Then for every  $\eps>0$  \begin{equation}
\label{Tdlower}
\lim_{n \to \infty} \PP_{\la}\left[\left|\frac{ \SS(\mathbb{T}_d(n))}{\mathbb{E}_{\la}[\SS(\mathbb{T}_d(n))]} -  1\right| >\eps \right]=0  .
\end{equation}
\begin{equation}
\label{Tdupper}
\mathbb{E}_{\la}[\SS(\mathbb{T}_d(n))]=\left(1 \pm o(1) \right)\frac{d}{ \la \rho(d) } \log n   .
\end{equation}
\end{itemize}
\end{maintheorem}
Let  $G=(V,E)$  be a finite connected  graph.
Consider the cover time of $G$ by $m$ particles performing simultaneously independent SRWs, each starting at a random initial position chosen  uniformly at random independently, where the cover time is defined as the first time by which every vertex is visited by at least one of the particles (see \eqref{e:defofDandC} for a formal definition). Denote it by $\mathcal{D}(G,m) $. The bounds from Theorem \ref{thm: tori} have a natural interpretation. Namely, in \S\ref{s:cover} we show that for $d \ge 2$ if  $\log n \ll m=m_{n}\ll n^{d} \log n $, then  $\mathcal{D}(\TT(n),m) $ is concentrated around $f(n,m/n^2)=\frac{2n^{2}}{\pi m} \log n \log(\frac{n^2}{m} \log n ) $ for $d=2$ and around $\frac{dn^{d}}{ m \rho(d) } \log n $ for $d \ge 3$.

 It is not hard to show that this implies the corresponding lower bounds on the susceptibility by substituting $m=\la n^d $ (this is done in \S\ref{s:lowertori}). It is interesting to note that for $d \ge 3$ the cover time of a single particle is up to smaller order terms  $\frac{dn^{d}}{  \rho(d) } \log n $. Hence for $d \ge 3$ the ``speed-up" to the cover time by having   $m_{n} $ particles (each starting from stationarity) compared to the cover time by a single particle is  $(1 \pm o(1) )m$.

It is substantially harder to show that the susceptibility can be bounded from above in terms of the cover time (by a corresponding number of particles). This is done in \S\ref{s:Tori}.
We conjecture that an analog of the above relation between the susceptibility and the cover time by multiple random walks holds for all vertex-transitive graphs, and that both can be understood in terms of the decay of the return probability SRW. See Conjecture \ref{con:VT} for a precise statement, which also provide some insights about the dependence  in our results of the susceptibility on the particle density and the number of sites.

\begin{rem}
\label{rem:lalog}
Let $G=(V,E)$ be a connected $d$-regular graph. It is not hard to  verify that $\PP_{(1+\delta) d \log |V|}[  \s(G)>1 ] \le |V|^{-\gd} $ and also that for all  $\delta>0$ and $d \ge 2$ we have that $\PP_{(2d+\delta)\log n}[ \SS(\mathbb{T}_d(n)) > 1   ] \le Cn^{-\delta/2} $.  For details see \S\ref{s:lalog}. Thus there is little harm in our assumption that $\la \ll \log n$ in Theorem \ref{thm: tori} and that $\la \le \frac{1}{2} \log n $ in Theorem \ref{thm: et1} below. 
\end{rem} 

\subsection{Expanders}
We denote by $ \gamma(G)$  the \emph{spectral-gap} of SRW on $G$, defined as  the second smallest  eigenvalue of $I-P$, where $P$ is the transition matrix of SRW  on $G$ and $I$ is the identity matrix.  
We say that a sequence of graphs $G_{n}=(V_n,E_n)$ is an \emph{expander family} if $\inf_n \gamma(G_n)>0 $ and $|V_n| \to \infty$. We say that a graph $G$ is a $\gamma$-\emph{expander}
if $\gamma(G)
\ge \gamma$.

Our next
theorem essentially states that the susceptibility of an $n$-vertex regular $\gamma$-expander in the $\Pois(\la)$ frog model is $\mathrm{w.h.p.}$ $O(\frac{\log n}{\la \gamma})$.
\begin{maintheorem}
\label{thm: et1}
There exist absolute constants $c,c',C,C'>0$ and some $\delta_n=o(1)$ such that for every  $n \ge C' $ and   $\gamma \in (0,2] $, for every  regular $n$-vertex  $\gamma$-expander  $G=(V,E)$, we have
\begin{equation}
\label{e:expandersintro}
\PP_{\la}[  \SS(G) > C \la^{-1} \gamma^{-1} \log n] \le \begin{cases}e^{-c \sqrt{\log n}} & \text{if }\la \in [ C n^{-1} \log n,\frac{1}{2}\log n], \\
\delta_n & \text{if } 1/n \ll \la  \le Cn^{-1} \log n]
\end{cases} ,
\end{equation}
where this holds uniformly for all possible choices of the origin $\oo$.
\end{maintheorem}
\begin{rem}
\label{r: gammalambda}
Note that neither $\gamma$ nor $\la$ are assumed to be  bounded away from 0. We note that when $\gamma $ is bounded away from 0, Theorem \ref{thm:lower} offers a lower bound on $\s(G)$, matching up to constants.
\end{rem}

\begin{rem}
The term $\gamma^{-1}  $ in \eqref{e:expandersintro} can easily be replaced by $\gamma^{-1/2} $. The origin of this term follows from the estimate $ \sum_{t=0}^{\infty} (P_L^{t}(x,x)-\frac{1}{|V|} )\lesssim  \sum_{t=0}^{\lceil 1/\gamma \rceil} (P_L^{t}(x,x)-\frac{1}{|V|} ) \le 2/\gamma $, where $P_L:=\half (I+P)$ is the transition matrix of lazy SRW. However, using the fact that for regular graphs $P_L^{t}(x,x)-\frac{1}{|V|} \lesssim \frac{1}{\sqrt{t+1}}$ (e.g., \cite{PS}) one has the stronger inequality $\sum_{t=0}^{\lceil 1/\gamma \rceil} (P_L^{t}(x,x)-\frac{1}{|V|} ) \lesssim \frac{1}{ \sqrt{ \gamma } }$. 
\end{rem}

Note that in the statement of Theorem \ref{thm: et1}, the (common) degree of the vertices plays no role. The argument which allows us to avoid degree dependence in \eqref{e:expandersintro} involves a use of a maximal inequality. We believe the usage of maximal inequalities in the context of particle systems to be novel.

\begin{rem}
\label{rem:badstart}
Let $\RR_t(\plant)$ be the range of  the length $t$ walk performed by the planted particle $\plant$.   Given  $|\RR_t(\plant)|=i$, the probability that $\RR_t(\plant)$ is not occupied at time 0 by any other particle apart from $\plant$ is $e^{- \la i } $. Hence when bounding $\PP_{\la}[\s(G)> t]$ from above, it is necessary to argue that $|\RR_t(\plant)|$  is  $\whp$ large. The error terms in our upper bounds on $\s$ are dominated by the contribution coming from the case $|\RR_t(\plant)|$ is unusually small. 
\end{rem}

\subsection{Giant components}
Let $\mathcal{R}_t$ be the set of vertices which are visited by the process, when active particles vanish after $t$ steps. Consider a sequence of graphs $G_n:=(V_n,E_n)$ with $|V_n| \to \infty$. Another natural question is whether for some fixed $t=t_{\la}$ we have that $\PP_{\la}[|\mathcal{R}_t(G_n)| \ge \delta |V_n|] \ge \delta >0 $. While this problem is interesting by itself, a related problem will be central in proving an upper bound on $\SS$ in all the examples considered in this note, apart from the case of the $n$-cycle. Consider the case that $\plant$, the planted walker at $\oo$, walks for $t=t_{|V|}$ steps (for some $t_{|V|}$ tending to infinity as $|V| \to \infty $), while the rest of the particles have lifespan $ M$ for some constant $M=M(\la)$ (which diverges as $\la \to 0$). Denote the set of vertices which are visited by this modified process  before it dies out by  $\mathcal{R}_{t,M}$. In all of the examples analyzed in this paper, other than the $n$-cycle, we show that, $|\mathcal{R}_{t,M}|>\delta |V| $ $\whp$ for some $\delta>0$, provided that $M=M(\la)$ is sufficiently large and that $t=t_{|V|} \to \infty $ as $|V| \to \infty$. To be precise, when $\la \to 0$ then we need (1) $t \gg 1/\la $ for $\TT(n)$ for $d \ge 3$, (2)  $t \gg |\log \la|/\la $ for $\mathbb{T}_2(n)$ and (3) $t \gg 1/(\gamma \la)$ for regular $\gamma$-expanders. For precise statements see Theorems \ref{thm:GCfortori} and \ref{thm:GCexp}. 
\subsection{Conjectures}
Recall that  $\mathcal{D}(G,m) $ is the cover time by $m$ independent particles, each starting at a vertex chosen uniformly at random, independently. 
 
 \begin{defn}
\label{def: VT}
We say that a bijection $\phi:V \to V$ is an automorphism
of a graph $G=(V,E)$ if $\{u , v \} \in E $ iff $\{ \phi(u) , \phi(v)\} \in E$. A graph $G=(V,E)$ is said to be \emph{vertex-transitive} if for all $u,v \in V$ there exists an automorphism  $\phi$ of $G$ such that $\phi(v)=u$. \end{defn}

Denote the transition matrix of simple random walk $(X_t)_{t=0}^{\infty}$ on $G$ by $P$ and its stationary distribution by $\pi$. For $v \in V $ let $T_v:=\inf\{t:X_t=v \}$ be the hitting time of $v$. Denote the law of  $(X_t)_{t=0}^{\infty}$  given $X_0 \sim \pi$ by $\Pr_{\pi}$. Let
\[\varrho(\eps):=\inf\{t: \min_{v \in V} \Pr_{\pi}[T_{v} \le t ]\ge \eps \}, \qquad \hat \varrho(\eps):=\inf\{t: \max_{v \in V} \Pr_{\pi}[T_{v} \le t ]\ge \eps \}, \]
\[ \hat \nu_t:=\max_v \sum_{i=0}^{t}P^i(v,v) , \qquad \nu_t:=\min_v \sum_{i=0}^{t}P^i(v,v). \]
When $G$ is vertex-transitive, the two quantities written in each row are equal to one another. Let
\[t_{\la}(G):=\min \{s:2s/\nu_s  \ge \la^{-1} \log |V| \}, \]

\begin{con}[Re-iterated from \cite{hermonfrog}] 
\label{con:VT}
Let  $G=(V,E)$  be a finite connected vertex-transitive graph.     If $ \la |V| \ge 2 \log |V|  $ then
\begin{equation}
\label{e:con25E}
\mathbb{E}_{\la}[\SS(G)] \asymp \mathbb{E}[\mathcal{D}(G,\lceil \la |V| \rceil)] \asymp t_{\la}(G) \asymp \varrho \left(\frac{\log |V|}{\la|V|}\right). 
\end{equation}
Moreover, if $1 \ll m\ll |V| \log |V| $ then \[\mathbb{E}[\mathcal{D}(G,m)]=(1 \pm o(1) ) \varrho\left( 1-\exp\left(- \frac{\log |V|}{m} \right) \right).\]
\end{con}
For a  derivation of a lower bound on $\SS(G)$ in terms of $\mathcal{D}(G,\cdot) $ see \S\ref{s:lowertori}. We note that when $m \ge 2 \log |V| $ we have that $\frac{\log |V|}{m} \asymp 1-\exp\left(- \frac{\log |V|}{m} \right)  $. We also note that (in the vertex-transitive setup) the expected number of vertices which are not visited by any of the $m$ walks by time $t$ is given by $|V|(1-\Pr_{\pi}[T_{v} \le t ])^{m}$. Hence $ t=\varrho\left( 1-\exp\left(- \frac{\log |V|}{m} \right) \right) $ roughly corresponds to the time at which this expectation is of order 1.
\begin{con}[Re-iterated from \cite{hermonfrog}]
\label{con:VT2}
Let $G_n=(V_n,E_n)$ be  a sequence of finite connected vertex-transitive graphs of diverging sizes,  $\la_n \gg \frac{1}{|V_n|}  $\math $ and $m_n \to \infty$. Then for all $\eps>0$ we have that  \[ \lim_{n \to \infty} \mathbb{P}_{\la_n}\left[\left|\frac{\SS(G_{n})}{\mathbb{E}_{\la_n}[\SS(G_{n})]} -1\right| \ge \eps \right] = 0=\lim_{n \to \infty} \mathbb{P}\left[\left|\frac{\mathcal{D}(G_{n},m_{n})}{\mathbb{E}[\mathcal{D}(G_{n},m_{n})]} -1\right| \ge \eps \right]  .\]
\end{con}
The following proposition and remark summarize our knowledge about Conjecture \ref{con:VT}.
\begin{prop}
\label{prop:con2.5}
Let $G=(V,E)$ be a finite connected regular graph. Define \[\check s_{\la}(G):=\inf \{s \in \N : 12s \la  \ge \nu_s\log |V| \}, \]
\[\hat s_{\la}(G)  :=\inf \{s \in \N :s \la  \ge 2 \hat \nu_s \log |V| \}, \]
Let $\la$ be such that $\frac{\log |V|}{6\la|V|} \ll 1 $. Then 
\begin{equation}
\label{checksupper} 
  \check s_{\la}(G) \le  \hat\varrho \left(\frac{\log |V|}{6\la|V|}\right) =: \hat t \le \bar t:=  \varrho \left(\frac{2\log |V|}{\la|V|}\right) \le \hat s_{\la}(G) \lesssim (\la^{-1} \log |V|)^2 .
\end{equation} 
Moreover, 
\begin{equation}
\label{e:hatsup}
\PP \left[\mathcal{D}(G,\lceil \la |V| \rceil)>  \bar t   \right]  \le 1/|V|, 
\end{equation}
\begin{equation}
\label{e:checkslower}
 \PP\left[ \mathcal{D}(G,\lceil \la |V| \rceil)<\hat t\right]\lesssim  \exp \left(- \frac{|V|^{2/3-o(1)}}{  16 \la  \hat t} \right),
\end{equation}
Finally,
\begin{equation}
\label{e:checkslower2}
 \PP_{\la}\left[ \SS(G)<\hat t\right]   \lesssim  \exp \left(- \frac{|V|^{2/3-o(1)}}{  64 \la \hat t} \right)+\max_{v} \Pr_v[|\{X_i:0 \le i < \hat t  \}|> |V|/2 ] .
\end{equation}
\end{prop}
\begin{rem}
\label{r:test}
For vertex-transitive graphs one should typically expect that (in the notation from \eqref{checksupper}) $ \hat t \asymp \bar t $
  and that $\check s_{\la}(G) \asymp t_{\la}(G) \asymp \hat s_{\la}(G)   $.
\end{rem}
The term $\max_{v} \Pr_v[|\{X_i:0 \le i < \hat t  \}|> |V|/2 ] $  from the r.h.s.\ of \eqref{e:checkslower2} is meant to treat the planted particle. In
 the vertex-transitive case one can treat the planted particle using symmetry, in the spirit of what is done in \S\ref{s:lowertori}, in order to get rid of this term. In any case, this probability is 0, unless $\la$ is very small.

\medskip

 The following conjecture is motivated by the results in \cite{SN} (see \S\ref{s:related} for more details).
\begin{con}\label{con:polylog} There exist some $C_{d,\la},\ell>0$, such that for every sequence of finite connected graphs $G_{n}=(V_{n},E_{n})$ with $|V_n| \to \infty$ of maximal degree at most $d$, we have that $\lim_{n \to \infty} \PP_{\la}[\SS(G_n) \le C_{d,\la}  \log^{\ell} |V_n| ]=1$.
\end{con}  
We suspect that one can take above $\ell=2$ and $C_{\la,d}=C \la^{-2}d^2 $ for some absolute constant $C>0$. Moreover, we suspect that for regular or vertex-transitive  $G$, one can even take above, respectively, $C_{\la,d}=C \la^{-2}d$ and  $C_{\la,d}=C \la^{-2}$ for some absolute constant $C>0$ (cf.~\cite[Conjectures 1.9 and 8.3]{SN}). If true, this suggests that up to a constant the $n$-cycle is extremal. See Examples \ref{ex:1}-\ref{ex:2} for more  about the dependence of the susceptibility on $d$.

\medskip

The following conjecture concerns a strengthening of Conjecture \ref{con:polylog} for the class of ``uniformly transient graphs".

\begin{con}
There exists some non-decreasing, diverging $f: \R_+ \to \R_+$ such that for every connected $d$-regular $n$-vertex graph $G=(V,E)$,  
\[\PP_{1}[\s(G) \ge C_{d} f(R_{\mathrm{eff}}^{*}) \log |V| ] \le 1/|V|, \]
where     $R_{\mathrm{eff}}^*:=1/\min_{u \neq v \in V }\Pr_v[T_u<T_v^+] $ (i.e.~the inverse of the minimum over all pairs of distinct vertices $(u,v)$ that a SRW started from $v$ will reach $u$ before returning to $v$).
\end{con}
The function $f(x)=\max\{ \log x,1\}$ works in all examples that were studied inn the literature.
\begin{con}
\label{con:1persite}
All of our results for $\la=1$ hold also when initially we have exactly one particle per site.
\end{con}
It is likely that one can give an affirmative answer to Conjecture \ref{con:1persite}   using ideas from \cite{JJ}. Alternatively, it is plausible that with some care all our arguments translate mutatis mutandis to the case where there is one particle per site, where the main technical obstacle appears to be that Poisson thinning no longer applies. Poisson thinning is used repeatedly in our analysis, however we strongly believe our arguments can be modified as to not rely on it.

\section{Propaedeutics}
\label{s:pro}
The \emph{cover time} $\CT(G) $ of a graph $G$ w.r.t.\ the frog model is the first time by which every vertex has been visited by an active particle. See \S\ref{s:construction} for a more precise definition.

\subsection{Review of general susceptibility lower bounds }
\label{s:genlowbounds}
Theorem \ref{thm:lower} below seems to be especially useful when $G$ is vertex-transitive (see Conjecture \ref{con:VT}). The bound offered by \eqref{e:lower2} is sharp up to a constant factor in all of the cases considered in this paper (at least when $\la$ does not vanish too rapidly). 
\begin{atheorem}[\cite{hermonfrog} Theorem 3]
\label{thm:lower}
 For every finite regular simple graph\ $G=(V,E) $ and all $\la_{|V|}>0$ such that $ \lambda_{|V|}^{-1} ( \log |V|)^5 \le |V|, $ 
\begin{equation}
\label{e:lower1}
\PP_{\la_{|V|}}[\la_{|V|} \SS(G) \ge L(|V|,\la_{|V|}) ] \to 1, \quad \text{as }|V| \to \infty,
\end{equation}
\[\text{where} \qquad L(n,\la):=\log n -4 \log \log n- \log ( \max \{1/\la,1 \} ). \]
Moreover, for  all   $\delta \in (0,1)$ and all  $\la_{|V|}\gg |V|^{-\delta/7}   $  we have that
\begin{equation}
\label{e:lower2}
\PP_{\la_{|V|}}[\SS(G) \ge t_{\la_{|V|},\gd}(G) ] \to 1, \quad \text{as }|V| \to \infty,
\end{equation}
where  $t_{\la,\delta}(G):=\min \{s: 2s \la  \ge (1-\delta)\nu_s\log |V| \} $ and
$\nu_k:=\min_v \sum_{i=0}^{k}P^i(v,v)$.  
\end{atheorem}
\begin{rem}
\label{rem:t_la(G)}
We note that  $t_{\la,0}(G) \le C \la^{-2}\log^2|V| $ for every regular graph $G=(V,E)$  (e.g.~\cite[Lemma 2.4]{PS}). This is tight up to a constant factor as for the $n$-cycle $\mathrm{C}_n $ we have that $t_{\la,1/2}(\mathrm{C}_n) \ge c \la^{-2}\log^2|V| $.
\end{rem}
\begin{atheorem}[\cite{hermonfrog} Proposition 1.1]
\label{prop:Kn}
Let $K_n$ be the complete graph on $n$ vertices. Let $(\la_n)_{n \in \N}$ be such that $\lim_{n \to \infty} \la_n n = \infty $. Then
\begin{equation}
\label{e:Knintro1}
\forall \eps \in (0,1), \quad \lim_{n \to \infty} \PP_{\la_{n}}[(1-\eps)\la_{n}^{-1}\log n  \le  \SS(K_n) \le \lceil (1+\eps)\la_{n}^{-1}\log n \rceil ]=1. \end{equation}
Moreover, there exists some $C>1 $ such that for every $(\la_n)_{n \in \N}$ 
\begin{equation}
\label{e:rem1.2}   \lim_{n \to \infty} \PP_{\la_{n}}[   \CT(K_n) \le C(\Ind{\la_n \le 2 }  \la_{n}^{-1} \log n +\Ind{\la_n > 2 } \lceil  \log_{\la_{n}} n \rceil)]=1.
  \end{equation}
\end{atheorem}

In light of Theorem \ref{prop:Kn} and \eqref{e:lower1},  $K_n$ is the regular graph with asymptotically the smallest susceptibility (at least when $\la  \ge |V|^{-o(1)} $).

\subsection{Frogs on trees}
\label{s:fot}
Beyond the Euclidean setup, there has been much interest in understanding the behavior of the model in the case that the underlying graph is a $d$-ary tree, either finite of depth $n$ or infinite, denoted by  $\cT_{d,n}$  and $\mathcal{T}_d $, respectively. In a sequence of dramatic papers  Hoffman, Johnson and Junge  \cite{hoffman,hoffman2,johnson,hoffman3} showed that the frog model on $\mathcal{T}_d $ indeed exhibits a phase transition w.r.t.\ $\la$. Namely, below a critical density of particles it is $\as$ transient, and above that density it is $\as$ recurrent. Johnson and Junge \cite{johnson} showed that the critical density grows linearly in $d$. In an ingenious recent work together with Hoffman \cite{hoffman3}  they showed that for $\la \ge Cd^2$ the frog model on $\mathcal{T}_d$ is strongly recurrent (namely, that the occupation measure of the origin at even times stochastically dominates some homogeneous Poisson process).
As an application, they showed that  $\whp$ $\CT(\cT_{d,n}) \le C_{d}\la^{-1} n \log n $ for $\la \ge C_{0}d^2$, while $\CT(\cT_{d,n}) \ge \exp(c_{d,\la}\sqrt{n} )$ for $\la \le d/100$.
The main results in \cite{hermonfrog} are
\begin{atheorem}[\cite{hermonfrog} Theorems 1-2]
\label{thm:tree}
There exist some absolute constants $C,c>0$ such that  for all $d \ge 2$, if $ d^{-n}  n^2 \log d \le \la_{n} \le c \log n $ then
\[ \lim_{n \to \infty} \PP_{\la_{n}}\left[c\left(\frac{n}{\la_n}  \log \frac{n}{\la_n}\right)  \le   \SS(\cT_{d,n})\le C\left(\frac{n}{\la_n}  \log \frac{n}{\la_n}\right) \right] = 1. \]
 \[ \lim_{n \to \infty} \PP_{\la_{n}}\left[  \CT(\cT_{d,n})\le Cn  \max \left\{1, \frac{1}{\la_n}  \log \frac{n}{\la_n}\right\}  3^{3\sqrt{ \log |\VV_{d,n}|  }} \right] = 1. \] 
\end{atheorem}

 Note that the bounds  from  \cite{hoffman3}  complement these bounds and match them up to a constant factor for  $\la \ge C_{0}d^2$ (as $\CT(G) \ge \s(G)  $) and up to the value of the constant in the exponent for $\la \le d/100$. Also, observe that combining the results from \cite{hoffman3} with those from \cite{hermonfrog} one can readily see that $\CT(\cT_{d,n})$ exhibits a phase transition w.r.t.\ $\la$. We do not expect the susceptibility to exhibit a phase transition in any natural family of graphs. We also see that the cover time, which is always as large as the diameter,  may be very large, in contrast to our Conjecture \ref{con:polylog} for the susceptibility.

We strongly believe that for $\la=1$ one has for all $d \ge 2$ that $\whp$ $\CT(\TT) \le C dn $ (note that the diameter of $\TT$ is proportional to $nd$ and so also $\CT(\TT) \ge c dn $). The intuition comes from the fact that the frog model on $\mathbb{Z}^d$ satisfies a shape theorem \cite{alves2002shape,alves3,RS}.    


\subsection{Related models}
\label{s:related}
The  $A+B \to 2B$ family of models are defined by the following rule: there are type $A$ and $B$ particles occupying a graph $G$, say with densities $\la_A,\la_B>0$. They perform independent either discrete-time SRWs with holding probabilities $p_A,p_B \in [0,1]$ or continuous-time SRWs with rates $r_A,r_B \ge 0$  (possibly depending on the type).  When a type $B$ particle collides with a type $A$ particle, the latter transforms into a type $B$ particle.
The frog model can be considered as a particular case of  the above dynamics in which the type $A$ particles are immobile ($p_A=1$ or $r_A=0$).  

 In a series of papers Kesten and Sidoravicius \cite{kesten2005spread,kesten2006,kesten2008shape} studied in the continuous-time setup the set of sites visited by time $t$  by a type $B$ particle in the   $A+B \to 2B$  model when the underlying graph is the $d$-dimensional Euclidean lattice $\Z^d$,  $r_A,r_B>0$ and initially there are $B$ particles only at the origin. In particular, they proved a shape theorem for this set when $r_A=r_B$ and $\la_A=\la_B $ \cite{kesten2008shape} (and derived bounds on its growth in the general case \cite{kesten2005spread}). An analogous shape theorem for the frog model on $\Z^d$ was
proven by Alves, Machado, and Popov in discrete-time \cite{alves2002shape,alves3} and by Ram\' irez
and Sidoravicius in continuous-time \cite{RS}.

Even when $p_A<1$ (or in continuous-time $r_A>0$), one may consider the case in which the $B$ particles have lifespan $t$ and initially only the particles at some vertex $\oo $ are of type $B$ (and we may plant a $B$ particle at $\oo$). One can then define the susceptibility in an analogous manner, as the minimal lifespan of a $B$ particle required so that all particles are transformed into  $B$ particles before the process dies out. Similarly, one can define it as the minimal lifespan of a $B$ particle required so that all sites are visited by a $B$ particle before the process dies out.

 We strongly believe that all of the results presented in this paper can be transferred into parallel results about the case of $p_A<1$. Moreover, we also believe that the corresponding versions of  Conjectures \ref{con:VT}-\ref{con:polylog} are true also in the case of $p_A<1$.

\medskip

In \cite{SN}, the first and third authors study the following model for a social network, called the random walks social network model, or for short, the \emph{SN} model. Given a graph $G=(V,E)$,  consider Poisson($|V|$) walkers performing independent lazy simple random walks on $G$ simultaneously, where the initial position of each walker is chosen independently w.p.~proportional to the degrees.
When two walkers visit the same vertex at the same time they are declared to be
acquainted. The social connectivity time, $\SC(G)$, is defined as the first time in which there
is a path of acquaintances between every pair of walkers.  The main result in  \cite{SN} is\  that when the maximal degree of $G$ is $d$, then $\whp$ \begin{equation}
\label{eq:mainSN} 
c \log |V| \le \SC(G) \le C_d \log^3 |V|.
\end{equation} Moreover, $\SC(G)$ is determined up to a constant factor in the case that $G$ is a regular expander and in the case it is the $n$-cycle. 

Note that Conjecture \ref{con:polylog} is the analog of \eqref{eq:mainSN} for the frog model (obtained by replacing $\SC(G)$ above with $\SS(G) $). In many examples,    $\mathbb{E}[\SC(G) ]$ and $\mathbb{E}[\SS(G) ]$ are of the same order (when $\la$ is fixed), and several techniques from \cite{SN} can be applied successfully to the frog model. Namely, the same technique used in  \cite{SN}   to prove general lower bounds on $\SC(G)$ is used in the proof of Theorem \ref{thm:lower}. Moreover, the analysis of the two models on  expanders and on $d$-dimensional tori ($d \ge 1$) are similar (in all of these cases  $\SS(G)$ and  $\SC(G)$ are $\whp$ of the same order). 

\subsection{Notation}
\label{s:notation}

Let $G=(V,E)$ be some finite graph. For SRW on a graph $G$, the \emph{hitting time} of a set $A \subset V$ is $T_A:=\inf \{t \ge 0 :X_t \in A \}$. Similarly, $T_A^+:=\inf \{t \ge 1 : X_t \in A \} $. When $A=\{x\}$ is a singleton, we instead write $T_x$ and $T_x^+$.  Let $P$ be the transition kernel of SRW on $G$. We denote by $P^t(u,v)$ the $t$-steps transition probability from $u$ to $v$. We denote by $\Pr_u$ the law of the entire walk, started from vertex $u $. We denote the uniform distribution on $V$ (resp.~$U \subset V$) by $\pi$ (resp.~$\pi_U$). When we want to emphasize the identity of the graph, we write $P_G^t,\Pr_x^G $ and $\mathbb{E}_x^G$ rather than $P^t,\Pr_x$ and $\mathbb{E}_x $. When certain expressions are independent of the initial point of the walk we sometimes omit it from the notation. Similarly, when we want to emphasize the identity of the base graph for the frog model, we write $\PP_{\la}^G $ and $\mathbb{E}_{\la}^G$. When certain events involve only the planted particle $\plant$, we often omit the parameter $\la$ from the subscript. 

\medskip

Consider the frog model on $G$ with particle density $\la$ and lifespan $\tau$. Recall that $\RR_{\tau}$ is the collection of vertices which are visited by an active particle before the process corresponding to lifespan $\tau$ dies out.  Denote the collection of particles whose initial position belongs to a set $U \subseteq V $ (resp.~is $v \in V$) by $\W_U$ (resp.~$\W_v$). Then $(|\W_v|-\Ind{v = \oo })_{v \in V}$ are i.i.d.~$\Pois(\la)$. Denote the collection of all particles by $\W=\W(V)$. Denote the range of the length $\ell$ walk picked (in the sense of \S\ref{s:construction}) by a particle $w$  by $\RR_{\ell}(w)$. Denote the union of the ranges of the length $\ell$ walks picked  by the particles whose initial positions lie in $U \subset V $ (resp.~is $v$) by $\RR_{\ell}(U):=\RR_{\ell}( \W_U)$ (resp.~$\RR_{\ell}(v):=\RR_{\ell}( \W_v)$). Denote the union of the ranges of the length $\ell$ walks picked  by the particles belonging to some set of particles $\mathcal{U} \subseteq \W $ by $\RR_{\ell}(\mathcal{U}) $.

 For every event $A$ we denote its complement by $A^c$. For $U \subseteq V$ we denote $U^c:=V \setminus U$. For $\mathcal{U} \subseteq \W$ we denote the collection of particles which do not belong to $\mathcal{U} $ by $\mathcal{U}^c:=\W \setminus \mathcal{U}$.

\medskip

The distance $ \dist(x,y)$ between vertices~$ x
$ and~$ y $ is the minimal number of edges along a path from $ x $ to $ y $. Vertices
are said to be neighbors if they belong to a common edge. We write $[k]:=\{1,2,\ldots,k \}$ and $]k[:=\{0,1,\ldots,k \}$. We denote the cardinality of a set $A$ by $|A|$. 
We write \emph{w.p.}~as a shorthand for ``with probability".

We use $C,C',C_0,C_1,\ldots$ (resp.~$\delta,\eps, c,c',c_0,c_1,\ldots $) to denote positive absolute constants which are sufficiently large (resp.~small) to ensure that a certain inequality holds. Similarly, we use $C_{d},C_{\la,d}$ (resp.~$c_{d},c_{\la,d}$) to refer to sufficiently large (resp.~small) positive constants, whose value depends on the parameters appearing in subscript. Different appearances of the same constant at different places may refer to different numeric values.

We write $o(1)$ for terms which vanish as $n \to \infty$ (or as some other parameter, which is clear from context, diverges). We write $f_n=o(g_n)$ or $f_n \ll g_n$ if $f_n/g_n=o(1)$. We write $f_n=O(g_n)$ and $f_n \lesssim g_n $ (and also $g_n=\Omega(f_n)$ and $g_n \gtrsim f_n$) if there exists a constant $C>0$ such that $|f_n| \le C |g_n|$ for all $n$. We write  $f_n=\Theta(g_n)$ or $f_n \asymp g_n$ if  $f_n=O(g_n)$ and  $g_n=O(f_n)$. If $a(\bullet )$ and $b(\bullet )$ are two functions from a certain class of finite graphs $\mathcal{G}$ to $\R_+$ we write $a \asymp b$ if for all $G \in \mathcal{G}$ we have that $1/C \le a(G)/b(G) \le C $ for some $C \ge 1$.

We say that a sequence of events $A_n$ defined with respect to some probabilistic model on a sequence of graphs $G_n:=(V_n,E_n) $ with $|V_n| \to \infty $ holds $\whp$ (``with high probability")  if the probability of $A_n$ tends to 1 as $n \to \infty$. 

\subsection{A formal construction of the model}
\label{s:construction}
In this section we present a formal  construction of the frog model. In particular  the susceptibility is defined explicitly in \eqref{e:SandCT}. In what comes we shall frequently refer to ``the walk picked by a certain particle". This notion is explained in the below construction. We also recall the notion of Poisson thinning, which is used repeatedly throughout the paper.

Clearly, in order for the susceptibility to be a random variable, the probability space should support the model simultaneously for all particle lifetimes. In order to establish the fact   that the laws of susceptibility is stochastically decreasing in $\la$,
 below we show that the probability space can be taken to  support the model simultaneously also for all particle densities. As this is a fairly standard construction, most readers may wish to skim this subsection.

 We denote the set of  $\Pois(\la)$ (or $1+\Pois(\la)$ for the origin)  particles occupying vertex $v$ at time 0 by $\W_{v}=\{w_1^v,\ldots w_{|\W_v|}^v \}$, where $\mathbf{W}:= (|\W_v|-\Ind{v=\oo})_{v \in V}$ are i.i.d.\ $\Pois (\la)$. We can assume that at time 0 there are infinitely many particles  $\mathcal{J}_v:=\{w_i^v:i \in \N \}$  occupying each site $v$ (where $w_i^v$ is referred to as the $i$th particle at $v$), but only the first $|\W_{v}|$ of them are actually involved in the dynamics of the model. We may think of each particle $w_i^v \in \mathcal{J}_v $ as first picking  an infinite SRW $\mathbf{S}^{v,i}:=(S_t^{v,i})_{t \in \Z_+}$ according to $\Pr_v$, where $\mathbf{S}:= (\mathbf{S}^{v,i})_{v \in V,i \in \N}$ and $\mathbf{W} $ are jointly independent. However, only in the case that $i \le |\W_v| $  and $v $ is visited by some active particle, say the first such visit occurs at time $s$, does $w_i^v$ actually perform the first $\tau$ steps of the SRW it picked (i.e.~its position at time $s+t$ is $S_t^{v,i}$ for all $t \in [\tau]$).

Consider a collection of rate 1 Poisson processes   $\mathbf{M}:= (M_v(\bullet ))_{v \in V}$ on $\R_{+}$ (i.e.\ for $(M_v(t))_{t \ge 0}$ is a rate 1 Poisson process on $\R_{+}$ for each $v \in V $), such that $\mathbf{S}$ and $\mathbf{M}$ are jointly independent. We can define above  $\mathbf{W}=\mathbf{W}^{\la}:= (M_v(\la))_{v \in V}$. From this construction it is clear that the law of $\SS(G)$ is stochastically decreasing in $\la$.

\medskip

Let $G=(V,E)$ be a graph.  A \emph{walk} of length $k$ in $G$ is a sequence of $k+1$ vertices $(v_0,v_1,\ldots,v_{k})$ such that $\{ v_{i} , v_{i+1}\} \in E $ for all $0 \le i <k.$ Let $\Gamma_k$ be the collection of all walks of length $k$ in $G$. We say that $w_i^v \in \W_v$ \emph{picked} the path  $\gamma=(\gamma_0,\ldots,\gamma_k) \in \Gamma_{k}$ if $S_t^{v,i}=\gamma_t $ for all $t \in ]k[ $. For each $\gamma \in \Gamma_k $ let $W_{\gamma}$ be the collection of particles in $\W_{\gamma_0} \setminus\{\plant\} $  which picked the walk $\gamma$.  For a walk $\gamma=(\gamma_0,\ldots,\gamma_k) \in \Gamma_{k}$ for some $k \ge 1$, we let $p(\gamma):=\prod_{i=0}^{k-1}P(\gamma_i,\gamma_{i+1}) $. By Poisson thinning we have that for every fixed $k$, the joint distribution of $(|W_{\gamma}| )_{\gamma \in \Gamma_k}$ (under $\PP_\la$) is that of independent Poisson random variables with $\mathbb{E}_{\la}[|W_{\gamma}|]=\la p(\gamma)$ for all $\gamma \in \Gamma_k$. 

\medskip

  For distinct $ x,y \in V $ and $\tau \in \N \cup \{\infty \}$  let 
\begin{equation}
\label{e:elltau}
 \ell_{\tau}(x,y) := \inf\{j \le \tau : S_j^{x,i} = y \quad \text{for some} \quad i \le |\W_x| \}
\end{equation}
 (employing the convention that $\inf \eset := \infty $). The \emph{activation time} of~$x$ (and also of $\W_x$) w.r.t.~lifespan $\tau$ is 
\begin{equation}
\label{e:AT}
\mathrm{AT}_{\tau}(x) := \inf\{\ell_{\tau}(x_0,x_1)+\cdots+\ell_{\tau}(x_{m-1},x_m)\},
\end{equation}
where the infimum is taken over all finite sequences $\oo=x_0, x_1,
\ldots, x_{m-1},x_m=x$ where $ x_i \in V$. Then (for lifespan $\tau$) when finite, $\mathrm{AT}_{\tau}(x) $ is precisely the first time at which $x$ is visited by an active particle, while
$\mathrm{AT}_{\tau}(x)=\infty$ iff  site~$x$ is never visited by an
active particle. The \emph{susceptibility} of $G$ can now be rigorously defined as
\begin{equation}
\label{e:SandCT}
\SS(G):=\inf \{\tau:\max_{v \in V}\mathrm{AT}_{\tau}(v)<\infty \}. 
\end{equation}
The \emph{cover time} of $\GG$ is the first time by which every vertex has been visited by an active particle. It can be defined as
\begin{equation}
\label{e:CTdef}
\mathrm{CT}(\GG):=\max_{v \in \VV}\mathrm{AT}_{\infty}(v).  
\end{equation}

\subsection{Examples}
\label{s:examples}

We now present a couple of examples  with a large $\SS$,
 demonstrating that   $\SS(G) $ may grow at least linearly as a function of the maximal degree of $G$, even if $G$ is regular. 
\begin{ex}
\label{ex:1}
Let $G_{n} $ be the graph obtained by attaching a distinct vertex to each site of the complete graph on $n$ vertices. It is not hard to see that $\whp$ $c \le \frac{ \la \SS(G_{n})}{n\log n} \le C $ for all fixed  $\la>0 $.
\end{ex}
The following example is borrowed from \cite{SN}.
\begin{ex}
\label{ex:2}
 Fix some $2 \le d$ and $n$ such that $2d \le n$. Let $J_k$ be a graph obtained from the complete graph on $k$ vertices by deleting a single edge. Consider $\lceil n/d \rceil $ disjoint copies of $J_{d}$: $I_0,\ldots,I_{\lceil n/d\rceil-1 }$, where  for all $0 \le j < \lceil n/d \rceil$, $I_j$ is connected to $I_{j+1}$ (where $j+1$ is defined $\mathrm{mod}\lceil n/d \rceil $) by a single edge that connects two degree $d-1$ vertices. This can be done so that the obtained graph, denoted by $H_{d,n} $, is $d$-regular. We argue that \begin{equation}
\label{e:H(d,n)}
\mathbb{E}_{\la}[\SS(H_{d,n})] \ge c \min \{  \max \{ ds, s^2 \}, n^2\} \text{ where } s=s_{\la,d,n}:=\frac{1}{\la} \log \left(\frac{ \la n}{d}\right).\end{equation}
\end{ex}
We provide a sketch of proof of \eqref{e:H(d,n)} in \S\ref{s:H(d,n)}.
\begin{con}
Let $d \ge 2$ and $\la>0$. Let $\mathfrak{G}(n,d)$ be the collection of all $n$-vertex $d$-regular connected graphs. Then  
\[\max_{G \in \mathfrak{G}(d\lceil n/d\rceil ,d)} \mathbb{E}_{\la}[\SS(G) ]=(1+o(1)) \mathbb{E}_{\la}[\SS(H_{d,n}) ] . \]
Moreover, if $(d_n)_{n \in \N}$ diverges and $d_n \le n$ for all $n $, then for all $(\la_n)_{n \in \N}$ we have that
\[\max_{G \in \mathfrak{G}(d_{n}\lceil n/d_{n}\rceil ,d_{n})} \mathbb{E}_{\la_{n}}[\SS(G) ]=(1+o(1)) \mathbb{E}_{\la_{n}}[\SS(H_{d_{n},n}) ] . \] 
\end{con}

\section{The cycle - Proof of Theorem \ref{thm: cycle}}
\label{s:cycle}
In this section we consider the case that $G$ is the $n$-cycle, $\mathrm{C}_n=(V(\mathrm{C}_n),E_n)$, and prove Theorem \ref{thm: cycle}.
\subsection{The lower bounds}
To prove the lower bounds \eqref{eq: cyclelower} and \eqref{eq: cyclelower2} we bound the susceptibility from below by the cover time when initially all particles are activated. For \eqref{eq: cyclelower} we look at a collection $J$ of $\Theta(n/t) $ vertices of distance at least $2t+1 $ from one another, where   $t \asymp \la^{-2} \log^2 n $ and exploit the fact that the number of particles to visit site $j \in J $ by time $t$ are independent for different vertices in $J$. The proof of \eqref{eq: cyclelower2} requires a more subtle variance estimate.    

\noindent \textbf{Proofs of \eqref{eq: cyclelower} and \eqref{eq: cyclelower2}:} Let $\eps \in (0,1)$ and $\la \ge n^{-(\frac{1-\epsilon}{2})}$.
Let $t=t_{n}=c_1 \eps^{2} \la^{-2} \log^2 n $ for some constant $c_1$ to be determined later. 
For a vertex $v \in V(\mathrm{C}_n)$ let \[\mathrm{U}_t(v):=|\{w \in \W \setminus \W_v :v \in \RR_t(w) \}| \] be the number of particles with initial positions other than $v$  that picked a walk (in the sense of \S\ref{s:construction}) that visits  $v $ in its first $t$ steps. Note that if $J \subset V(\mathrm{C}_n)$ satisfies that  $\mathrm{dist}(a,a')>2t$ for all $a,a' \in J$ then   $(\mathrm{U}_t(a))_{a \in J}$  are jointly independent. Let $V':=V(\mathrm{C}_n) \setminus \{v: 1 \le \mathrm{dist}(v, \oo) \le t \} $. By symmetry    $(\mathrm{U}_t(v))_{v \in V'}$ are identically distributed.  

 Consider a  collection $J \subset V ' \setminus \{ \oo \} $ of vertices, which are all of distance at least $2t+1$ from one another, of size at least $\frac{n}{3t}-1 $. We argue that in order to prove \eqref{eq: cyclelower} it suffices to show that $c_1$ can be chosen so that
\begin{equation}
\label{e:pU}
p=p(n,t,\la):=\PP_{\la}[\mathrm{U}_t(\oo)=0] \ge  n^{-\eps/2}.
\end{equation}
Indeed, $\{\s < t \} \subseteq \inter_{a \in J}\{\mathrm{U}_t(a) \neq 0 \} $. Hence if $p \ge  n^{-\eps/2}$ (which implies that $|J|p \ge c_2 \eps^{-2} n^{\eps/3}$) then \eqref{eq: cyclelower} holds, as by independence and symmetry  we have
\[\PP_{\la}[\s < t] \le (1-p)^{|J|} \le \exp (-|J|p) \le \exp (-c_2 \eps^{-2} n^{\eps/3} ).    \] 

We now prove \eqref{e:pU}. Let $\nu_t:=\sum_{i=0}^{t}P^i(v,v)\asymp \sqrt{t+1} $ and $\nu_t(u,v):=\sum_{i=0}^{t}P^i(u,v) $. Observe that $\nu_{2t}(u,v)=\sum_{j=1}^{2t} \mathbb{P}_{u}[T_v = j] \nu_{2t-j} $ for all $u \neq v$ and so \begin{equation}
\label{e:nut}
\nu_t \mathbb{P}_{u}[T_v \le t]=\sum_{j=1}^t \mathbb{P}_{u}[T_v = j] \nu_t \le \sum_{j=1}^t  \nu_{2t-j} \mathbb{P}_{u}[T_v = j]    \le \nu_{2t}(u,v).
\end{equation}
 By reversibility  $\sum_{u \in V(\mathrm{C}_n)  }P^i(u,v)=1$ for all $i$ and $v \in V(\mathrm{C}_n)$. This (used in the penultimate inequality below to argue that $\sum_{u:u \neq v}\sum_{i=1}^{2t}P^i(u,v) \le 2t $) together with   Poisson thinning and \eqref{e:nut} (used in the first inequality below) imply that  for all $v \in V' $,  $\mathrm{U}_t(v)$  has a Poisson distribution  with mean \[\mu_t:=\la \sum_{u:u \neq v }\mathbb{P}_{u}[T_v \le t] \le \frac{\la}{\nu_t}\sum_{u:u \neq v}\sum_{i=1}^{2t}P^i(u,v) \leq  \frac{C \la }{\sqrt{t+1}} \cdot (2t) \le C' \la \sqrt{t}.  \]
Thus if $c_1$ is chosen so that $t \le (\frac{ \eps \log  n}{2C' \la} )^2$ we get that $\mu_t \le \frac{\eps}{2} \log  n $ and so the probability that $\mathrm{U}_t(v)=0$   is at least $e^{-\mu_t} \ge n^{-\eps/2}$, as desired. This concludes the proof of \eqref{eq: cyclelower}.

\medskip

We now prove \eqref{eq: cyclelower2}. We employ the same notation as above. Let $r:=c \la^{-2} \log^2 (\la n) $ for some $c>0$ to be determined later. We now show that $\PP_{\la}[|\RR_{r}(\plant)|>n/2 ] \le 4 e^{-n^2/ 32r}$. Indeed, if $S_k:=\sum_{i=1}^k \xi_i $, where $\xi_1,\xi_2,\ldots$ are i.i.d.\ which each equal to $\pm 1$ with probability $1/2$ and $M_k:=\max_{i \le k}S_k$, then by the reflection principle and symmetry  \[\PP_{\la}[|\RR_{r}(\plant)|>n/2 ]\le 2 \Pr[M_{r}>n/4] \le 4 \Pr[S_r \ge n/4] \le 4e^{-n^2/(32 r)},  \] where the last inequality follows by \eqref{eq: LDSRW}.
Fix some $A \subset V(\mathrm{C}_n) $ of size at least $n/2$. For the remainder of the proof of \eqref{eq: cyclelower2} we condition on the event that $V(\mathrm{C}_n)  \setminus \RR_{ r}(\plant)=A $ (however, we shall not write this conditioning explicitly). Let $Y_a:=\Ind{\mathrm{U}_{ r}(a)=0}$ and $Y:=\sum_{a \in A}Y_a$.
By Chebyshev's inequality 
\begin{equation}
\label{e:laA}
\PP_{\la}[Y=0] \le \frac{\mathrm{Var}_{\la} Y}{(\E_{\la}[Y])^{2}} \le \frac{1}{\E_{\la}[Y]} +\frac{\sum_{a,b \in A:a \neq b }\mathrm{Cov}(Y_a,Y_b)}{(\E_{\la}[Y])^2}.
\end{equation}
We will show that $\E_{\la}[Y ] \gg 1 $ and that $\sum_{a,b \in A:a \neq b }\mathrm{Cov}(Y_a,Y_b) \ll (\E_{\la}[Y_a ] )^{2}  $.

  Let $D_{a,b}$ be the event that there exists some particle $w \notin \W_{a} \cup \W_b$ such that $\{a,b\} \subset \RR_{r}(w)$ (note that $a$ and $b$ need not be adjacent). It is not hard to see that by Poisson thinning, conditioning on $D_{a,b}^c$ can only increase the probability that $\mathrm{U}_{r}(a)=0$. Moreover, this remains true even if we condition further also on $\mathrm{U}_{r}(b)>0 $. That is 
\begin{equation*}
\begin{split}
\PP_{\la}[\mathrm{U}_{r}(a)=0 \mid \mathrm{U}_{r}(b)>0, D_{a,b}^c] & =\PP_{\la}[\mathrm{U}_{r}(a)=0 \mid D_{a,b}^c] \\ & \ge \PP_{\la}[\mathrm{U}_{r}(a)=0 ]=\E_{\la}[Y_a ].  
\end{split}
\end{equation*}
 Since $\E_{\la}[Y_a \mid (1-Y_b) \Ind{D_{a,b}^c} \neq 0 ]=\PP_{\la}[\mathrm{U}_{r}(a)=0 \mid \mathrm{U}_r(b)>0, D_{a,b}^c]  $ we get that \[ \E_{\la}[Y_a(1-Y_b) ]=\E_{\la}[Y_a(1-Y_b) \Ind{D_{a,b}^c} ] \ge \E_{\la}[Y_a ]\PP_{\la}[\mathrm{U}_{r}(b)>0, D_{a,b}^c].\] Thus
\begin{equation*}
\begin{split}
\label{e:Dab2}
\Cov( Y_a,Y_b ) & =-\Cov( Y_a,1-Y_b ) \\ & \le \E_{\la}[Y_a]\PP_{\la}[\mathrm{U}_{r}(b)>0 ] -\E_{\la}[Y_a ]\PP_{\la}[\mathrm{U}_{r}(b)>0, D_{a,b}^c] \\ & \le \E_{\la}[Y_a] \PP_{\la}[D_{a,b}] = p  \PP_{\la}[D_{a,b}].
\end{split}
 \end{equation*}
By the proof of \eqref{eq: cyclelower} the expected number of particles $w \notin \W_a $ such that  $a \in \RR_{r}(w) $ is at most $\mu_{r} \le C' \la \sqrt{r}$.  By Poisson thinning the number of particles which reached $b$ by time $r$ after reaching $a$ has a Poisson distribution with mean $\mu_{a \to b } \le \mu_{r }\Pr_a[T_b \le r] $.  By symmetry $\mu_{a \to b }=\mu_{b \to a }$ and so
\begin{equation}
\label{e:Dab}
\PP_{\la}[D_{a,b}] \le \min \{1, 2(1- \exp(-\mu_{r}  \Pr_a[T_b \le r])\} \le \min \{2\mu_{r} \Pr_a[T_b \le r],1\}.
\end{equation}
Recall that $\Pr_a[T_b \le r ] \le \exp [-c_{2}^{} \mathrm{(dist}(a,b))^{2}/r ]$ for all $a,b \in V(\mathrm{C}_n)$ (this follows from the reflection principle). Hence \[M:=\max_{a \in A}\sum_{b \in A \setminus \{a\} }  \PP_{\la}[D_{a,b}] \le C_{0} \sqrt{r \log ( \max \{ \mu_{r},e \} ) } = \hat C_{0} \sqrt{ r \log( \max \{ \la \sqrt{r},e\})},  \]which yields that \begin{equation}
\label{e:VY}
\mathrm{Var}_{\la} Y \le \E_{\la}[Y]+ \sum_{a,b \in A:a \neq b}\Cov( Y_a,Y_b ) \le \E_{\la}[Y]+2|A|p M \le 3M \E_{\la}[Y] .
\end{equation}
\[\frac{\mathrm{Var}_{\la} Y }{(\E_{\la}[Y])^{2}} \le \frac{3M}{\E_{\la}[Y]} \le \frac{C_0' \sqrt{ r \log( \max \{ \la \sqrt{r},e\})} }{n \exp (-C \la \sqrt{r})}. \quad \]  
Substituting $r=c \la^{-2} \log^2 (\la n)$ and simplifying, we obtain \eqref{eq: cyclelower2}, provided that $c$ is taken to be sufficiently small.
\subsection{The upper bound}

Let $s=s_n:=C_1 \la^{-2} \log^2 (\la n) $, where $C_1$ shall be determined shortly. We now prove \eqref{eq: cycleupper} that $\PP_{\la}[ \SS(\mathrm{C}_n) \le s_{n}] \le e^{-c  \log^{2/3}( \la n)}$ provided  that $\la \ge M/n$, for some large constant $M$ (by picking $C_2$ to be sufficiently large,  \eqref{eq: cycleupper} trivially holds for $\la \in [\frac{2}{n},\frac{M}{n}) $).  Let  $k_n:=\lceil10  \la^{-1} \log ( \la n ) \rceil $.  For a vertex $v$ let $v_{\mathrm{r}}$ be the vertex which is of distance $k_n$ from $v$ from its right. Denote the line segment of length $k_n $ to the right of $v_{\mathrm{r}}$ by  $R_v $ (this is the segment consisting of all vertices of distance between $k_n$ and $2k_n-1$ from $v$ from its right). Let $A_v=A_v(\la,n)$ be the event that $v \in \RR_s(R_v)  $ (i.e.~there is at least one particle whose starting position is in the interval $R_v$ which picked a walk that reaches $v$ by its $s$-th step).

Let $\ell:=\lfloor n/k_n \rfloor +1$. Fix a collection of vertices $u_1,u_2,\ldots,u_{\ell} $ such that for all $i$ we have that $u_{i+1}$ is of distance $k_n$ to the right of $u_i$. Let  $A:=\cap_{i=1}^{\ell}A_{u_{i}}$.   
It is not hard to verify that if  $C_1$ is taken to be sufficiently large, then for every $u \in R_v  $ the probability that a SRW starting from $u$ would reach $v$ by time $s_n$ is at least $0.4$ (in fact, we could have replaced 0.4 by any fixed number smaller than 1). Fix such $C_1$. By Poisson thinning we have for all $v$,  \[\PP_{\la}[A_v^c] \le   \prod_{u\in R_v} e^{-\la \Pr_u[T_v \le s]} \le e^{-0.4 \la k_n}=e^{-0.4 \la \left\lceil \frac{10}{\la} \log ( \la n ) \right\rceil } \le \frac{1}{(\la n)^{4}}. \]  Thus by a union bound $\PP_{\la}[A^{c}] \le \ell (\la n)^{-4} \le (\la n)^{-2} $.  Crucially, note that the union bound is over 
$ \ell \asymp \la n/\log (\la n)$ vertices and not over $n$ vertices.

\medskip

 Note that the set $\RR_{s_n}=\RR_s$ must be an interval containing $\oo $ (possibly the entire cycle). Let $B$ be the event that $|\RR_s| \ge 2k_n $. Observe that on the event $B$ there must be some $i$ such that  $R_{u_{i}} \subseteq \RR_s$.   By the definition of the event $A_{u_{i}}$, if  $R_{u_{i}} \subseteq \RR_s$ and $A_{u_{i}}$ occurs, then also $R_{u_{i+1}} \subset\RR_s$ (where $i+1$ is defined modulus $\ell$). Since $\cup_{i=1}^{\ell} R_{u_{i}} =V(\mathrm{C}_n) $  we get that on the event $A \cap B $,  deterministically,  $\RR_s =  V(\mathrm{C}_n)$ (i.e.~every site is visited before the process corresponding to lifespan $s$ dies out).
\medskip

Since $\PP_{\la}[A^{c}] \le (\la n)^{-2} $, in order to conclude the proof it suffices to verify that $\PP_\la[ B^{c}] \le e^{-c  \log^{2/3} (\la n)} $. Let $D \subset B $ be  the event that \[ |\RR_{s}(\RR_s({\plant}) ) | \ge 2k_n \] (where $\RR_s({\plant})$ is the range of $\plant$ by time $s$, and for a collection of vertices $F$, the set $\RR_{s}(F) $ is the union of the ranges of the length $s$ walks picked by $\W_F$, the collection of particles which initially occupy $J$). Denote $Z:=| \RR_s({\plant})|$. Observe that for all $F \subseteq V $ we have that \[ \Ind{D^{c}} \cdot \Ind{\RR_s({\plant})=F } \le  \Ind{\max_{w \in \W_F} |\RR_s(w)| <2 k_n } \cdot \Ind{\RR_s({\plant})=F }.   \] Thus by conditioning on $\RR_s({\plant})$ and applying Poisson thinning, symmetry (namely, that $(|\RR_s(w)|)_{w \in \W }$ are i.i.d.~with the same law as $Z$) and Lemma \ref{lem:range} (third inequality) we get that 
\begin{equation*}
\begin{split}
\PP_\la[ B^{c}] \le \PP_\la[ D^{c}] & \le  \PP_\la[| \RR_s({\plant})| \\ & < \la^{-1} \log^{2/3} (\la n) ]+\exp[-\la ( \la^{-1} \log^{2/3} (\la n)) \PP_{\la}(Z<2k_{n})]  
 \\ & \le 2\exp(-c' \log^{2/3} (\la n) ). \quad \text{\qed}  
\end{split}
\end{equation*}

\section{Auxiliary results}
\label{s:aux}
\subsection{Percolation}
\label{s:auxpercolation}

\begin{defn}
\label{def:Percolation}
Let $G=(V,E)$ be some graph. Let $\alpha \in [0,1]$. Let  $(X_{v})_{v \in V}$ be i.i.d.~Bernoulli($\alpha$) random variables. The random graph $(V,\{\{u,v\} \in E:X_u=1=X_v \})$ is called     \emph{Bernoulli site percolation} on $G$ with parameter $\alpha$. 
\end{defn}
The following proposition is standard (e.g.~\cite{penrose}). Below, for each $d$ the constants can be chosen so that $ C (d,p),R(p,d) \searrow 0 $, $c(p) \nearrow 1 $ and $\beta(d,p) \to \infty$ as $p \nearrow 1 $.
\begin{prop}
\label{p: positivedensity}
Let $d \ge 2$. Then there exist some $p_c(d) \in (0,1)$ and some positive constants $C(d,p),R(d,p),\beta(d,p)$ and $c(p)$ (for $p \in (p_{\mathrm{c}}(d),1] $)   such that for all $p \in (p_{\mathrm{c}}(d),1]$, the largest connected component of the random graph obtained from Bernoulli site percolation with parameter $p$  on $\TT(n)$, denoted by $\mathrm{GC}$ (``giant component"), satisfies the following:
\begin{itemize}
\item[(1)]
It is the unique connected component of size at least $R(d,p) (\log n)^{\frac{d}{d-1}}$ w.p.~at least $1-n^{-1}$.
\item[(2)]
With probability at least $1-n^{-1}$, in every box of side-length $L=L(d,p):=\lceil C (d,p) (\log n)^{\frac{1}{d-1}} \rceil$ there are at least $c(p) L^d $ vertices belonging to $\mathrm{GC} $.
\item[(3)]
For every $n$ and $U \subseteq \TT(n)  $ the probability that $U \cap \mathrm{GC}   $ is empty is at most $\exp(-\beta(d,p)| U|^{\frac{d-1}{d}})$.  
\end{itemize}
\end{prop}
\subsection{Markov chains}
\label{s:Markov}
Generically, we shall denote the state space of a Markov chain  $(X_t)_{t=0 }^{\infty}$  by $\Omega
$ and its stationary distribution by $\pi$. We denote such a chain by $(\Omega,P,\pi)$. We say that the chain is finite, whenever $\Omega$ is finite.  We say that $P$ is \emph{reversible} if $\pi(x)P(x,y)=\pi(y)P(y,x)$ for all $x,y \in \Omega$. Throughout, we consider only finite reversible Markov chains, even if this is not written explicitly. We say that $P$ is \emph{lazy} if $P(x,x) \ge 1/2 $ for all $x\in \Omega$.
 We denote by $\Pr_{x}^t$ (resp.~$\Pr_{x}$) the distribution of $X_t$ (resp.~$(X_t)_{t \ge 0 }$), given that the initial state is $x$. Similarly, for a distribution $\mu$ on $\Omega$ we denote by $\Pr_{\mu}^t$ (resp.~$\Pr_{\mu}$) the distribution of $X_t$ (resp.~$(X_t)_{t = 0 }^{\infty}$), given that $X_0 \sim \mu $. 

 The $L_p$ norm and variance of a function $f \in \R^{\Omega}$ are $\|f\|_p:=(\mathbb{E}_{\pi}[|f|^{p}])^{1/p}$ for $1 \le p < \infty$ (where $\mathbb{E}_{\pi}[h]:= \sum_x \pi(x) h(x)$ for $h \in \R^{\Omega}$),  $\|f\|_{\infty}:=\max_x |f(x)|$ and  $\mathrm{Var}_{\pi}f:=\|f-\E_{\pi}f\|_2^2 $. The $L_p$ norm of a signed measure $\sigma$ (on $\Omega$) is
\begin{equation*}
\label{eq: Lpdef}
\|\sigma \|_{p,\pi}:=\|\sigma / \pi \|_p, \quad \text{where} \quad (\sigma / \pi)(x)=\sigma(x) / \pi(x).
\end{equation*}
We denote the worst case $L_p$ distance at time $t$   by $d_{p}(t):=\max_x d_{p,x}(t)$, where $d_{p,x}(t):= \|\Pr_x^t-\pi \|_{p,\pi}$. Under reversibility for all $x \in \Omega$ and $k \in \N $ (e.g.~(2.2) in \cite{spectral}) we have that
\begin{equation}
\label{eq: generalLp}
\begin{split}
d_{2,x}^2(k)&= h_{2k}(x,x)-1, \; \text{ where }h_{s}(x,y):=P^s(x,y)/\pi(y), \; \text{and} \\ d_{\infty}(2k)&:=\max_{x,y} |h_{2k}(x,y)-1|=\max_{y} h_{2k}(y,y)-1,
\end{split}
\end{equation}
When $P$ is also lazy, a standard argument (cf.~\cite[p.~135]{levin2009markov}) shows that \eqref{eq: generalLp} holds also for odd times. That is, for all $x \in \Omega$ and $k \in \N $   
\begin{equation}
\label{eq: generalLplazy}
 d_{\infty}(k):=\max_{x,y} |h_{k}(x,y)-1|=\max_{y} h_{k}(y,y)-1.
\end{equation}
The $\epsilon$-$L_{p}$-\emph{mixing-time} of the chain (resp.~for a fixed starting state $x$) is defined as
\begin{equation}
\label{eq: taupeps}
\tau_{p}(\epsilon):= \max_{x} \tau_{p,x}(\epsilon), \quad \text{ where } \tau_{p,x}(\epsilon):= \min \{t:  d_{p,x}(t)\le \epsilon \}.
\end{equation}
When $\epsilon=1/2$ we omit it from the above
notation.

 We identify $P^{k}$ with the operator on $\ell_2(\Omega,\pi):=\{f \in \R^{\Omega} :\|f\|_2< \infty \}$ given by $P^{k}f(x):=\sum_yP^{k}(x,y)f(y)=\E_x[f(X_k)]$. If $P$ is reversible then it is self-adjoint and hence has $|\Omega|$ real eigenvalues. Throughout we shall denote them by $1=\gamma_1>\gamma_2 \ge \ldots \ge \gamma_{|\Omega|} \ge -1$ (where $\gamma_2<1$ since the chain is irreducible). The \emph{spectral gap} and the \emph{absolute spectral gap} of $P$ are given by $\gamma:=1- \gamma_2 $ and $\tilde \gamma:=1-\max \{ \gamma_2,|\gamma_{|\Omega|}| \} $, respectively. The following fact (often referred to as the Poincar\'e inequality) is standard. It can be proved by elementary linear-algebra using the spectral decomposition (e.g.~\cite[Lemma 3.26]{aldous}).
\begin{fact}
\label{f: contraction}
Let $(\Omega,P,\pi)$ be a finite   irreducible reversible  Markov chain.
Let
$\mu$ be some distribution on $\Omega$. Let $f \in \R^{\Omega}$. Then for all $t \in \N $ we have that
\begin{equation}
\label{e:poincare}
\|\Pr_\mu^{t}-\pi \|_{2,\pi} \le (1-\tilde \gamma)^{t} \|\mu-\pi \|_{2,\pi} \quad \text{and} \quad \mathrm{Var}_{\pi}P^tf \le (1-\tilde \gamma)^{2t}  \mathrm{Var}_{\pi}f .
\end{equation}
\end{fact}
If $P $ is reversible and lazy we have that $\gamma_{|\Omega|} \ge 0 $ and so $\gamma=\tilde \gamma $. If in addition $\pi$ is the uniform distribution, it  follows from \eqref{eq: generalLplazy} in conjunction with \eqref{e:poincare} that for all $k \in \N $   
\begin{equation}
\label{e:lazdecayrev}
 \max_{x,y} \left|P^k(x,y)-\frac{1}{|\Omega|} \right| = \max_{x} P^k(x,x)-\frac{1}{|\Omega|}\le (1-\gamma)^{k} .
\end{equation}

We now state a particular case of Starr's maximal inequality \cite[Theorem 1]{starr} (cf.~\cite[Theorem 2.3]{cutoff}).
\begin{atheorem}[Maximal inequality]
\label{thm:maxin}
Let $(\Omega,P,\pi)$ be a reversible irreducible Markov chain. Let $1<p<\infty$ and $p^*:=p/(p-1)$ be its conjugate exponent. Then for all $f \in L^{p}(\R^{\Omega},\pi)
$,
\begin{equation}
\label{eq: ergodic1}
\|f^{*} \|_{p}^p \le 2( p^{*})^p \|f \|_p^p,
\end{equation}
where $f^* \in \R^\Omega$ is the corresponding \emph{maximal function}, defined as $$f^{*}(x):=\sup_{0 \le k < \infty}|P^{k}f(x)|=\sup_{0 \le k < \infty}|\mathbb{E}_{x}[f(X_{k})]|.$$
\end{atheorem}
\subsection{Lazy simple random walk on expanders}
\label{s:LSRW}
A lazy SRW (\emph{LSRW}) on $G$ evolves according to the following rule. At each step it stays put in its current position w.p.~$1/2$. Otherwise it moves to a random neighbor as SRW. We denote its transition matrix by $P_L:=\frac{1}{2}(I+P)$. Note that the spectral gap of $P_L$ is precisely half the spectral gap of $P$. It follows from \eqref{e:lazdecayrev} that
LSRW on a regular $\gamma$-expander $G=(V,E)$ mixes rapidly in the following sense 
 \begin{equation}
 \label{eq: mix}
 |\max_{x,y \in V}P_{L}^{t}(x,y) -|V|^{-1}|\le  (1-\gamma/2)^{t}, \quad \text{ for all } t.
\end{equation}
Let  $( X_{t}^L)_{t=0}^{\infty}$ be a LSRW on a regular $\gamma$-expander $G=(V,E)$. Denote the hitting time of a state $y$ w.r.t.~the LSRW by $T_y^L:=\inf \{t>0:X_t^L=y \}$. It follows from \eqref{eq: mix} (by averaging over $T_y^L$) that for all $x,y \in V$ and $t \ge 0$ 
\begin{equation}
\label{e:Green}
\sum_{i=0}^{t} P_{L}^{i}(x,y) \le \PP_x[T_y^{L} \le t ] \sum_{i=0}^{t} P_{L}^{i}(y,y) \le \PP_x[T_y^L \le t ]\left(\frac{2 }{\gamma}+\frac{t}{|V|}\right).
\end{equation}
Similarly, if the eigenvalues of $P$ are $-1<\gamma_{|V|} \le \cdots \le \gamma_2=1-\gamma < \gamma_1=1 $ then by the spectral decomposition (e.g.\ \cite[Lemma 12.2]{levin2009markov}), for every $x \in V$ there exist some $a_1,\ldots,a_{|V|} \in [0,1) $ such that $\sum_{i}a_i=1 $, $a_1=1/|V|$ and for all $i \in \N$ we have that $P^{i}(x,x)=1/|V| + \sum_{j=2}^{|V|}a_j \gamma_j^i$. Thus 
\begin{equation*}
\begin{split}
P^{2i-1}(x,x)+P^{2i}(x,x)-2/|V| &  = \sum_{j>1}a_j \gamma_j^{2i-1}(1+\gamma_j ) \\ & \le  \sum_{j>1 : \gamma_j>0 }^{}a_j (\gamma_j^{2i-1}+\gamma_j^{2i}) \le \gamma_{2}^{2i-1}+\gamma_{2}^{2i}.  \end{split}
\end{equation*}
\begin{equation}
\label{e:Green'}
\begin{split}
\sum_{i=0}^{2t} P^{i}(x,x) & =1+\sum_{i=0}^{t} P^{2i-1}(x,x)+P^{2i}(x,x) \\ & \le 1 + \sum_{i=1}^t \frac{2}{|V|}+\gamma_{2}^{2i-1}+\gamma_{2}^{2i} \le \frac{2t}{|V|}+\frac{1}{\gamma}.
\end{split}
\end{equation}
\begin{equation}
\label{e:Green''}
\sum_{i=0}^{2t} P^{i}(x,y)\le \PP_x[T_y \le 2t ] \sum_{i=0}^{2t} P^{i}(y,y)\le \PP_x[T_y \le 2t ]\left(\frac{1 }{\gamma}+\frac{2t}{|V|}\right).
\end{equation}
\begin{lem}
\label{lem: hitinexpander}
Let $G=(V,E)$ be a connected regular $n$-vertex $\gamma$-expander.  Then \[\forall x,y \in V, \quad \PP_x[T_y^{L} > t ] \le\begin{cases} 1- \frac{ \gamma  t}{4n} & \text{if } \lceil 8 \gamma^{-1}\log
n \rceil \le t \le \gamma^{-1} n, \\
(3/4)^{\lfloor \gamma t/n \rfloor } & \text{if } t > \gamma^{-1} n \\ 
\end{cases}.  \] 
\end{lem}
\emph{Proof:}
Let   $x , y \in V$. Let $ \lceil 8 \gamma^{-1}\log
n \rceil \le t \le \gamma^{-1} n $. By \eqref{eq: mix} $P_{L}^{i}(x,y)  \ge \frac{1}{2n} $ for all $i \ge 4 \gamma^{-1} \log n$. Consequently, $\sum_{i=1}^{t} P_{L}^{i}(x,y)\ge \frac{t}{2n} $. Hence by \eqref{e:Green} $\PP_x[T_y^{L}  \le t ] \ge \frac{ \gamma  t}{4n}$. The case $t > \gamma^{-1} n $ follows from the previous case by the Markov property. \qed

\medskip

The following corollary is an immediate consequence of Lemma \ref{lem: hitinexpander}, obtained by a union bound over $V$, using Poisson thinning and independence.
\begin{cor}
\label{cor: prey}
Let $G=(V,E)$ be a connected regular $n$-vertex $\gamma$-expander. Let $\ell:= \lceil n/4 \rceil $. Let $(v_i)_{i=0}^{\ell} $ be an arbitrary collection of distinct vertices. Assume that at each of these vertices there are initially $\Pois(\la)$ particles, independently, and that the particles perform simultaneously independent LSRW on $G$.  Assume that $ 48 \frac{\log n}{n} \le \la \le 1  $. Let   $t:= \lceil \frac{ 2^{7}}{ \lambda\gamma} \log n \rceil $.  Then \[ \Pr [(\text{The union of the first $t$ steps performed by the particles}) \neq V ] \le n^{-2}.\]
Similarly, if $ 4n^{-1} \le \la \le 48 n^{-1} \log n $ and  $\ell:= \lceil n/32 \rceil $ then there exists some constant $M$ such that
\begin{equation}
\begin{split}
\label{e:lastihope}
& \Pr [(\text{The union of the first $\left\lceil \frac{M}{\lambda\gamma}  \log n \right\rceil$ steps performed by the particles}) \neq V ] \\ & \le \Pr[\mathrm{Pois}(\la n/32) \ge  \la n/64] + n^{-2}.
\end{split}
\end{equation}  
\end{cor}

The following lemma is inspired by the techniques from \cite{cutoff}. \begin{lem}
\label{lem:maxineq}
Let $G=(V,E)$ be a connected $d$-regular $n$-vertex $\gamma$-expander. Let $A \subset V $ and $R>0$. Let $s=s_R:=\lceil \gamma^{-1}\log (2^{7}R)\rceil $. Consider the set \[H_A=H_{A,R}:=\{v \in V:\sup_{t:\, t \ge s}|P_L^t(v,A) - \pi(A) |  \ge 1/4\}, \]
where $\pi$ is the uniform distribution on $V$. Then
\begin{equation}
\label{e:HAR}
\pi( H_{A,R})\le \frac{1}{R} \pi(A)\pi(A^{c}).
\end{equation}
 
\end{lem}
\begin{proof}
Consider $f:V \to \R$ defined by
\[f(x):=P_{L}^s (1_A-\pi(A))(x)=P_L^s(x,A)-\pi(A).\] By the Poincar\'e  inequality \eqref{e:poincare} and the choice of $s$, \[\|f\|_2^2=\mathrm{Var}_{\pi}P_L^s 1_A \le (1-\gamma/2)^{2s}\mathrm{Var}_{\pi}1_{A} \le e^{-\gamma s} \mathrm{Var}_{\pi}1_{A} \le (2^{7}R)^{-1} \pi(A) \pi(A^c)  .\]

Consider $f_*(x):=\sup_{t \ge 0 }|P_L^tf(x)|=\sup_{t:t \ge s }|P_L^t(x,A)-\pi(A)| $.
By Starr's Maximal inequality (Theorem \ref{thm:maxin}), $$\|f_*\|_2^2 \le 8 \|f\|_{2}^2 \le (16R)^{-1} \pi(A) \pi(A^c).   $$ 
Finally, note that \[H_{A}:=\{x:f_*^2(x) \ge 1/16 \}\subseteq \{x:f_*^2(x)\ge R\|f_*\|_2^2/  \mathrm{Var}_{\pi}1_{A}\} ,\] and so by Markov's inequality $\pi(H_A) \le  R^{-1} \mathrm{Var}_{\pi}1_{A}=R^{-1} \pi(A)\pi(A^{c})$. 
\end{proof}
\begin{cor}
\label{cor:largeE}
Let $G=(V,E)$ be a connected regular $n$-vertex $\gamma$-expander. Let $\la \in (0,1]$.  Let $A \subset V $.  Let $r:=\lceil 2^{17}\la^{-1} \gamma^{-1}\log 2^{9} \rceil $. Assume that $|A| \le \frac{n}{4}$. Let $(X_t)_{t=0}^{\infty}$ be SRW on $G$. Let $\kappa:=\min( \lceil 16 \la^{-1} \log 2^9 \rceil , \frac{n}{2^{11}}) $. Consider \[G_{A}:=\{a \in A: \Pr_{a}[| \{ X_{t}:t \in [r]\}  \setminus A  | \ge \kappa] \ge 1/16 \}. \]
Then,  \[ |G_{A}| > \frac{3}{4}|A|.  \]
\end{cor}
\noindent \emph{Proof:} 
Let $H_{A,4}$ be as in Lemma \ref{lem:maxineq}. Consider $B:=A \setminus H_{A,4}$. By Lemma \ref{lem:maxineq}, in order to conclude the proof it suffices to show that $B \subseteq G_A$. It is easy to see that we may replace  $(X_t)_{t=0}^{\infty}$ in the definition of $G_A$ by $(X_t^L)_{t=0}^{\infty}$, a LSRW on $G$, as this cannot increase $G_A$. (Observe that SRW can be coupled with LSRW starting from the same initial position so that they follow the same trajectory, with the LSRW spending at each site a random number of steps, with a Geometric(1/2) distribution, before moving to the next site.) Hence it suffices to show that for all $b \in B $,  \[\Pr_{b}[| \{ X_{t}^L:t \in [r]\} \setminus A | \ge  \kappa ] \ge 1/16  .\]   Let $b \in B$. By the definition of  $H_{A,4}$ we have for all $t \ge k:=\lceil \gamma^{-1}\log (2^{9})\rceil$ that $P_L^t(b,A^c) \ge 1-(\frac{|A|}{n}+\frac{1}{4}) \ge 1/2$.  Hence
\[\E_{b}[|\{k \le t \le  r : X_{t}^L \in A^{c}  \}  |] \ge \frac{1}{2} (r-k) .    \]
Since for any sum of indicators $D:=\sum_{i=1}^s 1_{D_i} $ we have that \[\E \left[D^{2}\right] \le s \E[D] \le \frac{ (\E[ D])^2 }{\min_{i \in [s]} \PP[D_i]},\]
by the Paley-Zygmund's inequality we get that
\[\Pr_b[|\{k \le t \le  r : X_{t}^L \in A^{c}  \}  | \ge \frac{1}{4} (r-k) ] \ge \frac{1}{2^2}\cdot \frac{1}{2}=1/8. \]
Call time $j \in [k, r] $ \emph{good} if the LSRW visits $X_j^L$ at most $2^{7}(2 \gamma^{-1}+\frac{r-k}{n}) $ times during the time interval $[k,r]$. Otherwise, call time $j$ \emph{bad}.  
By \eqref{e:Green}  $\max_{v \in V}\sum_{i=0}^{r-k}P_L^{i}(v,v) \le (2 \gamma^{-1}+\frac{r-k}{n}) $. Hence by Markov's inequality each time $j$ is bad w.p.~at most $2^{-7}$. Let $J$ be the event that there are at least $\frac{1}{8} (r-k)$ bad times between time $k$ and $r$.  Again by Markov's inequality,  $\Pr_b[J] \ge 8 \cdot 2^{-7} = 1/16$. Hence
\[\Pr_b\left[|\{t \text{ is good } :k \le t \le  r ,\, X_{t}^L \in A^{c}   \}  | \ge \frac{1}{8} (r-k) \right] \] \[ \ge \Pr_b\left[|\{k \le t \le  r : X_{t}^L \in A^{c}  \}  | \ge \frac{1}{4} (r-k) \right]- \Pr_b[J] \ge 1/8-1/16 = 1/16   \] 
On this event we have that $| \{ X_{t}^L:t \in [r]\} \setminus A | \ge \frac{\frac{1}{8} (r-k)}{2^{7}(2 \gamma^{-1}+\frac{r-k}{n})} \ge \kappa,$ as
\[| \{ X_{t}^L:t \in [r]\} \setminus A | \ge \sum_{i=k }^r \Ind{X_i \notin A, \text{ time $i$ is good, }X_i^L \neq X_j^L \text{ for all }j \in [k,i-1]  } \]
\[\ge \left(2^{7}\left(2 \gamma^{-1}+\frac{r-k}{n}\right)\right)^{-1} \sum_{i=k}^r \Ind{X_i \notin A, \text{ time $i$ is good }  }. \qquad  \text{\qed}  \]

\medskip

The third author learned the argument involving ``good times" from Yuval Peres (private communication). We thank him for allowing us to present this argument.

\begin{lem}
\label{lem:rangeexp}
Let $(X_i)_{i=0}^{\infty}$ be a SRW on some regular $n$-vertex $\gamma$-expander  $G=(V,E)$. Let $R(t)=\{X_i:i \in [t] \}$ be the range of its first $t$ steps. If  $t \le \gamma^{-1}n$ then for all $x \in V$
\[ \Pr_x \left[|R(t)| \ge \frac{t}{32 \gamma} \right] \ge \frac{3}{4}. \]
\end{lem}
\begin{proof}
Call time $j \le t $ \emph{bad} if the walk visits $X_j$ at least $8( \gamma^{-1}+\frac{t}{n}) $ times between time $j$ and time $t$. By \eqref{e:Green'}  the expected number of visits to $X_j$ between time $j$ and time $t$ is at most $\gamma^{-1}+t/n$. By Markov's inequality each time is bad w.p.~at most $1/8$. Again, by Markov's inequality w.p.~at least $3/4$ there are at most $\frac{t+1}{2}$ bad times. On this event, we have that $|R(t)| \ge \frac{(t+1)/2}{8( \gamma^{-1}+\frac{t}{n})} \ge \frac{t}{32 \gamma}$.
\end{proof} 
\subsection{Range, hitting time and Green function estimates}
\label{s:green}

Let us first recall the local CLT and a standard large deviation estimate. We note that the $\log t$ term from the definition of $A(t)$ below can be replaced by any diverging function of $t$.
\begin{fact}
\label{f:LCLT}
For every  $t \in \N $ let $A(t):=\{s \in \Z : t-s \text{ is even and }|s| \le \frac{ t^{3/4}}{\log t} \}$. Then
\begin{equation}
\label{e:LCLT1d}
\lim_{t \to \infty} \sup_{s \in A(t) } \frac{P_{\Z}^{t}(0,s)}{2\sqrt{\frac{1}{2 \pi t }}\exp \left(- \frac{s^{2}}{2t} \right)  }=1=\lim_{t \to \infty} \inf_{s \in A(t) } \frac{P_{\Z}^{t}(0,s)}{2\sqrt{\frac{1}{2 \pi t }}\exp \left(- \frac{s^{2}}{2t} \right)  }. \end{equation}
Moreover, if $(S_i)_{i=0}^{\infty}$ is a SRW on $\Z$ then for every $n,m \ge 1 $ we have 
\begin{equation}
\label{eq: LDSRW}
\begin{split}
 \Pr_0[S_n \ge m \sqrt{n} ] \le e^{-m^2/2}.
\end{split}
\end{equation}
\end{fact}
\emph{Proof.}
For \eqref{e:LCLT1d} apply Stirling's approximation. For \eqref{eq: LDSRW} we use the fact that $\mathbb{E}_0[e^{a (S_{i+1}-S_i)}]=\half (e^{a}+e^{-a}) \le e^{a^2/2}$ (as can be seen by comparing Taylor series coefficients) and independence to get that $\mathbb{E}_0[e^{a S_{n}}] \le e^{a^2n/2} $.  For $a=\frac{m}{\sqrt{n}} $ we have that $na^{2}/2-am \sqrt{n}=-m^2/2$ and so by the above   \[\Pr_0[S_n \ge m \sqrt{n} ]=\Pr_0[e^{a S_n} \ge e^{am \sqrt{n}} ] \le \mathbb{E}_0[e^{a S_{n}}]   e^{-am \sqrt{n}}\le e^{-m^2/2}. \text{ \qed}   \]

Let us now recall some heat kernel  estimates for SRW on $\Z^d$. Denote the origin by $\mathbf{0}:=(0,\ldots,0)$. For $a,b \in \Z^d$ and $i \in \N \cup \{ \infty \} $ let $\nu_i^{(d)}:=\sum_{j=0}^i P_{\Z^d}^j(\mathbf{0},\mathbf{0}) $ and  $\nu_i^{(d)}(a,b):=\sum_{j=0}^i P_{\Z^d}^j(a,b)$. Let $\|a \|_p := ( \sum_{i=1}^d |a_i|^p )^{1/p}$ (respectively, $\|a\|_{\infty}:=\max_{i \in [d]}|a_i| $) be the $\ell_p$ (respectively, $\ell_{\infty}$) norm of $a=(a_1,\ldots,a_d)\in \Z^d$.

Recall that the Gamma function is  $\Gamma(k):=\int_0^{\infty}r^{k-1}e^{-r}dr$. Recall that $\Gamma(k)=(k-1)!$ and  $\Gamma(k+ \frac{1}{2} )=2^{1-2k} \sqrt{\pi}\frac{(2k-1)!}{ (k-1)!} \asymp (\frac{k}{e})^k  $ for all $k \in \N$ and that $\Gamma(\frac{1}{2})=\sqrt{\pi} $.

\begin{fact}
\label{f:heatkernelZd}
For all  $d \ge 3$
\begin{equation}
\label{e:greenlclt}
\lim_{\|a \|_2 \to \infty} \nu_{\infty}^{(d)}(\mathbf{0},a)/\|a\|_2^{2-d} =  \left(\frac{1}{2\pi}\right)^{d/2}\Gamma\left(\frac{d}{2}-1 \right).
\end{equation}
\begin{equation}
\label{e:greenlclt2}
\lim_{\|a \|_2 \to \infty} \frac{ \nu_{\|a\|_2^2 }^{(d)}(\mathbf{0},a)}{\|a\|_2^{2-d}} \left(2\pi\right)^{d/2}=  \int_1^{\infty}r^{-\frac{d}{2}-2}e^{-r}dr\asymp    \Gamma\left(\frac{d}{2}-1 \right). 
\end{equation}
\end{fact}
\emph{Proof.}
The calculation is somewhat neater for the continuous-time SRW with jump rate 1, as for it the evolution of the walk in different coordinates is independent. The expectations (Green's functions) of both walks are the same via a standard coupling in which both walks follow the same trajectory, where the continuous-time walk waits a random number of time units between each jump, which is $\mathrm{Exp}(1)$ distributed (the Exponential distribution of parameter 1). Let $\mathrm{H}_t(\mathbf{0},a)$ be the transition probability from $\mathbf{0}$ to $a$ for time $t$ for the continuous-time walk. Then the local CLT takes the form  $\lim_{\|a\|_2 \to \infty} \mathrm{H}_{t\|a\|_2^2}(\mathbf{0},a)/\|a\|_2^{-d}=(\frac{1}{2 \pi t })^{d/2}e^{-1/t}$. Thus
\begin{equation*}
\begin{split}
\nu_{\infty}^{(d)}(\mathbf{0},a) &=\|a\|_2^2 \int_0^{\infty}\mathrm{H}_{t\|a\|_2^2}(\mathbf{0},a)dt \\ &=(1 \pm o(1) ) \|a\|_2^{2-d} \int_0^{\infty}\left(\frac{1}{2 \pi t }\right)^{d/2}e^{-1/t}dt \\& =(1 \pm o(1) ) \|a\|_2^{2-d} \left(\frac{1}{2\pi}\right)^{d/2} \int_0^{\infty}r ^{(d-4)/2}e^{-r}dr \\& =(1 \pm o(1) ) \left(\frac{1}{2\pi}\right)^{d/2} \Gamma\left(\frac{d}{2}-1 \right)    \|a\|_2^{2-d}.  
\end{split}
 \end{equation*}
Similarly,
\[\nu_{\|a\|_2^2 }^{(d)}(\mathbf{0},a) =(1 \pm o(1) ) \|a\|_2^{2-d}   \left(\frac{1}{2\pi}\right)^{d/2} \int_{1}^{\infty}r ^{(d-4)/2}e^{-r}dr \] \[ \asymp  \left(\frac{1}{2\pi}\right)^{d/2}\Gamma\left(\frac{d}{2}-1 \right)  \|a\|_2^{2-d} . \quad \text{\qed}   \] 
\begin{fact}
\label{f:LCLT2}
For all  $d \ge 2$. Let $a_n,b_n \in \Z^d$ be such that $\lim_{n \to \infty} \|a_n\|_2 = \infty $ and $\lim_{n \to \infty} \frac{ \|a_n\|_2 }{ \|b_n\|_2 }= 1 $. If $t_n \gtrsim \|a_n\|_2^2 $  then  $\lim_{n \to \infty} \frac{\nu_{t_{n}}^{(d)}(\mathbf{0},a_{n}) }{ \nu_{t_{n}}^{(d)}(\mathbf{0},b_{n})}= 1 $. Consequently if $t_n,\ell_n,L_n$ diverge and satisfy $\ell_n \ll \sqrt{ t_n } \ll L_n $ and  $A_n \subseteq \Z^d $ are such that $|A_{n} \cap B(r) | \ge \delta|B(r)|$ for all $r \in [\ell_n ,L_{n}] $, where $B(r):=\{v \in \Z^d:\|v\|_2 \le r \} $, then $\liminf_{n \to \infty }\frac{1}{t_{n}} \sum_{i=0}^{t_n} P_{\Z^d}^i(\mathbf{0},A_{n}) \ge \delta $.
\end{fact}
\begin{lem}
\label{l:ZdTTd}
There exist some $C,M(d)>0$ such that for all  $d \ge 2$ and $n \ge M(d)$,  and all $a=(a_{1},\ldots,a_d) \in \Z^d$ such that $\|a\|_{\infty} \le n/2$, if we identify $\mathbf{0}$ and $a$ with a pair of vertices of $\TT(n) $ then for all $t$ 
\[ 0 \le P_{\TT (n)}^{t} (\mathbf{0},a) - P_{\Z^d}^{t} (\mathbf{0},a) \le C n^{-d}.  \]  
\end{lem}
\begin{proof}
Let $A(a):=\{(b_1,\ldots,b_d) \in \Z^d:b_i \equiv a_i \text{ mod }n \text{ for all }i \in [d]  \}$. Then $P_{\TT (n)}^{t} (\mathbf{0},a)=P_{\Z^d}^{t} (\mathbf{0},A(a))$. The claim now follows from Fact \ref{f:LCLT}. We leave the details as an exercise.
\end{proof}
\begin{lem}
\label{l:hitviaexpectation}
Let  $N_i(a)$ be the number of visits to $a$ by time $i$ for some Markov chain. Then for all $a,b$ and $t,s \in \N$ we have that
\[ \frac{\mathbb{E}_{a}[N_{t}(b)]}{\mathbb{E}_{b}[N_{t}(b)]} \le \Pr_a[T_b \le t ]  \le \frac{\mathbb{E}_{a}[N_{t+s}(b)]}{\mathbb{E}_{b}[N_{s}(b)]}.  \]
\end{lem}
\emph{Proof:}
Note that $T_b \le t $ iff $N_t(b) > 0 $. Hence $\Pr_a[T_b \le t ]  =\Pr_a[  N_{t}(b)>0]$ and $\frac{ \mathbb{E}_{a}[N_{t}(b)]}{\mathbb{E}_{a}[N_{t}(b) \mid N_{t}(b)>0 ]  }  =\Pr_a[  N_{t}(b)>0] \le \frac{ \mathbb{E}_{a}[N_{t+s}(b)]}{\mathbb{E}_{a}[N_{t+s}(b) \mid N_{t}(b)>0 ]  } $, from which we get 
\[ \frac{ \mathbb{E}_{a}[N_{t}(b)]}{\mathbb{E}_{b}[N_{t}(b) ]  } \le \Pr_a[  N_{t}(b)>0] \le \frac{ \mathbb{E}_{a}[N_{t+s}(b)]}{\mathbb{E}_{b}[N_{s}(b) ]  }. \quad \text{\qed}    \]
\begin{lem}
\label{l:nu2d}
 Let  $R(t):=\{X_i:i \in [t] \}$ be the range of the first $t$ steps of SRW.  Then 
\begin{equation}
\label{e:greenlclt2d}
\lim_{k \to \infty }kP_{\Z^2}^{2k}(\mathbf{0},\mathbf{0})=\frac{1}{\pi} \quad \text{and so} \quad   \lim_{t \to \infty}\frac{\nu_{t}^{(2)}}{\log t} =\frac{1}{\pi}.  \end{equation}
\begin{equation}
\label{e:greenlclt2d2}
  \lim_{t \to \infty}\E^{\Z^2}\left[ \frac{1}{t}|R(t)| \log t \right]= \pi.  \end{equation}
Moreover, if $1 \ll t_n=o(n^2 \log n )$ then  $\E^{\mathbb{T}_2(n)}[|R(t_{n})|]=\frac{(1 \pm o(1) ) \pi t_n}{  \log t_n} $ and $ \sum_{i=0}^{t_n}P_{\mathbb{T}_2(n)}^i(x,x)=(1 \pm o(1) )\frac{ \log t_{n}}{\pi}$ for all $x \in \mathbb{T}_2(n)$.

\end{lem}
\begin{proof} By the local CLT and \eqref{eq: LDSRW} 
\begin{equation*}
\begin{split}
\lim_{k \to \infty }kP_{\Z^2}^{2k}(\mathbf{0},\mathbf{0})&=\lim_{k \to \infty }k \sum_{j:\, |2j-k| \le k^{2/3} } P_{\Z}^{2j}(0,0) P_{\Z}^{2k-2j}(0,0) \\ &=\lim_{k \to \infty }\frac{k}{2}(P_{\Z}^{2\lfloor \frac{k}{2} \rfloor}(0,0))^{2}=\frac{1}{\pi}.
\end{split}
\end{equation*}
 Thus $\nu_{t}^{(2)}= \sum_{k=0}^{\lfloor t/2 \rfloor}P_{\Z^2}^{2k}(\mathbf{0},\mathbf{0}) =(1 \pm o(1) ) \sum_{k=0}^{\lfloor t/2 \rfloor}\frac{1}{\pi k}=(1 \pm o(1) )\frac{\log t}{\pi} $. Hence also $\nu_{s }^{(2)}=(1 \pm o(1) )\frac{\log t}{\pi} $, where 
 $s=s_{t}:=\lceil t/\log^2 t \rceil$.  Observe that for every $x,y$ we have that \[\Pr_x^{\Z^2}[T_y \le t-s ]\nu_{s}^{(2)}  \le \sum_{i=0}^tP_{\Z^2}^i(x,y) \le \Pr_x^{\Z^2}[T_y \le t ]\nu_{t}^{(2)}.  \] Summing over all $y$ we get that  $\E[ |R(t)|  ] \le \E[ |R(t-s)|  ]+s  \le (t+1)/\nu_{s}^{(2)}+ s=(1 \pm o(1) )\frac{\pi t}{\log t}$ and that  $\E[ |R(t)|  ] \ge (t+1)/\nu_{t}^{(2)}=(1 \pm o(1) )\frac{\pi t}{\log t}$.  

We now prove that  for SRW on $\mathbb{T}_2(n) $, if $t_n=o(n^2 \log n )$ then  $\E^{\mathbb{T}_2(n)}[|R(t_{n})|]=(1 \pm o(1) )\frac{ \pi t_n }{ \log t_n} $. By a straightforward coupling argument (in which we let both walks evolve according to the same sequence of increments) for all $t >0 $ we have that    \[\E^{\mathbb{T}_2(n)}[|R(t)| ] \le \E^{\Z^2}[|R(t)|] =(1 \pm o(1) )\frac{\pi t}{\log t}.    \]  Conversely, for $t_n = o(n^2 \log n )$ by Lemma \ref{l:ZdTTd} (as $t_n n^{-2} = o(\log t_n) $) we have that \[\nu_{t_n}:= \sum_{i=0}^{t_n}P_{\mathbb{T}_2(n)}^i(x,x)=(1 \pm o(1) )\frac{ \log t_{n}}{\pi},\]   and so, as above,   $\E^{\mathbb{T}_2(n)}[ |R(t_{n})|  ] \ge (t+1)/\nu_{t_n}=(1 \pm o(1) )\frac{\pi t_{n}}{\log t_{n}}$.  
\end{proof}
\begin{lem}
\label{l:2dhitprob}
Consider SRW on $\Z^2$. Let $a_{n} \in \Z^2 $. If $\|a_{n}\|_2 \gg 1 $ then  \begin{equation}
\label{e:2dhitprob}
\forall \; C_n \ge 1, \quad \Pr_{\mathbf{0}}[T_{a_{n}} \le C_{n}\|a_{n}\|_2^{2}  ]  \asymp  (1+ \log C_{n} )/ \log (C_n \|a_n\|_2^2). \end{equation}
\begin{equation}
\label{e:2dhitprob0}
\forall \; \|a_{n}\|_2^{-2/3}  \ll c_n \le 1, \quad \Pr_{\mathbf{0}}[T_{a_{n}} \le c_{n}\|a_{n}\|_2^{2}  ]  \asymp  \frac{c_{n} \exp(-1/c_n ) }{ \log (c_n \|a_n\|_2^2)}. \end{equation}
\begin{equation}
\label{e:2dhitprob2}
\text{For all fixed }C,\alpha>0, \quad \Pr_{\mathbf{0}}[T_{a_{n}} \le C\|a_{n}\|_2^{2+ 2 \alpha }  ]  =(1 \pm o(1)) \alpha /(1+\alpha ). \end{equation}
\end{lem}
\begin{proof}
For \eqref{e:2dhitprob} use Lemma \ref{l:hitviaexpectation} with $t=\lceil C_{n}\|a_{n}\|_2^{2} \rceil $ and $s=\lceil t/\log t \rceil$. Indeed for this choice of parameters, in the notation of Lemma \ref{l:hitviaexpectation}, by  \eqref{e:greenlclt2d} we have that $\mathbb{E}_{a}[N_{s}(a) ]=\nu_{s}^{(2)}=(1+o(1))\log t/\pi = \mathbb{E}_{a}[N_{t}(a)] \asymp   \log (C_n \|a_n\|_2^2) $ and by the local CLT  \[\mathbb{E}_{\mathbf{0}}[N_{t}(a) ] \asymp \int_{1/2}^{C_n}\frac{dr}{r } \asymp (1+ \log C_{n} ) \asymp \mathbb{E}_{\mathbf{0}}[N_{t+s}(a) ] .\]
For \eqref{e:2dhitprob0} use Lemma \ref{l:hitviaexpectation} with $t:=\lceil c_{n}\|a_{n}\|_2^{2} \rceil $ and $s=\lceil t/\log t \rceil $. Indeed, as before by \eqref{e:greenlclt2d}  $\mathbb{E}_{a}[N_{s}(a) ]= \log (c_n \|a_n\|_2^2)= \mathbb{E}_{a}[N_{t}(a) ] $ and by the local CLT (which is applicable by the assumption that $c_n \gg \|a_{n}\|_2^{-2/3}$ and so $\|a_n\|_2 \ll (c_n \|a_n\|_2^2)^{3/4} $) we have that  
  \[\mathbb{E}_{\mathbf{0}}[N_{t}(a) ] \asymp \int_{0}^{c_n}\frac{dr}{r }e^{-1/r}=\int_{1/c_n}^{\infty} \frac{1}{r}  e^{-r}dr \asymp c_n e^{-1/c_n} \asymp \mathbb{E}_{\mathbf{0}}[N_{t+s}(a) ] .\]

For \eqref{e:2dhitprob2} use Lemma \ref{l:hitviaexpectation} with $t=\lceil C\|a_{n}\|_2^{2+2\alpha} \rceil $ and $s=\lceil t/\log t \rceil $.  Indeed for this choice of parameters, as before by  \eqref{e:greenlclt2d} $\mathbb{E}_{a}[N_{s}(a) ]=\nu_{s}^{(2)}=(1 \pm o(1))\log t/\pi=(1+\alpha \pm o(1))\log \|a_{n}\|_2^{2}/\pi=\mathbb{E}_{a}[N_{t}(a) ]  $ and  \[\mathbb{E}_{\mathbf{0}}[N_{t}(a) ]=(1\pm o(1) )\frac{1}{\pi} \sum_{i=\lceil \|a_{n}\|_2^{2} \rceil}^{t/2}1/i=(1\pm o(1) )\frac{\alpha \log \|a_{n}\|_2^2 }{\pi}=\mathbb{E}_{\mathbf{0}}[N_{t+s}(a) ] .\]
Substituting these estimates in Lemma \ref{l:hitviaexpectation} yields \eqref{e:2dhitprob2}. \end{proof}
 \begin{lem}
\label{l:ZdTTd2}
Let $d \ge 2$. Let $a_{n}\in \Z^d$ be such that $ \|a_{n}\|_{\infty} \le n/2$. Identify $\mathbf{0}$ and $a_n$ with a pair of vertices of $\mathbb{T}_d(n) $. Let $t_n$ be such that
$\|a_n\|_1 \le t_n \le dn^2\log d $ (the sole purpose of the assumption  $ t_n \ge \|a_n\|_1  $ is to ensure that $\Pr_{\mathbf{0}}^{\Z^d}[T_{a_{n}} \le t_{n}  ]>0$).
Then
\begin{equation}
\label{e:hittoZd1}
0 \le \Pr_{\mathbf{0}}^{\mathbb{T}_d(n)}[T_{a_n} \le t_n  ] - \Pr_{\mathbf{0}}^{\Z^d}[T_{a_n} \le t_n  ] = O(\Pr_{\mathbf{0}}^{\mathbb{T}_d(n)}[T_{a_n} \le t_n  ] ).  \end{equation}
In fact,  if  $ \|a_{n}\|_{2} \ll n$  then
\begin{equation}
\label{e:hittoZd2}
0 \le \Pr_{\mathbf{0}}^{\mathbb{T}_d(n)}[T_{a_n} \le t_n  ] - \Pr_{\mathbf{0}}^{\Z^d}[T_{a_n} \le t_n  ] = o(\Pr_{\mathbf{0}}^{\mathbb{T}_d(n)}[T_{a_n} \le t_n  ] ),  \end{equation}
while if $d=2$ and $s_n \ll n^2 \log n $ then

\begin{equation}
\label{e:hittoZd3}
0 \le \Pr_{\mathbf{0}}^{\mathbb{T}_2(n)}[T_{a_n} \le s_n  ] - \Pr_{\mathbf{0}}^{\Z^2}[T_{a_n} \le s_n  ] = o(1 ),  \end{equation}
 
\end{lem}
\begin{proof}
Let $a_{n}=(a_{n,1},\ldots,a_{n,d})\in \Z^d$ satisfy  $1 \ll \|a_{n}\|_{2} \ll n$.  Let $A(a_{n}):=\{(b_1,\ldots,b_{2}) \in \Z^d:b_i \equiv a_{n,i} \text{ mod }n \text{ for all }i \in [d]  \}$. Then (by the aforementioned coupling of SRW on $\TT(n)$ with SRW on $\Z^d$ in which both walks follow the same sequence of increments) $\Pr_{\mathbf{0}}^{\mathbb{T}_d(n)}[T_{a_{n}} \le t]= \Pr_{\mathbf{0}}^{\Z^d}[T_{A(a_{n})} \le t  ]$. Thus for all $t$ \[0 \le \Pr_{\mathbf{0}}^{\mathbb{T}_d(n)}[T_{a_{n}} \le t  ]-\Pr_{\mathbf{0}}^{\Z^d}[T_{a_{n}} \le t  ]  \le \sum_{b \in A(a_{n}) \setminus \{a_{n}\} } \Pr_{\mathbf{0}}^{\Z^d}[T_{b} \le t  ]=:L(t,n) .\] Using $\|a_n\|_1 \le t_n \le dn^2\log d $ we argue that 
\begin{itemize}
\item[$d \ge 3$:]
 $\Pr_{\mathbf{0}}^{\Z^d}[T_{a_{n}} \le t_{n}  ] \gtrsim  L(t_{n},n)$. To see this use the fact that \[\Pr_{\mathbf{0}}^{\Z^d}[T_{a_{n}} \le t_{n}  ] \asymp \nu_{t_n}^{(d)}(\mathbf{0},a_n) \text{ and }L(t_{n},n)\asymp \sum_{b \in A(a_{n}) \setminus \{a\} } \nu_{t_n}^{(d)}(\mathbf{0},b)=:J \] together with a similar calculation to \eqref{e:greenlclt2} in order to deduce that $\nu_{t_n}^{(d)}(\mathbf{0},a_n) \gtrsim J$.
Moreover, if $\|a_n\|_2 \ll n $ then  $\nu_{t_n}^{(d)}(\mathbf{0},a_n) \gg J$ and so $\Pr_{\mathbf{0}}^{\Z^d}[T_{a_{n}} \le t_{n}  ] \gg L(t_{n},n) $.
\item[$d = 2$:] Using  \eqref{e:2dhitprob}-\eqref{e:2dhitprob2} if   $\|a_n\|_1 \le t_n \le 2n^2\log 2 $  then $\Pr_{\mathbf{0}}^{\Z^2}[T_{a_{n}} \le t_{n}  ] \gtrsim L(t_{n},n) $. Moreover, if $\|a_n\|_{2} \ll n$ then by \eqref{e:2dhitprob}-\eqref{e:2dhitprob2}
$\Pr_{\mathbf{0}}^{\Z^2}[T_{a_{n}} \le t_{n}  ] \gg L(t_{n},n) $. This concludes the proofs of \eqref{e:hittoZd1}-\eqref{e:hittoZd2} for $d=2$.

 We now prove \eqref{e:hittoZd3}. By \eqref{e:hittoZd2} we may assume that $s_n \ge n^2$. Let $M(k):=|\{i \in [n^2,k]:X_i=a_n \}|$. We will show that $\Pr_{\mathbf{0}}^{\mathbb{T}_2(n)}[M(s_n)>0]=o(1)$. This concludes the proof, as \[\Pr_{\mathbf{0}}^{\mathbb{T}_d(n)}[T_{a_n} \le n^{2}  ] - \Pr_{\mathbf{0}}^{\Z^d}[T_{a_n} \le n^{2}  ] \le L(n^{2},n) =o(1),\] where the last equality holds by  \eqref{e:2dhitprob0}.    

 At time $n^2$ the $L_{\infty}$ distance of SRW on $\mathbb{T}_2(n)$ from the uniform distribution is bounded from above (this follows from Lemmas \ref{l:ZdTTd} and \ref{l:nu2d}). Thus  $\mathbb{E}_{\mathbf{0}}^{\mathbb{T}_2(n)}[M(s_n+\sqrt{s_n})]\lesssim (s_n + \sqrt{s_n} )/n^2  \ll \log n$. However,   \[\mathbb{E}_{\mathbf{0}}^{\mathbb{T}_2(n)}[M(s_n+\sqrt{s_n}) \mid M(s_n)>0   ] \ge \sum_{i=0}^{\lfloor s_n \rfloor} P_{\mathbb{T}_2(n)}(a_n,a_n) \gtrsim \log n.   \]  It follows that $\Pr_{\mathbf{0}}^{\mathbb{T}_2(n)}[M(s_n)>0 ] \le \frac{\mathbb{E}_{\mathbf{0}}^{\mathbb{T}_2(n)}[M(s_n+\sqrt{s_n})]  }{\mathbb{E}_{\mathbf{0}}^{\mathbb{T}_2(n)}[M(s_n+\sqrt{s_n}) \mid M(s_n)>0   ] } =o(1) $. 
\end{itemize}
\end{proof}    
\begin{lem}
\label{lem:rangeZdandtori}
Let  $d \ge 3$.  Let  $R(t):=\{X_i:i \in ]t[ \}$ be the range of the first $t$ steps of SRW on a graph $G$.  
  If $G=\Z^d$ for then \[t^{-1}|R(t)| \to \rho(d):= \Pr_x[T_x^+=\infty]>0 \quad \as \text{ and in }L_1.\] Moreover,  if $1 \ll t_n=o(n^d)$ then  $t_{n}^{-1} \E^{\mathbb{T}_d(n)}[|R(t_{n})|]=(1 \pm o(1) )\rho(d) $. 
\end{lem}
\emph{Proof.}
The fact that for $\Z^d$ we have  that  $t^{-1}|R(t)| \to \rho(d)$ $\as$ and in $L_1$ is classical (this can be proved using Birkohff's ergodic theorem, applied to the sequence whose $i$th element is the number of visits to $X_i$). Now consider a SRW on $\TT(n)$ for $d \ge 3$. As in the proof of Lemma \ref{l:nu2d} we have that \[\limsup_{n \to \infty} t_{n}^{-1} \E^{\mathbb{T}_d(n)}[|R(t_{n})|] \le \limsup_{n \to \infty} t_{n}^{-1} \E^{\Z^d}[|R(t_{n})|]=\rho(d). \] 
Conversely, let $s=s_n=\lceil \sqrt{t_n} \rceil$. Let $\nu_i:=\sum_{j=0}^iP_{\TT(n)}^j(\mathbf{0},\mathbf{0}) $ and  $\nu_i(a,b):=\sum_{j=0}^iP_{\TT(n)}^j(a,b)$. By Lemma \ref{l:ZdTTd} $\nu_s \rho(d)= 1 \pm o(1)=\nu_t\rho(d)$. Thus   \[\E^{\mathbb{T}_d(n)}[|R(t_{n})|] = \sum_v  \sum_{i=0}^{t_n }\Pr_{\mathbf{0}}^{\mathbb{T}_d(n)}[T_v=i] \] \[ \ge (1-o(1))\rho(d) \sum_v  \sum_{i=0}^{t_n -s}\Pr_{\mathbf{0}}^{\mathbb{T}_d(n)}[T_v=i] \nu_{t_n-i} \] 
\[\ge (1-o(1))\rho(d)\sum_v  \nu_{t_n-s}(\mathbf{0},v)=(1-o(1))\rho(d)(t_{n}-s)=(1-o(1))\rho(d)t_{n}. \quad \text{\qed} \] 
\begin{lem}
\label{l:exploration}
Let $d \ge 2$. Let $n \in \N$. Let $r \le n$. Let $s= 2dr^{2}  $. Let $(X_i)_{i=0}^{\infty}$ be a SRW on $\TT(n)$. Let $R(t):=\{X_i:i \in ]t[ \}$ be the range of its first $t$ steps. Let $B$ be a box in $\TT(n)$ whose side lengths are between $r$ and $2r$. Let   $D \subseteq B$ be such that $|D| \le \frac{3}{4}|B|$.  Let $\widehat R:=R(s) \cap B \setminus D $. There exist some absolute constants $c_0(d),c_1(d),c_2,c_{3}>0$ such that for all $x \in B_v$
\begin{equation}
\label{e:exploration1}
\forall d \ge 3, \quad \mathrm{P}_{x}[|\widehat R |  \ge c_{1}(d) r^2 ] \ge c_0(d) . 
\end{equation}
\begin{equation}
\label{e:exploration2}
\text{For } d =2, \quad \mathrm{P}_{x}[|\widehat R  |\ge c_{2} r^2/ \log r  ] \ge c_3 .
\end{equation}  
\end{lem}
\begin{rem}
The assertion of the previous lemma is sub-optimal.
\end{rem}
\begin{proof} We first consider the case that $d \ge 3$. Observe that $\max_{x,y \in B_v }\|x-y\|_2^2 \le dr^2$. Let $x \in B_v$. Using the local CLT and the definition of $s$ it is not hard to verify that $\mathbb{E}_{x}[|\widehat R  | ]\ge c'(d) r^2$. 
Clearly $\mathbb{E}_{x}[|\widehat R  |^{2} ] \le s^2$. Hence by the Paley-Zygmund inequality we have  \[\mathbb{P}_{x}[|\widehat R  |> \half  c' (d)r^2 ] \ge \frac{1}{4} \cdot \frac{\mathbb{(E}_{x}[|\widehat R  | ] )^2}{s^2} \ge c_0(d) . \]
We now consider the case $d=2$. Using similar reasoning as in Lemma \ref{l:nu2d} we have  that $\mathbb{E}_{x}[|\widehat R  | ] \ge c_3 r^2 / \log s \ge c_4 r^2 / \log r$. Similar reasoning also yields that $b:=\max_x \mathbb{E}_{x}[| R(s) \cap B_v   | ] \le C r^2 / \log r  $. Let $a(\ell):=\max_{x}\Pr_x[| R(s) \cap B_v   | \ge \ell]$.  By general considerations \[\forall \ell_1,\ell_2 \in \N, \quad a(\ell_1+\ell_2) \le a(\ell_1)a(\ell_2)   \]
(this is left as an exercise). Thus $a( k \lceil 2 b \rceil ) \le 2^{-k} $ for all $k \in \N$. Hence $\max_x \mathbb{E}_{x}[| R(s) \cap B_v   | ^{2}] \le C'b^2$. The proof is now concluded using the Paley-Zygmund inequality.
\end{proof}  
\begin{lem}
\label{lem:range}
Let $(X_i)_{i=0}^{\infty}$ be a SRW on some graph $G$. Let $R(t)=\{X_i:i \in ]t[ \}$ be the range of its first $t$ steps. Then for all $s,t,r \in \N $ and $x \in V$ we have 
\begin{equation}
\label{e:submulrang}
\Pr_x[|R(st)| \le r ] \le (\max_{v \in V} \Pr_v[|R(t)| \le r ])^{s}.
\end{equation}
Consequently, for $G=\Z$  there exists $c>0$ such that for all $t>0$ and $a \in (0,1] $ 
\begin{equation}
\label{e:RZ}
\Pr_x[|R(t)| \le a \sqrt{t} ] \le e^{-c a^{-2}}.
\end{equation}
\end{lem}
\begin{proof}
For \eqref{e:submulrang} apply the Markov property $s-1$ times at times $it$ for $i \in [s-1]$. For (i) apply \eqref{e:submulrang} with $s=\lfloor a^{-2} \rfloor $. 
\end{proof}
\section{Tori - Proof of the lower bounds}
\label{s:Toril}
\subsection{Reducing the lower bounds from Theorem \ref{thm: tori} to a cover time problem}
\label{s:lowertori}
Fix some $\la>0$, $d \ge 2$ and $n \in \N$. Let $\mathrm{U}_t(v):=|\{w \in \W :v \in \RR_t(w) \}| $ be the number of particles that picked a walk  that visits  $v $ in its first $t$ steps. Let $Y_v(t):=\Ind{\mathrm{U}_t(v)=0}$ and $Y(t):=\sum_{v \in V }Y_v(t)$. Clearly, if $Y(t)>0$ then $\SS(\TT(n))>t $.

As $\mathbb{T}_d(n)$ is vertex-transitive, we may assume that $\oo$ is random, chosen from $\pi$, the uniform distribution on $\TT(n)$. When $\oo \sim \pi $, by Poisson thinning the distribution on the initial configuration of particles can be described as follows. Conditioned on $|\W|=m $ there are $m$ particles, whose initial positions are chosen independently, uniformly at random. Since the probability that $|\W|> \lceil \la n^d +(\la n^d)^{2/3} \rceil  $ decays exponentially in $(\la n^d)^{1/3} $,  for the purpose of bounding the probability that $Y(t)=0$ from above,  we may assume that initially there are precisely $m=m(\la,n,d):= \lceil \la n^d +(\la n^d)^{2/3} \rceil $ particles (all of whom are active), each of which starts from a vertex chosen independently uniformly at random.

By the following paragraph, in order to bound $\s(\TT(n))$ from below for a certain range of $\la$, all we have to do is to bound from below the cover time by $m$ independent particles, each starting from the uniform distribution, for a corresponding range of $m$. This is achieved in Propositions \ref{thm:coverconcentration}-\ref{thm:coverconcentration3}. %
%
\subsection{Cover time by multiple random walks}
\label{s:cover}
Let $G=(V,E)$ be some finite connected graph. Let $(X_i)_{i=0}^{\infty}$ be a SRW on $G$.  Let $R(t)=\{X_i:i \in ]t[ \}$ be the range of its first $t$ steps. Denote the stationary distribution of the SRW by $\pi$.  The cover time is defined as $\tau_{\mathrm{cov}}(G):=\inf \{t:R(t)=V \} $. Let $t_{\mathrm{cov}}(G):=\max_{v \in V}\E_v[\tau_{\mathrm{cov}}(G)]$, $ H_{\mathrm{max}}(G):= \max_{x,y \in V }\E_x[T_y]$ and $ H_{\mathrm{min}}^A(G):= \min_{x,y \in A: x \neq y }\E_x[T_y]$. Aldous \cite{Cover} showed that  for a sequence $G_n:=(V_n,E_n)$ of finite connected graphs of diverging sizes if $ t_{\mathrm{cov}}(G_{n}) \gg H_{\mathrm{max}}(G_{n})$ then for any sequence of initial states $x_n \in V_n $ we have that $t_{\mathrm{cov}}(G_{n})-\E_{x_n}[\tau_{\mathrm{cov}}(G_{n})] \le H_{\mathrm{max}}(G_{n}) \ll t_{\mathrm{cov}}(G_{n})  $ and that  $\tau_{\mathrm{cov}}(G_{n})/t_{\mathrm{cov}}(G_{n})$  converges in distribution to 1.

 We now recall the elegant Matthews' bound \cite{matthews} and a variant of it due to Zuckerman \cite{coverlower}, which provides the lower bound on $t_{\mathrm{cov}}(G)$  below (see \cite[chapter 11]{levin2009markov} for a neat presentation of both bounds). Let $h(n):=\sum_{i=1}^n \frac{1}{i} $ be the harmonic sum.
\begin{atheorem}
\label{thm:mat}
For every graph $G=(V,E)$ and every $A \subseteq V $ we have that
\[ H_{\mathrm{min}}^A(G) h(|A|) \le t_{\mathrm{cov}}(G) \le H_{\mathrm{max}}(G) h(|V|). \]
\end{atheorem}
Let $\rho(d) $ be as in \eqref{rhodintro}. One can show that for $d \ge 3$ we have that  $H_{\mathrm{max}}(\mathbb{T}_d(n))=(1+o(1))\E_{\pi}[T_x]=(1+o(1))\frac{n^d}{\rho(d)}$ and that for any set $A$ of vertices whose distance from one another is at least $\log n$ (there exists such a set of cardinality $|A| \gtrsim  (n/ \log n)^d $) we have that $H_{\mathrm{min}}^A(\mathbb{T}_d(n))=(1-o(1))H_{\mathrm{max}}(\mathbb{T}_d(n)) $. The term $\log n$ can be replaced by any diverging $r(n)$ satisfying $ r(n)\le n^{o(1)}$.

 For $d=2$ 
one can show  that  \[H_{\mathrm{max}}(\mathbb{T}_2(n))=(1+o(1))\E_{\pi}[T_x]=(1+o(1))\frac{2}{\pi}n^{2}\log n\] and  that for any set $A$ consisting of vertices of distance at least $\sqrt{ n}$ from one another (there exists such a set of cardinality $|A| \gtrsim n$) we have that  $H_{\mathrm{min}}^A(\mathbb{T}_d(n))=(1-o(1)) \half H_{\mathrm{max}}(\mathbb{T}_d(n)) $. By Theorem \ref{thm:mat} it follows that for $d \ge 3$ we have    $t_{\mathrm{cov}}(\mathbb{T}_d(n))=(1 \pm o(1) )(\frac{n^d}{\rho(d)}d\log n )$ and that \[\left(\frac{1}{4}-o(1)\right)\frac{4}{\pi}(n\log n) ^2  \le t_{\mathrm{cov}}(\mathbb{T}_2(n))\le (1 + o(1))\frac{4}{\pi}(n\log n) ^2  .\]

It follows that $ t_{\mathrm{cov}}(\TT(n)) \gg H_{\mathrm{max}}(\TT(n))$ for all $d \ge 2$. Hence by the aforementioned result of Aldous \cite{Cover} we have that $\tau_{\mathrm{cov}}(\TT(n))  $ is concentrated around $t_{\mathrm{cov}}(\TT(n))$ for all $d \ge 2$. As described above, for $d \ge 3$ one can determine the asymptotic of $t_{\mathrm{cov}}(\mathbb{T}_d(n))$ (up to smaller order terms) via   Matthews' bound. In this case,    $\tau_{\mathrm{cov}}(\TT(n))$ exhibits Gumbel fluctuations of order $n^d$ around its mean \cite{Gumbel}. The case $d=2$ is much more involved.  
Dembo et al.\ \cite{2dcover} showed that  $t_{\mathrm{cov}}(\mathbb{T}_2(n))=(1 \pm o(1) )\frac{4}{\pi}(n\log n)^2$. More refined results can be found at \cite{ding,2dcoverLD,2dcoversubleading}.  

\medskip

The cover time of a graph using many independent random walks was first studied in \cite{Cover1} and later also in \cite{Cover2} and \cite{many3}. 
We now consider the cover time of $\TT(n) $ (for $d \ge 2$) by multiple independent random walks starting from the uniform distribution. The analysis below is used to derive the lower bound on $\s(\TT(n))$  from Theorem \ref{thm: tori}, as explained in \S\ref{s:lowertori}.

Let $G=(V,E)$ be a finite connected vertex-transitive graph. Let $\mathbf{S}^1=(S_m^1)_{m=0}^{\infty} $, $\mathbf{S}^2=(S_m^2)_{m=0}^{\infty}, \ldots$ be independent SRWs on $G$ such that $S_0^i \sim \pi$ for all $i$ (i.e.\ the initial position of each walk is chosen from the uniform distribution). We think of $\mathbf{S}^i$ as the walk performed by the $i$th particle from some collection of particles. For $t \in \N$ let  $R_i(t):=\{S_0^i,\ldots,S_t ^{i}\}$ be the range of the first $t$ steps of the $i$th particle. We may consider the cover time when the length of the walks of the particles is fixed and the number of walks varies or vice-versa: For  $s,t \in \N$ let
\begin{equation}
\label{e:defofDandC}
\mathcal{C}(G,t):=\inf \{s:V=\cup_{i \in [s] }R_i(t) \} \; \text{ and } \; \mathcal{D}(G,s):=\inf \{t:V=\cup_{i \in [s] }R_i(t) \}.
\end{equation}
While for our applications we are mostly interested in $\mathcal{D}(G,s) $ as a function of $s$, it is  easier to first analyze $\mathcal{C}(G,t) $ as a function of $t$, and then relate the two via the relation  
\begin{equation}
\label{e:CD}
\PP[\mathcal{D}(G,s)>t]=\PP[\cup_{i \in [s] }R_i(t) \neq V ]=\PP[\mathcal{C}(G,t)>s]. 
\end{equation}

\medskip

Fix some $v \in V$. By symmetry we have that $\mathbb{P}[v \in R_1(t) \mid|R_1(t)|=k ]=k/|V|$. Let $R(t)$ be the range of a length $t$ walk. By averaging over $|R_1(t)| $ we get that
\begin{equation}
\label{e:ptr}
p_{t}:=\Pr_{\pi}[T_v \le t ]\mathbb{=P}[v \in R_1(t) ]=\frac{1}{|V|} \mathbb{E}[|R_1(t)|]=\frac{1}{|V|} \mathbb{E}[|R(t)|] . \end{equation}
Observe that by a union bound on $V$, for all $\delta >0$ we have that
\begin{equation}
\label{e:triv}
\mathbb{P}[\mathcal{C}(G,t) > \frac{1}{p_t} (1+\delta)\log |V|   ] \le |V| (1-p_{t}) ^{\lfloor \frac{1}{p_t} (1+\delta)\log |V|  \rfloor} \le \frac{1}{  (1-p_{t})|V|^{\delta}}.
\end{equation}
Consequently, \[\frac{ p_t\mathbb{E}[\mathcal{C}(G,t) ]}{ \log |V|} -1\le  \int_{1}^{\infty}\PP[\frac{p_t\mathbb{}\mathcal{C}(G,t)}{ \log |V|} > s]ds \le  \int_{1}^{\infty}  \frac{ds}{  (1-p_{t})|V|^{s-1}}=\frac{1}{ (1-p_{t})\log |V|}.\] Thus
\begin{equation}
\label{e:triv2}
\mathbb{E}[\mathcal{C}(G,t) ] \le \frac{1}{p_t} \left(\log |V| +\frac{1}{1-p_{t}} \right)=\frac{ |V| [\log |V|+(1-p_{t})^{-1}]}{\mathbb{E}[|R(t)|]}.
\end{equation}
The following is a variant of Matthews' argument (or more precisely, of Zuckerman's refinement of it) for multiple random walks. Similar variants appear in  \cite{Cover1} and \cite{Cover2}. 

\begin{thm}
\label{thm:babymat}
Let $G=(V,E)$ be a vertex-transitive graph. Let $A \subseteq V$. Let $\alpha_{t}(A):= \max_{x,y \in  A :x \neq y }\Pr_{x}[T_y \le t]$. Then
\[\mathbb{E}[\mathcal{C}(G,t) ] \ge (1-\alpha_{t}(A)) \frac{1}{p_t}  h(|A|).  \]
\end{thm}
\begin{proof} Let $\sigma:[|A|] \to A $ be a bijection chosen uniformly at random.  Recall that $R_i(t):=\{S_0^i,\ldots,S_t ^{i}\}$ is the length $t$ walk performed by the $i$th particle. Let $L_0=0$ and for $\ell \in [|A|]$ let $L_\ell$ be the first $j$ such that $\{\sigma(k):k\in [\ell] \} \subseteq \cup_{r=1}^{j}R_r(t)$. We argue that \[\forall \, \ell \in [|A|], \quad  \mathbb{E}[L_{\ell}-L_{\ell-1}] \ge (1-\alpha_{t}(A)) \frac{1}{ \ell p_t},\] where the expectation is taken jointly over the walks and the random labeling. This implies the assertion of the theorem by summing over $\ell \in [|A|]$. 

We argue that $\PP[L_{\ell}-L_{\ell-1} \neq 0 ] \ge  \frac{1}{ \ell}(1-\alpha_{t}(A)) $, which concludes the proof, as clearly \[\mathbb{E}[L_{\ell}-L_{\ell-1} \mid  L_{\ell}-L_{\ell-1} \neq 0]= \mathbb{E}[L_{1}]=\frac{1}{p_t}.  \] 
Denote the hitting time of $x$ by the $i$th particle by $T_x^i:=\inf \{s:S_s ^{i}=x \} $. Let $\tau_x:= \inf \{j:x \in R_j(t)  \} $ be the index of the first particle to hit $x$ if each particle walks for $t$ steps. Let $\hat \tau_x:=(\tau_x-1)t+T_{x}^{\tau_x}$ be 
the total number of steps until $x$ is hit (remember that each one of the initial $\tau_x$ particles involved, walks for $t$ steps; Imagine the first particle first performing $t$ steps, followed by the $t$ steps of the second particle, etc.). 
As we labeled the set $A$ using a  labeling chosen uniformly at random,  $\PP[\hat \tau_{\sigma(\ell)} = \max_{i \in [ \ell ] }\hat \tau_{\sigma(i)}  ]=\frac{1}{\ell} $. 
Let us condition on $\hat \tau_{\sigma(\ell)} = \max_{i \in [ \ell ] }\hat \tau_{\sigma(i)}    $ and on that $\sigma(\ell)=y$ (for some $y \in A$). 
Let us  condition further on $\max_{i \in [ \ell-1 ] }\hat \tau_{\sigma(i)} =\hat \tau_{\sigma(r)} $, $\tau_{\sigma(r)}=m $, that $\sigma(r)=x$ (for some $x \in A$) 
and that $T_{\sigma(r)}^m=s$. The conditional probability that $\tau_{\sigma(\ell)}=\tau_{\sigma(r)} $ (i.e.\ that $L_{\ell}-L_{\ell-1}=0$) is by the Markov property the same as 
the probability that a SRW of length $t-s $ starting from $x$ hits $y$, which is at most $\alpha_{t}(A) $. 
\end{proof}

\begin{prop}
\label{thm:coverconcentration}
Let $d \ge 3$. Let $\rho(d) $ be as in \eqref{rhodintro}.  Then the following holds

\begin{itemize}
\item[(1)] If $1 \ll t_n \ll n^d   $ then $\E[\mathcal{C}(\TT(n),t_{n})]=(1 \pm o(1) )\frac{dn^d \log n}{t_{n} \rho(d) }$  and for every fixed $\eps>0$ we have that $\mathbb{P}\left[\left|\frac{\mathcal{C}(\TT(n),t_{n})}{\E[\mathcal{C}(\TT(n),t_{n})]}-1\right|>\eps \right]=o(1) $. 
\item[(2)] If  $ \log n \ll s_n \ll n^d \log n   $ then for every fixed $\eps>0$ we have that \[\mathbb{P}\left[\mathcal{D}(\TT(n),s_{n}) \le (1- \eps) \frac{dn^d \log n}{s_{n} \rho(d) } \right]=o(1).\]
\item[(3)] If  $ \log n \ll \la_n n^d \ll n^d \log n   $ then for every fixed $\eps>0$ we have that  \[\mathbb{P}_{\la_n}\left[\s ( \TT(n)) \le (1- \eps) \frac{d\log n}{\la_{n} \rho(d) } \right]=o(1).\]
\end{itemize}
\end{prop}

\begin{proof}
We first prove part (1). Let  $1 \ll t_n \ll n^d $.  By \eqref{e:ptr} in conjunction with Lemma \ref{lem:rangeZdandtori} $p_{t_n}=(1 \pm o(1) )\frac{\rho(d) t_n }{n^d} $. By \eqref{e:triv}-\eqref{e:triv2} we have that  $\E[\mathcal{C}(\TT(n),t_{n})] \le (1 + o(1) )\frac{dn^d \log n}{t_{n} \rho(d) }$ and that for every fixed $\eps>0$ for all sufficiently large $n$ we have that 
\begin{equation}
\label{e:coverd}
\mathbb{P}\left[\mathcal{C}(\TT(n),t_{n})>(1+\eps)\frac{dn^d \log n}{t_{n} \rho(d) }\right] \le n^{- \eps d  /2}.
\end{equation}

Now pick some collection $A$ of vertices at distance at least $ \log n $ from one another such that $|A| \gtrsim (n / \log n)^d $ (we could have replaced $\log n$ by any other diverging function which is $ \le n^{o(1)} $). By Fact \ref{f:heatkernelZd} and Lemma \ref{l:ZdTTd} $\alpha_{t_n}(A)=o(1)$. 

Using Theorem \ref{thm:babymat} we have that  $\E[\mathcal{C}(\TT(n),t_{n})] \ge (1 - o(1) )\frac{dn^d \log n}{t_{n} \rho(d) }$. Let \[Y:=\max \{\E[\mathcal{C}(\TT(n),t_{n})]-\mathcal{C}(\TT(n),t_{n}),0 \} \quad \text{and} \] \[Z:=\max \{\mathcal{C}(\TT(n),t_{n})-\E[\mathcal{C}(\TT(n),t_{n})],0 \} .\] Then $\mathbb{E}[Y]=\mathbb{E}[Z]=o(\E[\mathcal{C}(\TT(n),t_{n})]) $ (by \eqref{e:coverd}). Finally,
\begin{equation*}
\begin{split}
\mathbb{P}[\mathcal{C}(\TT(n) ,t_{n})\le (1- \eps ) \E[\mathcal{C}(\TT(n),t_{n})]] &=\mathbb{P}[Y \ge  \eps  \E[\mathcal{C}(\TT(n),t_{n})]] \\ & \le \frac{\mathbb{E}[Y]}{\eps  \E[\mathcal{C}(\TT(n),t_{n})]}=o(1). 
 \end{split}
 \end{equation*}
Part (2) follows from part (1) via \eqref{e:CD}. Part (3) follows in turn from part (2) via \S\ref{s:lowertori}.
\end{proof}

\begin{prop} 
\label{thm:coverconcentration2}
\begin{itemize}
Let $\eps>0$. Let $\frac{1}{\log n} \ll \delta_n=o(1) $. Let \[f_{n}(m):=\frac{2n^2 \log n}{ \pi m   } \log \left( \frac{n^2 \log n }{m} \right).\]  
\item[(1)] If $1 \ll t_n \le n^{\delta_n }   $ then $\E[\mathcal{C}(\mathbb{T}_2(n),t_{n})]=(1 \pm o(1) )\frac{2n^2 \log n \log t_n }{\pi t_{n}  } $  and \[\mathbb{P}\left[\left|\frac{\mathcal{C}(\mathbb{T}_2(n),t_{n})}{\E[\mathcal{C}(\mathbb{T}_2(n),t_{n})]}-1\right|>\eps \right]=o(1) .\] 
\item[(2)] If  $ n^{2-\delta_n}   \le s_n \ll n^2 \log n   $ then \[\mathbb{P}[ \mathcal{D}(\mathbb{T}_2(n),s_{n}) \le (1-\eps) f_{n}(s_{n}) ]=o(1).\]
\item[(3)] If  $ n^{2-\delta_n}   \le \la_n n^2 \ll n^2 \log n   $ then \[\mathbb{P}_{\la_n}[ \s(\mathbb{T}_2(n)) \le (1-\eps) f_{n}(\la_n n^2 ) ]=o(1).\]
\end{itemize}
\end{prop}

\begin{proof}
The proof is analogous to that of Theorem \ref{thm:coverconcentration}. Let $t_n \gg 1$. We first prove part (1). By \eqref{e:ptr} in conjunction with Lemma \ref{l:nu2d} $p_{t_n}=(1 \pm o(1) )\frac{  \pi t_n }{  n^2 \log t_n } $. By \eqref{e:triv}-\eqref{e:triv2} we have that $\E[\mathcal{C}(\mathbb{T}_2(n),t_{n})] \le (1 + o(1) )\frac{2n^2 \log n \log t_n }{\pi t_{n}  }$ and that for every fixed $\eps>0$ for all sufficiently large $n$ we have that 
\begin{equation}
\label{e:cover2}
\mathbb{P}\left[\mathcal{C}(\mathbb{T}_2(n),t_{n})>(1+\eps)\frac{2n^2 \log n \log t_n }{\pi t_{n}  }\right] \le n^{- \eps }.
\end{equation}

Let $\frac{1}{\log n} \ll \delta_n=o(1) $. Assume that $ t_n \le n^{\delta_n} $. Pick some collection $A$ of vertices at distance at least $ n^{\delta_n} $ from one another such that $|A| \gtrsim n^{2-2 \delta_n }$. By Theorem \ref{thm:babymat} we have that  $\E[\mathcal{C}(\mathbb{T}_2(n),t_{n})] \ge (1-\delta_n - o(1) )\frac{2n^2 \log n \log t_n }{\pi t_{n}  }$. The proof of part (1) is concluded in an analogous manner to that of Theorem \ref{thm:coverconcentration}. 
Part (2) follows from part (1) via \eqref{e:CD} together with some algebra (namely, $s_n=(1 \pm o(1) )\frac{2n^2 \log n \log t_n }{\pi t_{n}  }$ iff $t_n=(1 \pm o(1))f_{n}(s_n)$). Part (3) follows from part (2) via \S\ref{s:lowertori}.
\end{proof}
We now consider $\mathcal{C}(\mathbb{T}_2(n),t_{n})$ for $1 \ll t_n \ll n^{2 } \log n $  and $ \mathcal{D}(\mathbb{T}_2(n),s_{n})$ for $\log  n   \ll s_n \ll n^2 \log n $. Observe that \eqref{e:cover2} is still valid. 
 Besides that, by taking $A$ to be a set of vertices at distance at least $\sqrt{ n }$ from one another such 
that $|A| \gtrsim n$, from the proof of Proposition \ref{thm:coverconcentration2} we get that
\begin{equation}
\label{e:CT2n}
(1/4 -o(1) ) \left(\frac{2n^2 \log n \log t_n }{\pi t_{n}  }\right)\le  \E[\mathcal{C}(\mathbb{T}_2(n),t_{n})]\le( 1 + o(1)) \left(\frac{2n^2 \log n \log t_n }{\pi t_{n}  }\right),
\end{equation}
where we have used the estimate $\alpha_{n^2 k_n }(A) \le \frac{1}{2}+o(1)$ for $k_n \ll \log n$, which follows from \eqref{e:2dhitprob2} in conjunction with 
Lemma \ref{l:ZdTTd2}. We strongly believe that $\mathcal{C}(\mathbb{T}_2(n),t_{n})$ is concentrated around $\frac{2n^2 \log n \log t_n }{\pi t_{n}  }$  
for $1 \ll t_n \ll n^{2 } \log n $   and that  $ \mathcal{D}(\mathbb{T}_2(n),s_{n})$ is concentrated around 
$f_n(s_n):=\frac{2n^2 \log n}{ \pi s_{n}} \log \left( \frac{n^2 \log n }{s_{n}} \right)$ for  $\log  n   \ll s_n \ll n^2 \log n$. 
We note that it is not hard to deduce from \eqref{e:cover2}-\eqref{e:CT2n} that  $\mathbb{E}[ \mathcal{D}(\mathbb{T}_2(n),s_{n})] \gtrsim f_n(s_n)$. 
However, one has to work harder in order to show that for some $c \in (0,1) $ we have that $\mathbb{P}[ \mathcal{D}(\mathbb{T}_2(n),s_{n}) \le c f_{n}(s_{n}) ]=o(1)$. 
For instance, using Theorem 1 in \cite{Cover} with some additional work one can show that $\mathcal{C}(\mathbb{T}_2(n),t_{n})$  is concentrated around its mean. 
Below we take a different approach. 
   
\begin{prop}
\label{thm:coverconcentration3}
\begin{itemize}
 Let $f_{n}(m):=\frac{2n^2 \log n}{ \pi m   } \log \left( \frac{n^2 \log n }{m} \right)$.  

\item[(1)] If  $ \log n \ll s_n \ll n^2 \log n   $ then \[\mathbb{P}\left[ \mathcal{D}(\mathbb{T}_2(n),s_{n}) \le  \frac{f_{n}(s_{n})}{16}  \right]=o(1).\]
\item[(2)] If  $ \log n   \ll \la_n n^2 \ll n^2 \log n   $ then \[\mathbb{P}_{\la_n}\left[ \s(\mathbb{T}_2(n)) \le  \frac{ f_{n}(\la_n n^2 )}{16} \right]=o(1).\]
\end{itemize}
\end{prop}

\begin{proof}
Part (2) follows from part (1). To be precise, this follows formally from the fact that in the proof below we actually get that for some $\eps>0$,  \[\mathbb{P}\left[ \mathcal{D}(\mathbb{T}_2(n),s_{n}) \le f_{n}(s_{n})\frac{1+\epsilon}{16}  \right]=o(1).\]  

We now prove part (1). Let $t_n:=\lceil f_{n}(s_n)/16\rceil $, where   $ \log n\ll s_n \ll n^2 \log n   $. Then $1 \ll t_n \ll n^2 \log n $.  Consider a walk on $\mathbb{T}_2(n)$ which follows the following rule. At each step w.p.\ $\frac{1}{2t_n} $ it moves to a vertex chosen from the uniform distribution on the vertex set. Otherwise, it makes a SRW step. An equivalent description is that this walk makes a random number of steps of SRW, according a $\mathrm{Geometric}(1/{2t_n})$ random variable (here and below, $Z \sim \mathrm{Geometric}(p) $ means that $\Pr (Z=k)=p(1-p)^{k-1}$ for all $k \in \N$ and thus its mean is $\frac{1}{p}$), before moving to a vertex chosen from the uniform distribution on the vertex set. After doing so, it repeats this rule. We call each such duration between two consecutive jumps to a vertex chosen from the uniform distribution on the vertex set a \emph{mini-walk}.

 We argue that the cover time for this walk is $\whp $ at least $r_{n}:=(1-o(1))n^{2}\frac{\log t_n \log n}{2\pi}  $. This implies the assertion of part (1) as $\whp$ by time  $r_{n}$ the walk makes at least \[(1-o(1))\frac{r_n}{2t_n}e^{-1/2} =(1-o(1))n^{2}\frac{\log t_n \log n}{4\pi t_{n}}e^{-1/2}   \ge n^{2}\frac{\log t_n \log n}{4\pi t_{n}} \cdot  \frac{1}{2} \ge s_{n}  \] mini-walks of length at least $t_n$ (where the last inequality follows from the definition of $t_n$ after some algebra). To see this, observe that by time $r_n$ the walk makes $\whp$ $(1\pm o(1))\frac{r_n}{2t_n}=(1\pm o(1))n^{2}\frac{\log t_n \log n}{4\pi t_{n}}$ mini-walks  and $\whp$ roughly a $\mathbb{P}[\mathrm{Geometric}(\frac{1}{2t_n}) \ge t_n] = (1 - o(1))e^{-1/2}>\frac{1}{2} $ fraction of them are of length at least $t_n$. Consequently, $\whp$ $s_n$ walks of length $t_n$ do not cover $\mathbb{T}_2(n)$, which is the assertion of part (1).

Denote the new walk by $\mathbf{\widehat X}:= (\widehat X_i)_{i=0}^{\infty}$. Denote its law   by $\mathrm{\widehat P}$ and its transition matrix by $\widehat P$. Denote the corresponding expectation by $\mathbb{\widehat E}$.
 Let $A $ be a set of vertices at distance at least $\sqrt{n}$ from one another such that $|A| \ge n/4$. We argue that for the new walk 
\begin{equation}
\label{e:widehatET}
(1 - o(1) )n^{2}\frac{\log t_n}{2\pi} \le \min_{a,b  \in A,a \neq b} \mathbb{\widehat E}_a[T_b] \le \max_{a,b  \in \mathbb{T}_2(n)} \mathbb{\widehat E}_a[T_b] \le (1+o(1))n^2\log t_n    .
\end{equation}
By Theorem \ref{thm:mat} this implies that the expectation of the cover time for  $\mathbf{\widehat X}$ is indeed at least $r_{n}=(1 - o(1) )n^{2}\frac{\log t_n \log n }{2\pi}$. As $r_n \gg \max_{a,b  \in \mathbb{T}_2(n)} \mathbb{\widehat E}_a[T_b]  $,  by the aforementioned general result of Aldous \cite{Cover} the cover time  of  $\mathbf{\widehat X}$ is concentrated around its mean, which is larger than $r_n$ as desired. It remains only to prove  \eqref{e:widehatET}. 

 Let $Z_{a,b}:=\sum_{i \ge 0} (\widehat P^{i} (a, b)-n^{-2}) $. It is classical (\cite[Lemma 2.12]{aldous}) that
\begin{equation}
\label{e:ZabZ} \mathbb{\widehat E}_{x}[T_y]=n^{2}(Z_{y,y}-Z_{x,y})=n^{2}\sum_{i \ge 0} (\widehat P^{i} (y, y)-\widehat P^{i} (x, y)) \end{equation}
(cf.\ \cite[\S6.2]{hermonfrog} for a proof that $\sum_{i \ge 0} |\widehat P^{i} (u, v)-n^{-2}|<\infty $ for all $u,v \in \mathbb{T}_2(n)$). 

Let $\rho \sim \mathrm{Geometric}(1/{2t_{n}})$ be the first time at which $\mathbf{\widehat X}$ moves to a random position chosen uniformly at random (i.e.\ $\rho$ is the duration of the first mini-walk). Consider a coupling of  $\mathbf{\widehat X}$ with SRW on $\mathbb{T}_2(n)$ $\mathbf{X}:=(X_i)_{i=0}^{\infty} $,  in which both walks agree up to time $\rho -1$ and  $\mathbf{X}$ is independent of $\rho$.  Now consider a coupling of  $\mathbf{\widehat X}$ started from $y$ with $\mathbf{\widehat X}$ started from $x$ in which both walks have the same duration   $\rho$ for their first mini-walk and at time $\rho$ both walks move to the same location, and from that moment on both walks are equal to each other. Until time $\rho$ both walks are coupled with SRW on $\mathbb{T}_2(n) $ as described above (for instance, one can couple them with two independent SRWs started from $y$ and $x$, respectively). Let $N(y) $ be the number of visits to $y$ during a single mini-walk.  It is easy to see from this coupling that \[\sum_{i \ge 0} (\widehat P^{i} (y, y)-\widehat P^{i} (x, y))=\mathbb{E}[\sum_{i = 0}^{\rho-1} ( P^{i} (y, y)- P^{i} (x, y))]=\mathbb{\widehat E}_y[N(y)]-\mathbb{\widehat E}_x[N(y)]. \] By the memoryless property of the Geometric distribution and the Markov property (used in the second equality below to argue that $\mathbb{\widehat E}_x[N(y)]=\Pr_x[  T_{y}  < \rho ] \mathbb{\widehat E}_y[N(y)]$) we get that
\begin{equation}
\label{e:hatPyy}
\sum_{i \ge 0} (\widehat P^{i} (y, y)-\widehat P^{i} (x, y))=\mathbb{\widehat E}_y[N(y)]-\mathbb{\widehat E}_x[N(y)]= \widehat \Pr_x[  T_{y} \ge \rho ] \mathbb{\widehat E}_y[N(y)].  \end{equation}
Finally, as $1 \ll t_n \ll n^2 \log n$  by \eqref{e:2dhitprob2} in conjunction with Lemma \ref{l:ZdTTd2} we have that $ \widehat \Pr_x[  T_{y} \ge \rho ] \ge 1/2 -o(1) $ for all $x \neq y \in A$ and by Lemma \ref{l:nu2d} $ \mathbb{\widehat E}_y[N(y)]=(1 \pm o(1) )\frac{\log t_n}{\pi}$. Thus by \eqref{e:ZabZ}-\eqref{e:hatPyy} \[\mathbb{\widehat E}_{x}[T_y] \ge n^{2}[1/2 -o(1)]\left(\frac{\log t_n}{\pi}\right) =(1 - o(1) )n^{2}\frac{\log t_n}{2\pi}, \quad   \text{for all }x \neq y \in A,\] and  $\mathbb{\widehat E}_{x}[T_y] \le (1 + o(1) )n^{2}\frac{\log t_n}{\pi}  $ for all $x , y \in \mathbb{T}_2(n) $,  as desired.
\end{proof}
\subsection{Proof of Proposition  \ref{prop:con2.5}}
\label{s:prop2.8}
The proof of Proposition  \ref{prop:con2.5} uses McDiarmid's inequality, which we now state. Let $f :\XX^n \to \R $. Let $c_i:=\sup|f(\mathbf{x})-f(\mathbf{y})|$, where the supremum is taken over all $\mathbf{x}=(x_1,\ldots,x_n), \mathbf{y}=(y_1,\ldots,y_n) \in \XX^n $ such that $x_j=y_j $ for all $j \neq i$. Let $X_1,\ldots,X_n $ be i.i.d.\ $\XX$ valued random variables. Then for all $\eps > 0$ we have that
\begin{equation}
\label{e:Mcstatement}
 \Pr[f(X_1,\ldots,X_n)< \mathbb{E}[f(X_1,\ldots,X_n)] - \eps ] \le \exp \left(-\frac{2\eps^2}{\sum_{i \in [n]}c_i^2} \right).
 \end{equation}

\noindent \emph{Proof of Proposition \ref{prop:con2.5}:}
We first prove \eqref{checksupper}. The rightmost inequality follows from the fact that for regular graphs $P^{t}(v,v)-\frac{1}{|V|} \lesssim \frac{1}{\sqrt{t+1}} $ for all $v \in V $ and all $t $ (e.g., \cite{PS}). Let $N_k(x)$ be the number of visits to $x$ by time $k$ by the first walk. By Lemma \ref{l:hitviaexpectation}
\begin{equation*}
\frac{k/|V|}{\hat \nu_k} \le \Pr_{\pi}[T_{v} \le k ]  \le \frac{2k/|V|}{\nu_k},
\end{equation*}
from which the rest of the inequalities in \eqref{checksupper} follow.

Consider $m:=\lceil \la |V| \rceil$  i.i.d.\ walks $(\mathbf{X}_1,\ldots,\mathbf{X}_m)$ started from stationarity (where $\mathbf{X}_i=(X_i(t))_{t=0}^{\infty} $ for all $i$). Let $f_{t}((\mathbf{X}_1,\ldots,\mathbf{X}_m))$ be the number of vertices not visited by any of the walks in their first $t$ steps.  By the McDiarmid's inequality with $\max_{i}c_i \le t+1$, writing $\mu_{t}:= \mathbb{E}[f_{t}((\mathbf{X}_1,\ldots,\mathbf{X}_m))$ we have that
\begin{equation}
\label{e:ft}
\Pr[f_{t}((\mathbf{X}_1,\ldots,\mathbf{X}_m))< \mu_{t}/2] ] \le \exp \left(-\frac{ \mu_{t}^2}{8mt^2} \right).
 \end{equation}

Using the independence of the walks together with the definitions of $\hat t, \bar t $ and $m$ we see that $\mu_{\hat t -1 } \ge |V| \left(1-\frac{\log |V|}{6\la |V|} \right)^{m} \ge |V|^{5/6-o(1)} $ while \[\mu_{\bar t } \le |V| \left(1-\frac{2\log |V|}{\la |V|} \right)^{m} \le \frac{1}{ |V|}.\] Thus \eqref{e:hatsup} follows by Markov's inequality and the bound on $\PP[\mathcal{D}(G,\lceil \la |V| \rceil)<  \bar t  ] $ is obtained by substituting the bound on $\mu_{\bar t  }  $ in \eqref{e:ft}. The corresponding bound on $\PP_{\la}[ \SS(G)< \hat t ]$ follows by considering the event that the set $A$  of vertices visited by the planted particle (by time $\hat t  $) is of size at most $|V|/2$, conditioning on this set, and then arguing  that even if the rest of the particles are all activated at time 0, some vertex in $V \setminus A $ will not be visited by time $\hat t-1$.  The  probability of this failing can be controlled by  conditioning on the total number of particles,  using the concentration of the Poisson distribution around its mean, and then using the above argument (cf.\ the proof of (2.4) in \cite{SN} for a completely analogous calculation).
\qed

\section{Tori - Proof of upper bounds of Theorem \ref{thm: tori}}
\label{s:Tori}
We start by introducing some notation.  We
think of the vertices of $\TT(n)$  as being labeled by the set $[0,n-1]^d \cap \Z^d$. By abuse of notation, we denote the vertex set of $\TT(n)$ again by $\TT(n)$.     A \textbf{\emph{box}} of side length $r$ is a set of the form \[ \{(x_{1},\ldots ,x_d) \in \TT(n)  : \forall \, i, \exists \, j_i \in \{0,\ldots,r-1\} \text{ such that }   x_i \equiv v_i+j_{i} \text{ mod }n \},\] for some $(v_1,\ldots,v_d) \in [0,n-1]^d \cap \Z^d$. We define the $\ell_p$ distance ($p \ge 1$) between $x,y \in \mathbb{T}_d(n)$, $\|x-y\|_{p} $, as $\min \|x'-y' \|_p$, where the minimum is taken over all pairs $x',y'$ in $\Z^d$ such that $x' \equiv x $ and $y' \equiv y $ mod $n$ (coordinate-wise) and $\| \cdot \|_p$ is the usual $\ell_p$ norm on $\R^d$. The same convention is utilized when we consider a renormalized torus of the form  $\mathbb{T}_d(\lfloor n/r \rfloor )$ (when we replace mod $n$ by mod $\lfloor n/r \rfloor$). We write $x \sim y$ whenever $\|x-y\|_{1}=1 $ (note that for $x,y \in \TT(n)$,  $\|x-y\|_1=1$ iff $x$ and $y$ are neighbors in $\TT(n) $).

Let $ r \in [ n] $ and $m:=\lfloor n/r \rfloor$. Below we often take a partition of $\TT(n)$ into boxes of side length $r$. What we actually mean by this is that
 we partition $\mathbb{T}_d(n)$ into $m^d$ boxes of side length $r$, apart from $O(m^{d-1})$ boxes which may be of uneven side lengths, which are between $r$ and $2r$. The boxes naturally inherit the structure of $\mathbb{T}_d(m )$. Namely, for every $v=(v_{1},\ldots,v_d) \in \mathbb{T}_d(m ) $ we denote by $B_v^{r}$ the unique box in the partition containing $rv:=(rv_{1},\ldots,rv_d)$ (more precisely, one can partition  $\TT(n)$ into $m^r$ boxes as above, such that for each $v \in \TT(m)$ we have that $rv$ belongs to precisely one of these boxes, which we denote by $B_v$). When $r$ is clear from context, we omit it and write $B_v$.

%
%
%

\subsection{Reducing the upper bound on $\s( \TT(n))$ to a spatial homogeneity condition.}
\begin{defn}
\label{d:dense}
Let $\alpha \in (0,1)$. We say that $A \subseteq \TT(n)$ is $(\alpha,r)$-\emph{dense} if for every $v \in \mathbb{T}_d(m ) $ we have that $|A \cap B_v^r | \ge \alpha |B_v^r |$ (where   $m:=\lfloor n/r \rfloor$ and $B_v^r$ is as above). That is, the density of $A$ at each of the $m^d$ boxes of side length $r$ of the partition is at least $\alpha$.
\end{defn}
\begin{defn}
\label{d:hom}
We denote the event that $\RR_t $ is $(\alpha,r)$-\emph{dense} by \\ $\mathrm{Hom}(t,\alpha,r) $.
\end{defn}
\noindent \emph{\textbf{Proof of Theorem \ref{thm: tori}:}} The lower bounds have been established in \S\ref{s:cover} via the discussion in \S\ref{s:lowertori}. Namely,  part 3 of Proposition \ref{thm:coverconcentration} for $d \ge 3$ and for $d=2$, parts 3 and 2, respectively, of Propositions \ref{thm:coverconcentration2} and  \ref{thm:coverconcentration3}.
We now turn the the upper bounds. Let $\eps ,\alpha \in (0,1) $ where $\eps$ is arbitrary, and $\alpha$ will be determined later. We take the lifespan to be $t$ which we now define. The lifespan will depend on $\la,d$ and $n$ (it has a different expressions in the cases $d \ge 3$ and $d=2$).

For  $d \ge 3$ let $\rho(d) $ be as in \eqref{rhodintro}.  For $d \ge 3$ we consider  $\la_n >0 $  such that  $\log n \ll \la_{n}n^d \ll n^{d}\log n$. Let  $t=t(n,\la_n,d,\eps):=(1+3 \eps ) \frac{d\log n}{ \la_n \rho(d) } $. For $d=2$ we set $t=t(n,\la_n,d,\eps):=(1+3 \eps)f(n,\la_{n})$, where  $f(n,\la):=\frac{2}{\pi} \la^{-1} \log n\log(\la^{-1} \log n )$. For $d = 2$ we consider $\la_n >0 $  such that $\log n \ll \la_{n}n^2 \ll n^{2}\log n $. 

We shall consider below boxes of size $L_n$ where $1 \ll L_n \ll n $ and for $d=2$ we have that  $L_n^2 \asymp \la_n^{-1} \log n $, while for $d \ge 3$ we have that  $L_n^2 \ll \la_n^{-1} \log n $.  

The cases $d \ge 3$ and $d=2$ are analogous, with each ingredient from the proof of one having a counter-part in the proof of the other. In both cases we employ a three steps strategy. We partition the particles into three independent sets of densities $\eps \la_n,\eps \la_n$ and $(1-2 \eps) \la_n$, respectively. We include the planted particle $\plant$ in the first set.

 First consider the dynamics only w.r.t.\ the first set (as if the other two sets of particles do not exist) in the case that the particle lifespan is $t$. Observe that this dynamics is exactly the frog model with particle density $ \eps \la_n$ and lifespan $t$. 

Let $A_1$ be the collection of vertices visited by the dynamics of the particles belonging to the first set of particles. Let $\mathcal{A}_1$ be the event that $A_1$ is $(\alpha,L_n)$-dense. Let $\mathcal{W}^2 $ be the collection of particles from the second set which initially occupy $A_1$. Let $A_2$ be the collection of vertices visited by the particles from $\W^2$ during their length $t$ walks. Let $\mathcal{A}_2$ be the event that $A_2$ is $(1- \delta_n ,L_n)$-dense for some $\delta_n=o(1)$ to be determined later. Let $\mathcal{W}^3 $ be the collection of particles from the third set which initially occupy $A_2$. Let $A_3$ be the collection of vertices visited by the particles from $\W^3$ during their length $t$ walks. Let $\mathcal{A}_3$ be the event that $A_3 = \TT $. Denoting the complement of the event $\mathcal{A}_i$ by $\mathcal{A}_i^c$  
we clearly have that

\[\PP_{\la_n}[\s(\TT(n))>  t ] \le \PP_{\la_n}[\mathcal{A}_1^{c}]+\PP_{\la_n}[\mathcal{A}_2^c \mid \mathcal{A}_1]+\PP_{\la_n}[\mathcal{A}_3^c \mid \mathcal{A}_1,\mathcal{A}_2 ].    \]
Note that $ \PP_{\la_n}[\mathcal{A}_1^{c}]= 1- \PP_{\eps \la_n}[\mathrm{Hom}( t ,\alpha,L_{n})  ] $. Theorems \ref{thm:GCfortori2d} and \ref{thm:GCfortori3d} below ensure that for $d \ge 3$ and $d=2$, respectively, we have that $ \PP_{\eps \la_n}[\mathrm{Hom}( t ,\alpha,L_{n})  ]=1-o(1) $ for some $\alpha \in (0,1)$. This is done via a renormalization argument in which $\TT $ is partitioned into boxes of size $L_n$. A variant of this argument is later used to prove that it is sufficient that the lifespan is taken to be of order $\max \{ \la^{-1},1\} $ for $d \ge 3$ and of order  $\max \{ \la^{-1} |\log \la| ,1\} $ for $d =2 $ in order for a fraction of the vertices to be visited by the process before it terminates. We include this variant despite the fact that it is not be used in the proof of Theorem \ref{thm: tori},  since we believe it is of interest in its own right and as its proof involves an elegant comparison with Bernoulli site percolation. 

Finally, we show that for some $\delta_n=o(1)$ we have that  $\PP_{\la_n}[\mathcal{A}_2^c \mid \mathcal{A}_1]=o(1)$ and $\PP_{\la_n}[\mathcal{A}_3^c \mid \mathcal{A}_1,\mathcal{A}_2 ]=o(1) $ in Lemmas   \ref{lem: spreadnew2} and \ref{lem: spreadnew}, respectively, for $d \ge 3$ (respectively, in Lemmas    \ref{lem: spreadnew22d} and \ref{lem: spreadnew2d}, respectively, for $d =2 $).  \qed

\begin{lem}
\label{lem: spreadnew}
Let $d \ge 3$. Let $\rho(d) $ be as in \eqref{rhodintro}. Let $\eps>0$. Let $\la_n >0 $ be such that $\log n \ll \la_{n}n^d \ll n^{d}\log n $.   Let $1 \ll L_n \ll n $ be such that $L_n^2 \ll \la_n^{-1} \log n $. Let $\delta_n=o(1)$. Let $A \subseteq \TT(n)$ be $(1- \delta_n ,L_n)$-dense.   Assume that at each vertex $a \in A $ there are $\Pois(\la_n)$ particles performing $t=t(n,\la_n,d):=\lceil\frac{(1+\eps)d\log n}{ \la_n \rho(d) } \rceil $ steps of SRW, independently. Let $D$ be the collection of vertices which are not visited by a single particle. Then for all sufficiently large $n$ we have that $\mathbb{E}[|D|] \le n^{-\eps d/2} $.
\end{lem}
\begin{proof} Let $v \in \TT(n)$. For $i \in \N$ and $a,b \in \TT(n)$ let $\nu_i:=\sum_{j=0}^i P^j(a,a) $ and $\nu_i(a,b):=\sum_{j=0}^i P^j(a,b) $. Recall that by Lemma \ref{l:ZdTTd} if $1 \ll i \ll n^d$ we have that $\nu_i=(1 \pm o(1) )\frac{1}{\rho(d)} $. Let $s:=\lceil \sqrt{t}\rceil$. Note that for all $a$ we have that 
\begin{equation*}
\begin{split}
\frac{1}{(1 - o(1) ) \rho(d)} \sum_{i=0}^{t}\Pr_{v}[T_a = i ]& \ge  \sum_{i=0}^{t-s}\Pr_{v}[T_a = i ]\nu_{t-i} \\ & \ge  \sum_{i=0}^{t-s}\Pr_{v}[T_a = i ]\nu_{t-s-i}=\nu_{t-s} (v,a).
\end{split}
\end{equation*} 
Summing over all $a \in A$ gives $  \sum_{a \in A}\Pr_{v}[T_a \le t ] \ge (1 - o(1) ) \rho(d) \sum_{i=0}^{t-s}P^i(v,A)$. By symmetry $\sum_{a \in A}\Pr_{a}[T_v \le t ]=  \sum_{a \in A}\Pr_{v}[T_a \le t ]$. 
By Poisson thinning, the number of particles which visited $v$ from $A$ in time $t$ has a Poisson distribution with mean $\mu_v$, where \[\mu_v/\la_n= \sum_{a \in A}\Pr_{a}[T_v \le t ]=  \sum_{a \in A}\Pr_{v}[T_a \le t ]    \]
\[\ge (1 - o(1) ) \rho(d) \sum_{i=0}^{t-s}P^i(v,A) = (1 - o(1) ) \rho(d)(t-s) \ge \la_n^{-1}(1+\eps/2)\log n^{d},   \]
where $\sum_{i=0}^{t-s}P^i(v,A) \ge (1-o(1))(t-s)=t(1-o(1))$ follows from the fact that $A$ is  $(1- \delta_n ,L_n)$-dense  and that  $L_n^2 \ll \la_n^{-1} \log n $ using the local CLT (cf.\ Fact \ref{f:LCLT2}). Thus  $\mathbb{E}[|D|] \le n^d e^{- \min_{v}\mu_v } \le n^{-\frac{\eps d}{2} } $, as desired. 
\end{proof}
\begin{lem}
\label{lem: spreadnew2}
Let $d \ge 3$. Let $\eps,\alpha,c \in (0,1)$ be arbitrary. Let $\la_n >0 $ be such that $\log n \ll \la_{n}n^d \ll n^{d}\log n $.   Let $1 \ll L_n \ll n $ be such that $L_n^2 \ll \la_n^{-1} \log n $. Let $A \subseteq \TT(n)$ be $(\alpha ,L_n)$-dense.   Assume that at each vertex $a \in A $ there are $\Pois(\la_n)$ particles performing $t=t(n,\la_n,d):=\lceil\frac{cd\log n}{ \alpha \la_n \rho(d) } \rceil $ steps of SRW, independently. Let $D$ be the collection of vertices which are  visited by at least one of the  particles.  Then there exists some $\delta_n=o(1)$, depending only on $c$, such that $\whp$ $D $ is  $(1-\delta_n ,L_n)$-dense.
\end{lem}
\begin{proof}
Fix some box $B_v^{L_n}$ of side length $L_n$. Fix some $ (1/L_n)^{1/2d} \ll \beta_n  \ll 1$ to be determined later. Consider an arbitrary collection $F \subseteq B_v^{L_n} $ of vertices of distance at least $\beta_n L_n$ from one another, such that $|F | \gtrsim 1/\beta_n^{d}$. We will show that for some $\delta_n=o(1)$ we have that \[\PP[|D \cap F| <(1-\delta_n) |F|] \le n^{-2 d}.\] This clearly implies the assertion of the lemma via a union bound (as we may partition $\TT(n)$ into $O(n^d)$ such sets as $F$).

Let $s:=\lceil \sqrt{t}\rceil$. Fix some $u \in F$.
 As in the proof of Lemma \ref{lem: spreadnew}, we have that $\sum_{i=0}^{t-s}P^i(u,A) \ge (1-o(1))t \alpha$ (here $A$ is only assumed to be $(\alpha,L_n)$-dense, not $(1-\delta_n,L_n)$-dense as in Lemma \ref{lem: spreadnew}). Again, as in the proof of Lemma \ref{lem: spreadnew} the number of particles which visit $u $ has a Poisson distribution with mean 
\begin{equation}
\label{e:cdlogn}
\begin{split}
\mu_u & \ge (1-o(1))\la_n \rho (d) \sum_{i=0}^{t-s}P^i(v,A) \\ & \ge (1-o(1))\la_n \rho (d)t \alpha  \ge (1-o(1))c d\log n  .
\end{split}
\end{equation}
Given that a certain particle is at $u$ at some time $j \in ]t[ $ the probability that it visited another vertex from $F$ during its length $t$ walk is by reversibility (used to explain the factor 2)   Fact \ref{f:heatkernelZd} and Lemma \ref{l:ZdTTd}   at most \[2\sum_{i=1}^{t}P_{\TT(n)}^{i}(u,F \setminus \{u\} )\le C(d) |F|(\beta_n L_n)^{2-d}+Cn^{-d} t|F|=o(1),\]
where the last equality holds provided that $\beta_n$ is taken to tend to 0 sufficiently slowly (as $t \ll n^d$,  $ |F|(\beta_n L_n)^{2-d} \asymp \beta_n^{2-2d}L_n^{2-d}$ and $L_n \gg 1$, the last equality holds if $\beta_n^{-2d} \ll \min \{ n^{d}/t,L_n^{d-2}\} $).
 Thus (by summing over all  $j \in ]t[ $) the expected number of particles which visit both $u$ and at least one other vertex from $F$, denoted by $\widetilde \mu_n $ is at most $   \mu_u  \cdot o(1) $. Thus the number of particles which visit $u$ and no other vertex in $F$, denoted by $Q_u$, has a Poisson distribution with mean $\widehat \mu_u$, where (by \eqref{e:cdlogn})  
\[ \widehat \mu_u = \mu_u-\widetilde \mu_u=(1-o(1))\mu_u \ge (1-o(1))c d\log n.  \]
By Poisson thinning we have that $(Q_x)_{x \in F}$ are independent and for all sufficiently large $n$ we have that \[\max_{x \in F}\PP[Q_x=0] \le e^{- \min_{x \in F}\widehat \mu_x }  \le e^{-\frac{1}{2} c d\log n } = n^{-cd/2}.\]
Hence if   $\delta_n=o(1) $ is such  that $J:=\lceil \delta_n |F|\rceil  \ge 16/c$, the  probability that  at least $ J$ vertices $x \in F$ satisfy that  $Q_x=0$ is at most \[{|F| \choose J} n^{-|J|cd/2} \le \left(\frac{|F|e}{J n^{cd/2}}\right)^{J} \le \left(\frac{e}{ \delta_n n^{cd/2}}\right)^{ J} \le   \left(\frac{1}{  n^{cd/4}}\right)^{ J} \le n^{-2d}\] for all sufficiently large $n$ (where we have used ${a \choose b  } \le (\frac{ae}{b})^b $ for all  $b \le a \in \N $, and $J \ge 16/c$ is used only in the last inequality).  
\end{proof}
We now state versions of the previous two lemmas for the case $d=2$. 
\begin{lem}
\label{lem: spreadnew2d}
Let $\eps>0$. Let $f(n,\la):=\frac{2}{\pi} \la^{-1} \log n\log(\la^{-1} \log n )$. Let $\la_n >0 $ be such that $\log n \ll \la_{n}n^2 \ll n^{2}\log n $.   Let $1 \ll L_n \ll n $ be such that $L_n^2 \lesssim \la_n^{-1} \log n $. Let $\delta_n=o(1)$. Let $A \subseteq \mathbb{T}_2(n)$ be $(1- \delta_n ,L_n)$-dense.   Assume that at each vertex $a \in A $ there are $\Pois(\la_n)$ particles performing $t=t(n,\la_n,d,\eps):=\lceil (1+ \eps )f(n,\la_n) \rceil $ steps of SRW, independently. Let $D$ be the collection of vertices which are not visited by a single particle. Then for all sufficiently large $n$ we have that $\mathbb{E}[|D|] \le n^{-\eps } $.
\end{lem}
\begin{proof} Let $v \in \mathbb{T}_2(n)$. For $i \in \N$ and $a,b \in \mathbb{T}_2(n)$ let $\nu_i:=\sum_{j=0}^i P^j(a,a) $ and $\nu_i(a,b):=\sum_{j=0}^i P^j(a,b) $. Recall that by Lemma \ref{l:nu2d} if $1 \ll i \ll n^2 \log n $ we have that $\nu_i=(1 \pm o(1) )\frac{\log i }{\pi} $. Let $s:=\lceil  t / \log t \rceil$. By Poisson thinning and symmetry the number of particles which visited $v$ has a Poisson distribution with mean $\mu_v$ where \[\mu_v/\la_n= \sum_{a \in A}\Pr_{a}[T_v \le t ]=  \sum_{a \in A}\Pr_{v}[T_a \le t ] \ge \frac{ (1 \pm o(1) ) \pi}{\log t}  \sum_{a \in A}\sum_{i=0}^{t-s}\Pr_{v}[T_a = i ]\nu_{t-i}   \]
\[\ge \frac{ (1 \pm o(1) ) \pi}{\log t} \sum_{i=0}^{t-s}P^i(v,A) = \frac{ (1 \pm o(1) ) \pi}{\log t}(t-s) \ge \la_n^{-1}(1+\eps/2)\log n^{2},   \]
where the penultimate equality holds using the local CLT, the fact that $A$ is  $(1- \delta_n ,L_n)$-dense  and that  $L_n^2 \ll t $ (cf.\ Fact \ref{f:LCLT2}). Thus  \[\mathbb{E}[|D|] \le n^2 e^{- \min_{v}\mu_v } \le n^{-\eps } .\] 
\end{proof}
\begin{lem}
\label{lem: spreadnew22d}
 Let $\eps,\alpha,c \in (0,1)$ be arbitrary. Let \[f(n,\la):=\frac{2}{\pi} \la^{-1} \log n\log(\la^{-1} \log n ).\] Let $\la_n >0 $ be such that $\log n \ll \la_{n}n^2 \ll n^{2}\log n $.   Let $L_n^2 \asymp \la_n^{-1} \log n $.    Let $A \subseteq \TT(n)$ be $(\alpha ,L_n)$-dense.   Assume that at each vertex $a \in A $ there are $\Pois(\la_n)$ particles performing $t=t(n,\la_n,d):=\lceil\frac{c}{ \alpha }f(n,\la_n) \rceil $ steps of SRW, independently. Let $D$ be the collection of vertices which are  visited by at least one of the  particles.  Then there exists some $\delta_n=o(1)$, depending only on $c$, such that $\whp$ $D $ is  $(1-\delta_n ,L_n)$-dense.
\end{lem}
\begin{proof}
Fix some box $B_v^{L_n}$ of side length $L_n$. Fix some $\sqrt{ 1/ \log L_n} \ll \beta_n  \ll 1$ to be determined later. Consider an arbitrary collection $F \subseteq B_v^{L_n} $ of vertices of distance at least $\beta_n L_n$ from one another, such that $|F | \gtrsim 1/\beta_n^{2}$. We will show that for some $\delta_n=o(1)$ we have that $\PP[|D \cap F| <(1-\delta_n) |F|] \le n^{-4}$. This clearly implies the assertion of the lemma via a union bound (as we may partition $\mathbb{T}_2(n)$ via $O(n^2)$ such sets).

Let $s:=\lceil \sqrt{t / \log t }\rceil$. Fix some $u \in F$.
 As in the proof of Lemma \ref{lem: spreadnew2d} $\sum_{i=0}^{t-s}P^i(u,A) \ge (1-o(1))t \alpha$ (here $A$ is only assumed to be $(\alpha,L_n)$-dense, not $(1-\delta_n,L_n)$-dense as in Lemma \ref{lem: spreadnew2d}). Again, as in the proof of Lemma \ref{lem: spreadnew2d} the number of particles which visit $u $ has a Poisson distribution with mean $\mu_u \ge \frac{(1-o(1))\pi}{\log t}  \la_n \sum_{i=0}^{t-s}P^i(v,A) \ge (1-o(1))c \log n^{2}  $.

Given that a certain particle is at $u$ at some time $i \in t $, the probability that it visited another vertex from $F$ during its length $t$ walk is $o(1)$, by reversibility, \eqref{e:2dhitprob},   and Lemma \ref{l:ZdTTd2}, provided that $\beta_n$ tends to $0$ sufficiently slowly.
 Thus the expected number of particles which visit both $u$ and at least one other vertex from $F$, denoted by $\tilde \mu_u$, is at must $\mu_u \cdot o(1) $. Thus the number of particles which visit $u$ and no other vertex in $F$, denoted by $Q_u$, has a Poisson distribution with mean $\widehat \mu_u$, where  
$ \widehat \mu_u =(1-o(1))\mu_u$. Form this point the proof is concluded in an analogous manner to the proof of Lemma \ref{lem: spreadnew2}. 
\end{proof}

\subsection{Giant component in constant lifespan for tori}
The only missing ingredient in the proof of Theorem \ref{thm: tori} is verifying that, in the notation from that proof, we have that  $ \PP_{\eps \la_n}[\mathrm{Hom}( t ,\alpha,L_{n})  ]=1-o(1) $ for some $\alpha \in (0,1)$. This will be done in Theorems  \ref{thm:GCfortori2d} and \ref{thm:GCfortori3d}.   Before tending to that, we take a detour and establish the emergence of a ``giant component" in constant lifespan (when $\la \gtrsim 1 $), in a sense that will be made precise below. Some of the ideas below will be useful for the proofs of Theorems  \ref{thm:GCfortori2d} and \ref{thm:GCfortori3d}.  
 
Let $d \ge 2$. Let $ K(d) \ge 1$ be some constants to be determined later. Throughout we take  $\la=\la_n$  such that  $ \log n \ll \la n^d \ll n^d \log n$.  Let $\widetilde \la := \min \{1,\la \}$, $\beta:=1/ \widetilde \la$,  
\begin{equation}
\label{e:r(d)}
r=r(d,\la):= \begin{cases} \lceil K(d)  \sqrt{\beta }\rceil & d \ge 3 \\
\lceil K(2) \sqrt{\beta (1+ \log \beta ) }\rceil  & d=2 \\
\end{cases},  
\end{equation}  
\begin{equation}
\label{e:s(d)}
m=m(n,d,\la):=\lfloor n/r \rfloor \quad \text{and} \quad  s=s(d,\la):=  8d  r^{2}(d,\la).
\end{equation}

Consider the variant of the frog model in which the planted particle $\plant$ walks for $\ell $ steps while the rest of the particles walk for $k$ steps. Denote the set of vertices visited by the process before it dies out by $\RR_{k,\ell} $. The following theorem asserts that if $t \gg s=s(d,\la)$, then $\RR_{s,t } \text{ is }(\alpha,Cr(\log n)^{1/(d-1)})\text{-dense}$  $\whp$ for some $\alpha \in (0,1)$ and $C>0$, where $s$ is as in \eqref{e:s(d)}. Note that when $\la_n \asymp 1 $ we have that $s \asymp 1$ and so it follows that $|\RR_{s,t}| \asymp n^d $ $\whp$ provided that $s$ is a sufficiently large constant (in terms of $\la$ and $d$) and that $t \gg 1$. Below we write $\asymp_{d} $ to indicate that the implicit constant may depend on $d$.   
\begin{thm}
\label{thm:GCfortori}
Let $d \ge 2$. Let $\la=\la_n$ be such that   $ \log n \ll \la n^d \ll n^d \log n$.  Let $r=r(d,\la)$ and $s=s(d,\la)$ be as above. Let $t=t(d,\la_{n}) \gg s(d,\la_n) $.   Provided that the constant  $K(d)$ is sufficiently large,
 there exist some $\alpha \in (0,1) $ and some $L_{n,d} \asymp_{d}  (\log n)^{1/(d-1)}  $ such that \[\PP_{\la_{n}}^{\mathbb{T}_d(n)}[\RR_{s(d,\la_{n}),t } \text{ is }(\alpha,r(d,\la_n)L_{n,d})\text{-dense}]=1-o(1) .\]
\end{thm}
Observe that for $d \ge 4$ we have that $r(d,\la_n)(\log n)^{1/(d-1)} \ll \sqrt{ \la_n^{-1} \log n }$, provided $\la_n \ll (\log n)^{1- 2/(d-1)}   $. While this  suffices to conclude the proof of Theorem \ref{thm: tori} only for $d \ge 4$ when  $\la_n \ll (\log n)^{1- 2/(d-1)}$, we think this result is interesting in its own right. Moreover, its proof contains some of the ideas that will be used to prove Theorems \ref{thm:GCfortori2d} and \ref{thm:GCfortori3d} (which are used in order to conclude the proof of Theorem \ref{thm: tori}) in the following two subsections.
  
Let $e_i \in \Z^d $ be the vector whose $j$th coordinate is $\Ind{j=i}$. We partition the particles into $2d+2$ independent sets, $\W^{a},\W^b,\W^{\pm e_1 },\ldots,\W^{\pm e_d }$, where $\W^{a},\W^{b}$ both have density $\la/4$ and each of the other sets has density $\la/(4d)$. We denote the collection of all particles in $\W^{i} $ whose initial position is $u \in \mathbb{T}_d(n) $ (resp.\ $B \subseteq \TT(n) $) by $\W_u^i$ (resp.\ $\W_B^i $), where $i \in \{a,b,\pm e_1,\ldots,\pm e_d \}$.  Then $(|W_u^i|)_{u \in \TT(n),i \in \{a,b\} }$ are i.i.d.\ $\Pois(\frac{\la}{4})$ and  $(|W_u^{i}|)_{u \in \TT(n),i \in\{\pm e_1,\ldots,\pm e_d \} }$ are i.i.d.\ $\Pois(\frac{\la}{4d})$. We still denote the corresponding probability by $\PP_{\la}$. We note that the set $\W^{b}$ will play no role in the analysis  in this subsection. The reason we introduce it now is that it will be used in the following subsections.

\begin{defn}
\label{def:blockdynamics}
Consider a partition of $\TT(n) $ into $m^d$ boxes $(B_v^{r})_{v \in \TT(m) }$ of side length $r$, where $m:=\lfloor n/r \rfloor$. Let $v \in  \mathbb{T}_d(m ) $. We define the $(r,s)$  $a$-\emph{dynamics} on $B_v^{r}$ started  from $B \subseteq B_v^{r}$ to be the variation of the frog model  with lifespan $s$ in which:
\begin{itemize}
\item[(1)]
Initially only $\W_B^a$  is activated.
\item[(2)]
Initially at each $u \in B_v^{r}$ there are $\W_u^{a} \sim \Pois(\frac{\la}{4})$ particles,
and no planted particles. 
\item[(3)]
Initially there are no particles outside $B_v^{r}$.
\end{itemize}
 We denote the collection of vertices in $B_v^{r}$ that are visited by the $(r,s)$   $a$-dynamics on $B_v^{r}$ started from $B$ before it dies out by  $A_B=A_{B}(r,s)$. We say that $B \subseteq B_v^{r}$ is $(r,s)$-\emph{\textbf{good}} if $|A_B| \ge \frac{1}{4} |B_v^{r}| $. We say that $x \in B_v^{r}$ is $(r,s)$-\emph{\textbf{good}} if $|A_x| \ge \frac{1}{4} |B_v^{r}| $, where $A_x:=A_{\{x\}}$. Finally, we define $\mathrm{good}_{r,s}(v):=\{x \in B_v^{r}:x \text{ is $(r,s)$-good}  \}$. When $r$ and $s$ are clear from context we omit them from the aforementioned notation and terminology.   
\end{defn}
Throughout this subsection $r$ and $s$ shall be as in \eqref{e:r(d)}-\eqref{e:s(d)}. However in the following two subsections we shall use Definition \ref{def:blockdynamics} with different choices of $r$ and $s$.

Recall that $\mathcal{R}_s(\mathcal{U}) $ is the union of the ranges of the length $s$ walks performed by the particles in $\mathcal{U}$.

\begin{prop}
\label{p:good}
Let $d \ge 2$. Let $r,m$ and $s$ be as in \eqref{e:r(d)}-\eqref{e:s(d)}. Let $v \in  \mathbb{T}_d(m ) $ and  $B \subseteq B_v^{r}$. Assume that $|B| \le \frac{|B_v^{r}|}{2}$.
 Then there exist some constants $c(d),C>0$ such that
\begin{equation}
\label{e:ABlarge}
 \PP_{\la}[|A_{B}| <\frac{3}{4} |B_v^{r}|  ] \le C \exp[-c(d)\la |B|],
\end{equation} 
\begin{equation}
\label{e:goodness}
\PP_{\la}[\mathrm{good}(v) \cap B = \eset ]  \le C \exp[-c(d)\la |B|]. 
\end{equation}  
\end{prop}
\begin{rem}
When $\la \ge C' d$ it is not hard to use Poisson thinning along with a comparison with Bernoulli site percolation with parameter $1-e^{-\la}-e^{-\la/d} $, along with Proposition \ref{p: positivedensity}, in order to argue that  a slightly weaker assertion than that of   Proposition \ref{p:good} holds with $s=r=1$. This observation can also be used to simplify the proof of Theorem \ref{thm:GCfortori} for $\la \ge C' d$. Below we take a different approach.   
\end{rem}
\begin{proof}
 Fix some $B \subseteq B_v$ (recall that $B_v$ is an abbreviation  of $B_v^{r} $). We first prove \eqref{e:ABlarge}. The proof of \eqref{e:goodness} is essentially identical, but requires slightly more care (the relevant details will be provided later).  Denote $\mathcal{B}_0:=B$. We may expose $A_B$ by first exposing $\mathcal{D}_0:=\RR_s(\W_x^a)$ for some $x=x_{0} \in \mathcal{B}_0$. Continue in this fashion, by exposing in the $i$th stage $\mathcal{D}_i:= \RR_s(\W_{x_i}^{a}) $ for some vertex \[x_{i} \in \mathcal{B}_{i-1} \setminus \mathcal{X}_{i-1}, \quad \text{where} \quad  \mathcal{X}_{\ell}:=\{x_j:0 \le j \le \ell \}, \]   \[\mathcal{B}_{\ell}:= B \cup \mathcal{H_{\ell}}  \cap B_v, \quad \mathcal{H_{\ell}}  := \cup_{j=0}^{\ell}\mathcal{D}_j   \]
and $x_i$ is chosen according to some predetermined rule. 
 Observe that as long as $|\mathcal{B}_i|>i+1$ we can pick some $x_{i+1}\in \mathcal{B}_{i} \setminus \mathcal{X}_{i}$ and continue the above exploration process by exposing  $\mathcal{D}_{i+1}:= \RR_s(\W_{x_{i+1}}^{a}) $. The exploration process is terminate at the first stage $L$ at which $|\mathcal{B}_L|=L+1 $. At that stage $L$ we have that $\mathcal{B}_L=A_B $. 
Observe that for $i<|B|-1$ we have that $|\mathcal{B}_i| \ge |\mathcal{B}_0|=|B|>i+1 $, so the exploration process cannot be terminated by step $|B|-1$.  
Let $U_i$ be the event that $i+1<|\mathcal{B}_i|<\frac{3}{4}|B_v|$.

We first deal with the case $d \ge 3$. 
 By \eqref{e:exploration1}  there exist $c_{0}(d),c_{1}(d)>0$   so that on $U_i$ we have that
$
\E_{\la}[|\mathcal{B}_{i+1}  \setminus  \mathcal{B}_i| \mid\mathcal{D}_0,\mathcal{D}_1 \ldots,\mathcal{D}_i]$   
stochastically dominates a random variable which equals $J:= \lceil c_1 (d) r^2 \rceil   $ with probability at least $p:=(1-\exp[-\la c_0(d)] )$ and otherwise equals 0. We choose $K(d)$ in the definition of $r$ such that $Jp \ge 6 $.

Let  $Z_i$ be the indicator of the event that either $U_i^c $ occurs, or that $ |\mathcal{B}_{i+1}  \setminus  \mathcal{B}_i| \ge J$. We get that the joint law of $Z_1,\ldots,Z_{r^d}$ stochastically dominates that of i.i.d.~Bernoulli r.v.'s with mean $p$ (while they are not independent, by considering the two cases  $U_i^c $ and  $U_i$, we see that the probability that $Z_i=1$ is at least $p$, regardless of the values of $Z_1,\ldots,Z_{i-1}$). Let $\eta_i:=1-Z_i$. Let $S_j:=\sum_{i=1}^j \eta_i$ and  $\widehat S_j:=\sum_{i=1}^j Z_i$. 

We first deal with the case that $\la c_0(d) \le 1$.
Recall that $|A_{B}|  \ge \frac{3}{4} |B_v|$ if for all $i \ge |B|-1 $ we have that  $|\mathcal{B}_i| \ge i+2$. This holds in particular if for all $i \ge |B|-1 $ we have that $ J \widehat S_{i} \ge  i+2 $. Thus
 \[\PP_{\la}\left[|A_{B}| <\frac{3}{4} |B_v|  \right] \le \PP_{\la}\left[ \exists \, j \in [ |B|-1,    r^d] \text{ such that } J \widehat S_{j} \le      j+1\right].\] As $1/J \le p/6 $, by the previous paragraph we have  this probability decays exponentially in $p|B| \asymp c(d) \la |B|$ by \eqref{LDber4}.

We now consider the case $\la c_0(d) \ge 1$. Similarly to $Z_1,Z_2,\ldots$, while $\eta_1,\eta_2,\ldots $ are not independent,  by considering the two cases  $U_i^c $ and  $U_i$ we see that the probability that $\eta_i=1$ is at most $1-p=\exp[-\la c_0(d)]$, regardless of the values of $\eta_1,\ldots,\eta_{i-1}$. Thus  $\eta_1,\eta_2,\ldots \eta_{r^d}$ are  stochastically dominated by i.i.d.\ Bernoulli random variables of mean $1-p$. By the above discussion, and the fact that $J \widehat S_{j} \le      j+1 $ iff $S_{j} \ge j-\frac{j+1}{J} $  
\[ \PP_{\la}\left[|A_{B}| <\frac{3}{4} |B_v|  \right] \le \PP_{\la}\left[ \exists \, j \in [ |B|-1,    r^d] \text{ such that } S_{j} \ge j-\frac{j+1}{J}\right].\]  
As $1-\frac{j+1}{jJ} \ge 1-3/J \ge (1-p)(1+c_2 e^{c(d) \la} )$, by \eqref{LDber3} the probability on the r.h.s.\ decays exponentially in $|B|(1-p)c_2 e^{c(d) \la} \log(1+c_2 e^{c(d) \la} ) \asymp c(d) \la |B|$. This concludes the proof of \eqref{e:ABlarge} when $d \ge 3 $. We now prove \eqref{e:goodness} for $d \ge 3$. The only change is in the choice of $x_i$.
\begin{itemize}
\item If $\mathcal{X}_i \supseteq \mathcal{H}_i $ and $\mathcal{X}_i\supseteq B $ then the exploration process is terminated.

\item If $\mathcal{X}_i \supseteq \mathcal{H}_i$ and $\mathcal{X}_i \nsupseteq B $ then we let $x_{i+1}$ be an arbitrary vertex in $B \setminus \mathcal{X}_i $.
\item  If $\mathcal{H}_i \setminus \mathcal{X}_i \neq \eset $ let $k \in ]i[$ be the maximal index such that  $\mathcal{D}_k \setminus \mathcal{X}_i \neq \eset $. In this case we let $x_{i+1}$ be some arbitrary vertex in  $\mathcal{D}_k \setminus \mathcal{X}_i   $.
 \end{itemize}
As $|B| \le |B_v|/2$ and $3/4 - \half =1/4$, if for all  $j \in [ |B|-1,    r^d]$ we have that $ J\widehat S_{j} \ge j+ 1$ (equiv.\ $S_{j} < j-\frac{j+1}{J}$) then there must be some $x \in B$ which is good.

The proof for the case $d=2$ is analogous with \eqref{e:exploration2} replacing \eqref{e:exploration1} above. 
\end{proof}

\begin{lem}
\label{lem:nice}
Let $d \ge 3$.  Let $\la=\la_n,r=r(d,\la),m=m(n,d,\la)$ and $s=s(d,\la)$ be as above. Let $t=t(n,d,\la_{n}) \gg s(d,\la_n)  $. We say that $v \in \TT(m)$ is \emph{nice} if $|\RR_t(\plant) \cap B_v| \ge r^2(d,\la_n)/16 $. Let $\mathrm{Nice}:=\{v \in \TT(m) :v \text{ is nice} \}$. Then $\whp$ $|\mathrm{Nice}| \gg 1 $. 
\end{lem}
\begin{lem}
\label{l:nice2}
 Let $\la=\la_n,r=r(2,\la_{n}),m=m(n,2,\la_{n})$ and $s=s(2,\la_{n})$ be as above. Let $t=t(n,\la_{n}) \gg s$. We say that $v \in \mathbb{T}_2(m)$ is \emph{nice} if $|\RR_t(\plant) \cap B_v| \ge r^2 /(16 \log r)  $. Let $\mathrm{Nice}:=\{v \in \mathbb{T}_2(m) :v \text{ is nice} \}$. Then $\whp$ $|\mathrm{Nice}| \gg 1 $. 
\end{lem}
 For our purposes the constant $1/16$ in the last two lemmas could have been replaced by any positive constant $c(d)$ (this would only result in a larger choice of the constants $K(d)$ in \eqref{e:r(d)}). Let $\tau_0:=0 $ and $v_0=o$. Let  $\tau_1:= \inf \{t>C s: X_t \in B_v \text{ for some }v \neq o \} $. Let $B_{v_i}$ be the block at which the walk is at time $X_{\tau_i} $ and define inductively  \[\tau_{i+1}:= \inf \{t> \tau_i + Cs : X_t \in B_v \text{ for some }v \notin \{v_{0},v_1,\ldots,v_i \} \}.\] One way of proving the lemmas is by showing that $\whp$ for some $k \gg 1$ we have that $\tau_k \le t $ and that on this event $\whp$ we have that  $|\mathrm{Nice}| \gg 1 $. We omit the details.     

\medskip

\emph{Proof of Theorem \ref{thm:GCfortori}:}
Let $r,m$ and $s$ be as in \eqref{e:r(d)}-\eqref{e:s(d)}. We take an arbitrary ordering of $\TT(m)$. Let $\eps \in (0,1/4) $ be such that $1-\eps$ is greater than $p_\mathrm{c}(d)$, the critical density for Bernoulli site percolation for $\Z^d$ (as $p_{\mathrm{c}}(d)$ is non-increasing in $d$, we may take $\eps < 1- p_{\mathrm{c}}(2)$). While the particles walk on $\TT(n)$ we define   an auxiliary site percolation process on the renormalized torus  $\TT(m)$.  Let $t \gg s $ and  $\mathrm{Nice} \subseteq \TT(m) $ be as in Lemmas \ref{lem:nice} and \ref{l:nice2} (depending on whether $d=2$ or $d \ge 3$). For every $v \in \mathrm{Nice}$ we say that $v$ is \emph{fantastic} if  $\RR_t(\plant) \cap B_v $ is good, where as in Definition \ref{def:blockdynamics} $B \subseteq B_v$ is good if $|A_B| \ge |B_v|/4$. 

By construction and Proposition \ref{p:good}, given that $\RR_t(\plant)=U $ and that $\mathrm{Nice}=F \subseteq \TT(m) $ we have that $(\Ind{v \text{ is fantastic} })_{v \in F } $ are independent, and each indicator equals 0 with probability at most $Ce^{-c(d) \la r^2(d,\la_n)} $ for $d \ge 3$ and at most  $Ce^{-c \la r^2(d,\la_n)/\log r} $ for $d=2$. By the choice of $r$, if $K(d)$ are sufficiently large we have that these probabilities are at most $\eps $. Let $\mathcal{F}_0:=\{v \in F:v \text{ is fantastic} \}$ and $\mathcal{NF}_0:=\{v \in F:v \text{ is not fantastic} \} $ (note that the elements of $\mathcal{F}_0$ and $\mathcal{NF}_0$ lie in $\TT(m)$). For each $v \in \mathcal{F}_{0}$ we set $\mathrm{GC}(v):=A_{\RR_t(\plant) \cap B_v}$. By construction  $|\mathrm{GC}(v)|  \ge  |B_v|/4$
 for every $v \in \mathcal{F}_{0}$.
 
We now describe an exploration process on $\TT(m)$. Its initial input is  $(\mathcal{F}_0,\mathcal{NF}_0) $.  At the beginning of  stage $i$ of the exploration process we will have two sets $\mathcal{F}_i $ and  $\mathcal{NF}_i $ of vertices in $\TT(m)$ and the vertices explored thus far will be $\mathcal{F}_i \cup \mathcal{NF}_i  $. For each $v \in \mathcal{F}_i  $, there will already be a set $\mathrm{GC}(v) \subseteq B_v $ of size at least $\alpha  r^d $     that are guaranteed to be activated from the information exposed in the previous stages. \\ For $A \subseteq \TT(m)$ let  $\pd A:=\{b \in A^{c}  :\exists \, a \in A \text{ such that }\|a-b\|_1=1 \} $ be the external vertex boundary of $A$. The process is terminated at the first stage $i$ at which  $\pd \mathcal{F}_i \subseteq \mathcal{NF}_i $. At the $i$th, if     $\pd \mathcal{F}_i \not \subseteq \mathcal{NF}_i $  we pick  $u \in \pd \mathcal{F}_i \setminus \mathcal{NF}_i$ to be the smallest element of $ \pd \mathcal{F}_i \setminus \mathcal{NF}_i $  in the ordering.  
 If $\mathrm{good}(u) = \eset $ we set $\mathcal{NF}_{i+1}=\mathcal{NF}_i \cup \{u\} $ and  $\mathcal{F}_{i+1} =\mathcal{F}_i$. If $\mathrm{good}(u) \neq \eset $ pick some $\mathbf{o}_u \in \mathrm{good}(u)$. We attempt to ``recruit" $u$ to $\F_{i+1}$ by finding some neighbor $v$ of $u$ in $\F_i$ such that at least of the walkers whose initial location is in $\mathrm{GC}(v)  $ reached $\mathbf{o}_u $. This will allow us to define  $\mathrm{GC}(u):=A_{\oo_u} $ and set $\mathcal{F}_{i+1} =\mathcal{F}_i\cup \{u\}$.

    We now describe this in more detail. Let $v \in \mathcal{F}_i $ be such that $v+\xi=u $ for some $\xi \in \{\pm e_1,\ldots,\pm e_d \}$, where as above $e_k$ is the vector in $\Z^d$ whose $j$th coordinate is $\Ind{j=k}$. We pick this $v$ to be minimal w.r.t.\ the ordering. If  $\oo_u \in \RR_s(\W_{\mathrm{GC}(v)}^{\xi}) $ (i.e.\ if there is a particle from the set $\W^{\xi} $ whose initial position is in $\mathrm{GC}(v)$ which reached $\oo_u$ in its length $s$ walk) we set  $\mathcal{NF}_{i+1}=\mathcal{NF}_i  $,   $\mathcal{F}_{i+1} =\mathcal{F}_i\cup \{u\}$ and $\mathrm{GC}(u):=A_{\oo_u} $. Otherwise, we set  $\mathcal{NF}_{i+1}=\mathcal{NF}_i \cup \{u\} $ and  $\mathcal{F}_{i+1} =\mathcal{F}_i$.

By Proposition \ref{p:good}, at each stage the probability that  $\mathrm{good}(u) = \eset $  is at most $Ce^{-c(d)|B_v|} \le \eps/2 $ (apply Proposition \ref{p:good} taking $B$ to be an arbitrary set of size $|B_v|/2$). Given that   $\mathrm{good}(u) \neq \eset $,  that $\mathrm{GC}(v)=A$ and that $\mathbf{o}_u=y $
 the number of particles from $\W_{A}^{\xi}$ which visit $y$ in their length $s$ walk has a Poisson distribution. Using Poisson thinning and the fact that $\W^{\xi}$ has density $\frac{\la}{4d}$, we argue that the mean of this random variable is at least $ \log (2/\eps)$, provided that $K(d)$ is sufficiently large. Indeed, by the choice of $s$, the fact that $|A| \ge |B_v|/4$ and the local CLT, for  $d \ge 3$  this mean is \[\frac{\la}{4d}\sum_{a \in A }\Pr_a[T_{y} \le s ]=\frac{\la}{4d}\sum_{a \in A }\Pr_y[T_{a} \le s ] \ge \frac{c\la}{4d}\sum_{a \in A }\sum_{i \in [s] }P^i(y,a) \] \[ =\frac{c\la}{4d}\sum_{i \in [s] }P^i(y,A) \ge \la c'(d)s \ge \log (2/\eps) ,  \]
while for $d=2$ it is  \[\frac{\la}{8}\sum_{a \in A }\Pr_y[T_{a} \le s ] \ge \frac{\la}{8}\sum_{a \in A }\Pr_y[T_{a} \le s/2 ]  \ge \frac{c\la}{ \log s }\sum_{a \in A }\Pr_y[T_{a} \le s/2 ] \sum_{i \in [s/2] }P^i(a,a) \] \[ \ge \frac{c\la}{\log s}\sum_{a \in A }\sum_{i \in [s] }P^i(y,a)=\frac{c\la}{\log s}\sum_{i \in [s] }P^i(y,A) \ge \la c's/\log s \ge \log (2/\eps) .  \]
Finally, we get that $u \in \mathcal{NF}_{i+1}  $ w.p.\ at most $\eps/2 + \eps/2=\eps$. Let $i_* $ be the stage at which the process is terminated. By the above analysis we can couple $\mathcal{F}_{i_{*}}$ with Bernoulli site percolation on $\TT(m)$ with parameter $1-\eps> p_{\mathrm{c}}(d) $, such that  $\mathcal{F}_{i_{*}}$ contains the union of the connected components of the vertices in $\mathrm{Nice}$. Since by Lemmas \ref{lem:nice} and \ref{l:nice2} we have that $|\mathrm{Nice}| \gg 1$ $\whp$, the assertion of the theorem now follows from Proposition \ref{p: positivedensity}. Indeed, parts (2)-(3) of Proposition \ref{p: positivedensity} (with the set $\mathrm{Nice}$ here playing the role of $U$ in part (3) of Proposition \ref{p: positivedensity}) assert that conditioned on $|\mathrm{Nice}| \gg 1$ we will have that $\whp$   $\mathcal{F}_{i_{*}}$ is $(\alpha_{1},\lceil C(d)( \log m)^{1/(d-1)} \rceil )$-dense, for some $\alpha_1,C(d)>0$. On this event, the set $\cup_{v \in \mathcal{F}_{i_{*}}  }\mathrm{GC}(v) $, which by construction is contained in $\RR_{s(d,\la_{n}),t }$, must be  $(\alpha,r\lceil C(d)( \log m)^{1/(d-1)} \rceil )$-dense, for some fixed $\alpha>0$, as desired.      \qed

\subsection{$d=2$ }
\label{s:red2}
Throughout this subsection we let $\la=\la_n $ be such that $\log n \ll n^2 \la \ll n^2 \log n $ (even when this is not explicitly specified). Let $\hat s_2=\hat s_2(n,\la,\gd_n):=16 \hat r_2^2 $, where
\begin{equation}
\label{e:rhat}
 \hat r_2=\hat r_2(n,\la,\gd_n):=\left\lceil \gd_n \sqrt{\frac{1}{\la} \log n \log (\frac{1}{\la} \log n) } \right\rceil,
\end{equation}
for some $ \gd_{n} = o(1) $ such that $\hat r_2\gg 1$, to be determined later. Let $\hat m_2=\hat m_2(n,\la):=\lfloor n/ \hat r_2 \rfloor $.
\begin{thm}
\label{thm:GCfortori2d}
Let $\la=\la_n$ be such that   $ \log n \ll \la n^2 \ll n^2 \log n$.   There exists some $\alpha \in (0,1) $ and $\delta_n=o(1) $ such that  for all $t \gg \hat s_2=\hat s_2(n,\la,\gd_n) $ we have that  \[\PP_{ \la_n}^{\mathbb{T}_2(n)}[\RR_{\hat s_2,t} \text{ is }(\alpha,\hat r_2)\text{-dense}  ]=1-o(1). \] 
\end{thm}
Consider a partition of $\mathbb{T}_2(n) $ into $\hat m_2^2$ boxes $(B_v^{\hat r_2})_{v \in \mathbb{T}_2( \hat m_2) }$ of side length $\hat r_2$.

\begin{prop}
\label{p:good2d}
Let  $B \subseteq B_v^{\hat r_2}$ for some $v \in  \mathbb{T}_2( \hat m_2 ) $. Assume that $|B| \le |B_v^{\hat r_2}|/2$.
 Then there exist some constants $c,C>0$ such that
 \begin{equation}
\label{e:ABlarged2}
 \PP_{\la}\left[|A_{B}(\hat r_2,\hat s_2)| <\frac{3}{4} |B_v^{\hat r_2}|  \right] \le C \exp[-c\la |B|],
\end{equation}
\begin{equation}
\label{e:goodness2d}
\PP_{\la}[\mathrm{good}_{\hat r_2,\hat s_2}(v) \cap B = \eset ]  \le C \exp[-c\la |B|] 
\end{equation}  
\end{prop}
\begin{proof}
The proof is identical to that of Proposition \ref{p:good}.
\end{proof}
\begin{defn}
\label{def:blockdynamics2}
Let $v \in  \mathbb{T}_2(\hat m_{2} )$. As in Definition \ref{def:blockdynamics}, consider the $(\hat r_2,\hat s_2)$  $a$-\emph{dynamics} on $B_v^{\hat r_2}$. We say that $x \in B_v $ is  \emph{\textbf{neat}} if $\RR_{\hat s_2}(\W_{x}^b) \cap \mathrm{good}_{\hat r_2,\hat s_2}(v) \neq \eset  $.  Let $\mathrm{neat}(v)$ be the collection of all neat vertices in $B_v^{\hat r_2}$. Let   $ \mathrm{NEAT}:= \cup_{v \in  \mathbb{T}_2( \hat m_2 )} \mathrm{neat}(v) $.  
 \end{defn}
\begin{prop}
\label{p:neatness2d}
Let  $\hat \ell:= \lceil C_1 \delta_n^{-1} \sqrt{ \frac{1}{\la}\log n }  \rceil $. Provided that $C_1$ is sufficiently large and that $\delta_n$ from the definition of $\hat r_2$ tends to 0 sufficiently slowly  the following hold
\begin{itemize}
\item
For $\la \ge C_{2} \delta_n^{-2} $ we have that $\mathrm{NEAT}$ is $\whp$ $(c, \hat \ell )$-dense for some $c >0$.
\item
For $\la \le C_{2} \delta_n^{-2} $ we have that $\mathrm{NEAT}$ is $\whp$ $(c\la \gd_n^2, \hat \ell )$-dense for some $c >0$.    
\end{itemize}
\end{prop}

\begin{proof}
We partition $B_v^{\hat r_2}$ into sub-boxes of side length $\ell:= \lceil \sqrt{C' \frac{1}{\la}\log n }  \rceil $. We pick $\delta_n$ so that $\ell \ll \hat r_2 $.  By \eqref{e:goodness2d} we may
 pick $C'$ such that for each such sub-box $B$ (of side length $\ell$) we have that $\PP_{\la}[\mathrm{good}_{\hat r_2,\hat s_2}(v) \cap B = \eset ] \le n^{-4}$. By a union bound,  the event that $\mathrm{good}_{\hat r_2,\hat s_2}(v) \cap B  \neq \eset $ for all sub-boxes $B \subset B_v^{\hat r_2} $ of side length $\ell$, for all $v \in \mathbb{T}_2(\hat m_2)$ holds $\whp$. Let us condition on this event. We now argue that each particle in $\mathcal{W}_{B_v}^b $ has probability $p \gtrsim \delta_n^2 $ of visiting $\mathrm{good}_{\hat r_2,\hat s_2}(v)$ (during its length $\hat s_2$ walk).

Consider an arbitrary set $A$ which contains precisely one vertex from each sub-box $B$ of side length $\ell$ of $B_v^{\hat r_2}$ from the aforementioned partition. In order to show (under the above conditioning) each particle in $\mathcal{W}_{B_v}^b $ has probability $p \gtrsim \delta_n^2 $ of visiting $\mathrm{good}_{\hat r_2,\hat s_2}(v)$.
 it is enough to show that for such a set $A$, each particle in $\mathcal{W}_{B_v}^b $ has probability $q\gtrsim \delta_n^2 $ of visiting $A$ (during its length $\hat s_2$ walk).
  Using results from \S\ref{s:green} it is not hard to verify that provided that $\delta_n$ tends to 0 sufficiently slowly we have that
\begin{itemize}
\item[(1)] The expected number of visits to $A$ by a SRW of length $\hat s_2$ started from a vertex in $x \in B_v$, denoted by $e(x)$, satisfies  $\min_{x \in B_v}e(x) \gtrsim  \hat s_2 \ell^{-2} \asymp \delta_n^2 \log (\frac{1}{\la} \log n)$, and
\item[(2)] conditioned on hitting $A$ by time $\hat s_2$, the expected number of visits to $A$ by time  $\hat s_2$, denoted by $e(A)$, satisfies that $e(A)\lesssim \log \hat s_2 \lesssim  \log (\frac{1}{\la} \log n)  $. 
\end{itemize}
We omit the proofs of the last two calculations. It follows that the probability that $T_A \le \hat s_2 $, denoted by $q$, satisfies  $q \gtrsim \frac{ \hat s _2 \ell^{-2} }{ \log \hat s_2 } \gtrsim \delta_n^2 $, as desired.

Let  $\hat \ell:= \lceil C_1 \delta_n^{-1} \sqrt{ \frac{1}{\la}\log n }  \rceil $. We pick  $\delta_n$ so that it tends to 0 sufficiently slowly so that $\hat \ell \ll \hat r_2 $. We now partition $B_v^{\hat r_2}$ into sub-boxes of side length $\hat \ell$. By the previous two paragraphs together with Poisson thinning, it follows that the probability that for a sub-box $B$ in the last partition there are at most  $c_0  \hat \ell ^2 (1- \exp[-c_1 \la \gd_n^2 ] )$ vertices $x$ such that $\RR_{\hat s_2}(\W_x^b) \cap \mathrm{good}_{\hat r_2,\hat s_2}(v) \neq \eset $ is at most $n^{-4}$ (the case that $c_1 \la \gd_n \le 1 $ follows straightforwardly from  \eqref{LDber2}, while the case that $c_1 \la \gd_n > 1 $ follows from  \eqref{LDber1}; We omit the details). Hence by a union bound $\whp$ this does not occur for any such $B \subset B_v^{\hat r_2}$ for all $v  \in \mathbb{T}_2(\hat m_2) $. The proof is concluded by noting that $1- \exp[-c_1 \la \gd_n^2 ] \asymp \la \gd_n^2  $ when $\la \delta_n^2 \le 1 $ and that $1- \exp[-c_1 \la \gd_n^2 ] \asymp 1  $ when $\la \delta_n^2 > 1 $.        
\end{proof}

\begin{lem}
\label{l:nice3}
 Let $\la=\la_n, \hat r_2, \hat m_2$ and $\hat s_{2}$ be as above. Let $t=t(n,\la_{n}) \gg s$. We say that $v \in \mathbb{T}_2(\hat m_{2})$ is $(\hat r_2,\hat s_2 )$-\emph{nice} if $|\RR_t(\plant) \cap B_v^{\hat r_2}| \ge \hat r_2^2 /(16 \log \hat r_2)  $. Let $\mathrm{\widehat {Nice}}:=\{v \in \mathbb{T}_2(\hat m_{2}) :v \text{ is $(\hat r_2,\hat s_2 )$-nice} \}$. Then $\whp$ $|\mathrm{\widehat {Nice}}| \ge 1 $. 
\end{lem}

\noindent \emph{\textbf{Proof of Theorem \ref{thm:GCfortori2d}:}}
 Let $t \gg \hat s_2 $ and  $\mathrm{\widehat {Nice}}\subseteq \mathbb{T}_2( \hat m_{2}) $ be as in Lemma \ref{l:nice3}.  For every $v \in \mathrm{\widehat {Nice}}$ we say that $v$ is \emph{fantastic} if  $\RR_t(\plant) \cap B_v^{\hat r_2} $ is $(\hat r_2,\hat s_2)$-good, where as in Definition \ref{def:blockdynamics}   $B \subseteq B_v^{\hat r_2}$ is good if  $|A_B(\hat r_2,\hat s_2)| \ge \frac{ |B_v^{\hat r_2}|}{4}$. By Lemma \ref{l:nice3} we may condition on $\mathrm{\widehat {Nice}} \neq \eset $. We pick an arbitrary $x=x_{0} \in \mathrm{\widehat {Nice}}  $. By construction and Proposition \ref{p:good2d}, the probability that $x$ is not fantastic is at most  $Ce^{-c \la \hat r_2^2/\log \hat r_2}=o(1) $. Hence we may condition on $x$ being fantastic. Let $\mathrm{GC}(x_0):=A_{\RR_t(\plant) \cap B_{x_{0}}^{\hat r_2}  }(\hat r_2,\hat s_2) $. By construction we have that $|\mathrm{GC}(x_0)| \ge |B_{x_{0}}^{\hat r_2}|/4$ (conditioned on $x_0$ being fantastic).     

Let  $\hat \ell:= \lceil C_1 \delta_n^{-1} \sqrt{ \frac{1}{\la}\log n }  \rceil $ be as in Proposition \ref{p:neatness2d}. By Proposition \ref{p:neatness2d} we may condition on the event that 
 $\mathrm{NEAT}=D$, where $D$ is some $(c \la \gd_n^2 \Ind{ \la \le C_2 \gd_n^{-2}}+c\Ind{ \la > C_2 \gd_n^{-2}} , \hat \ell )$-dense set for some $c ,C_{2}>0$.

Let  $x_1$ be an arbitrary neighbor of $x_0$. Then   $x_{0}+\xi=x_{1} \in \mathbb{T}_2( \hat m_{2}) $ for some $\xi  \in \{\pm e_1,\ldots,\pm e_d \}  $. Observe that if $\RR_{\hat s_2}(\W^{\xi}_{\mathrm{GC}(x_0)}) \cap \mathrm{neat}(x_{1}) \neq \eset  $  then the particles in $\W^{\xi}_{\mathrm{GC}(x_0)}$ will activate some $z \in \mathrm{neat}(x_{1})   $. By the definition of $\mathrm{neat}(x_{1})$, the particles in $\W_z^b $ will activate some $z' \in \mathrm{good}_{\hat r_2,\hat s_2}(x_1)$. We may then define  $\mathrm{GC}(x_1):=A_{z'  }(\hat r_2,\hat s_2) $. By construction we have that $|\mathrm{GC}(x_1)| \ge |B_{x_{1}}^{\hat r_2}|/4$ (assuming  $\RR_{\hat s_2}(\W^{\xi}_{\mathrm{GC}(x_0)}) \cap \mathrm{neat}(x_{1}) \neq \eset  $).     

It suffices to show that if $x+\xi=y \in \mathbb{T}_2( \hat m_{2}) $ for some $\xi  \in \{\pm e_1,\ldots,\pm e_d \}  $ and $A \subseteq B_x^{\hat r_2} $ is of size at least $|B_x^{\hat r_2}|/4 $, then the probability that $\RR_{\hat s_2}(\W_A^{\xi}) \cap \mathrm{neat}(y) = \eset  $ (under the aforementioned conditionings) is $\ll (n/\hat r_2)^{-2} $.     

Pick some set $B \subset  \mathrm{neat}(y) $ of cardinality $ \lceil \delta_n^{-3} \rceil $ such that each pair of vertices in $B$ lie within distance at least $ \frac{1}{4} \delta_n^{3/2} \hat r_2 $ from one another. Since we conditioned on $\mathrm{NEAT}=D$, the set $B$ is non-random. For $b \in B$ let $J_b$ be the number of particles from $\W_A^{\xi} $ which reached $b$ in their length $\hat s_2$ walks. Let $Q_b$ be the number of particles from $\W_A^{\xi} $ which reached $b$ in their length $\hat s_2$ walk, which did not reach any other vertex in $B$ in their length $\hat s_2$ walk. As in the proof of Lemma \ref{lem: spreadnew22d} we have that for $b \in B$ both  $J_b$ and $Q_b$ have Poisson distributions and that $\mathbb{E}_{\la}[Q_b]=(1-o(1))\mathbb{E}_{\la}[J_b] \gtrsim \la \hat s_2 /\log \hat s_2 \gtrsim \gd_n^{2} \log n$, provided that $\gd_n$  tends to 0 sufficiently slowly. By independence it follows that $- \log \PP_{\la}[ \sum_{b \in B} Q_b =0 ] \gtrsim \gd_n^{-1} \log n  $. \qed

\subsection{$d \ge 3$}
\label{s:dge3}
Throughout this subsection we let $d \ge 3$ and $\la=\la_n $ be such that $\log n \ll \la n^d \ll n^d \log n $ (even when this is not explicitly specified). Let $ s= s(n,\la,\gd_n):=8d k^2 $, where
\begin{equation}
\label{e:rhat2}
 k=k(n,\la,\gd_n):=\left\lceil \gd_n \sqrt{\frac{1}{\la} \log n  } \right\rceil,
\end{equation}
for some $ \gd_{n} = o(1) $ such that $k\gg 1$, to be determined later. Let $m=m(n,\la, \delta_n):=\lfloor n/ k \rfloor $.
\begin{thm}
\label{thm:GCfortori3d}
Let $d \ge 3$. Let $\la=\la_n$ be such that   $ \log n \ll \la n^d \ll n^d \log n$.   There exists some $\alpha \in (0,1) $ and $\delta_n=o(1) $ such that  for all $t \gg s=s(n,\la,\gd_n) $ we have that  \[\PP_{ \la_n}^{\mathbb{T}_d(n)}[\RR_{s,t} \text{ is }(\alpha,k(n,\la,\gd_n))\text{-dense}  ]=1-o(1). \] 
\end{thm}
Consider a partition of $\mathbb{T}_d(n) $ into $m$ boxes $(B_v^{k})_{v \in \mathbb{T}_d(m) }$ of side length $k$.
\begin{prop}
\label{p:good3d}
Let  $B \subseteq B_v^{k}$ for some $v \in  \mathbb{T}_d(m) $. Assume that $|B| \le |B_v^{k}|/2$.
 Then there exist some constants $c,C>0$ such that
 \begin{equation}
\label{e:ABlarged3}
 \PP_{\la}[|A_{B}(k,s)| <\frac{3}{4} |B_v^{k}|  ] \le C \exp[-c(d)\la |B|],
\end{equation}
\begin{equation}
\label{e:goodness3d}
\PP_{\la}[\mathrm{good}_{k,s}(v) \cap B = \eset ]  \le C \exp[-c(d)\la |B|]. 
\end{equation}  
\end{prop}
\begin{proof}
The proof is identical to that of Proposition \ref{p:good}.
\end{proof}
\begin{defn}
\label{def:blockdynamics3}
Let $d \ge 3$. Let $v \in  \mathbb{T}_d(m )$. As in Definition \ref{def:blockdynamics}, consider the $(k,s)$  $a$-\emph{dynamics} on $B_v^{k}$. We say that $x \in B_v^{k} $ is  \emph{\textbf{neat}} if $\RR_{s}(\W_{x}^b) \cap \mathrm{good}_{k,s}(v) \neq \eset $.  Let $\mathrm{neat}(v)$ be the collection of all neat vertices in $B_v^{k}$. Let   $ \mathrm{NEAT}:= \cup_{v \in  \mathbb{T}_d(m )} \mathrm{neat}(v) $.  
 \end{defn}
\begin{prop}
\label{p:neatness3d}
Let  $\hat \ell:= \lceil C_1  ( \frac{1}{\la \delta_n^{2}}\log n )^{1/d}  \rceil $. Provided that $C_1$ is sufficiently large and that $\delta_n$ from the definition of $k$ tends to 0 sufficiently slowly  the following hold
\begin{itemize}
\item
For $\la \ge C_{2} \delta_n^{-2} $ we have that $\mathrm{NEAT}$ is $\whp$ $(c, \hat \ell )$-dense for some $c >0$.
\item
For $\la \le C_{2} \delta_n^{-2} $ we have that $\mathrm{NEAT}$ is $\whp$ $(c\la \gd_n^2, \hat \ell )$-dense for some $c >0$.    
\end{itemize}
\end{prop}

\begin{proof}
We partition $B_v^{k}$ into sub-boxes of side length $\ell:= \lceil ( C' \frac{1}{\la}\log n )^{1/d}  \rceil $. We pick $\delta_n$ so that $\ell \ll k $.  By \eqref{e:goodness3d} we may pick $C'$ such that for each such sub-box $B$ of side length $\ell$ we have that $\PP_{\la}[\mathrm{good}_{k,s}(v) \cap B = \eset ] \le n^{-4}$. By a union bound,  the event that $\mathrm{good}_{k,s}(v) \cap B  \neq \eset $, for all such $B \subset B_v^{k} $, for all $v \in \mathbb{T}_d(m)$, holds $\whp$. We condition on this event. We now argue that each particle in $\mathcal{W}_{B_v^{k}}^b $ has probability at least $p \gtrsim \delta_n^2 $ of visiting $\mathrm{good}_{k,s}(v)$.

By the last conditioning, it suffices to consider an arbitrary set $A$ which contains precisely one vertex from each sub-box $B$ of side length $\ell$ of $B_v^{k}$ from the aforementioned partition, and show that  each particle in $\mathcal{W}_{B_v^{k}}^b $ has probability at least $q \gtrsim \delta_n^2 $ of visiting $\mathrm{good}_{k,s}(v)$. Using results from \S\ref{s:green} it is not hard to verify  that provided that $\delta_n$ tends to 0 sufficiently slowly we have that
\begin{itemize}
\item[(1)] the expected number of visits to $A$ by a SRW of length $s$ started from a vertex  $x \in B_v^{k}$, denoted by $e(x)$, satisfies $\min_{x \in B_v^k} e(x)\gtrsim s \ell^{-d} \asymp \delta_n^2    $, and
 \item[(2)] conditioned on hitting $A$ by time $s$,  the expected number of visits to $A$ by time  $s$, denoted by $e(A)$, satisfies $e(A)=O(1)  $.
\end{itemize}
 We omit the details of the last two calculations. It follows that the probability that $T_A \le s $, denoted by $q$, satisfies $q \gtrsim \delta_n^2 $, as desired.

Let  $\hat \ell:= \lceil C_1  ( \frac{1}{\la \delta_n^{2}}\log n )^{1/d}  \rceil $. Assume that $\delta_n$ tends to 0 sufficiently slowly so that $\hat \ell \ll k $. We now partition $B_v^{k}$ into sub-boxes of side length $\hat \ell$. By the previous two paragraphs together with Poisson thinning, it follows that the probability that for a sub-box $B$ in the last partition there are less than $c_0  \hat \ell ^d (1- \exp[-c_1 \la \gd_n^2 ] )$ vertices $x$ such that $\RR_{s}(\W_x^b) \cap \mathrm{good}_{k,s}(v) \neq \eset $ is at most $n^{-4}$ (the case that $c_1 \la \gd_n \le 1 $ follows straightforwardly from  \eqref{LDber2}, while the case that $c_1 \la \gd_n > 1 $ follows from  \eqref{LDber1}; This is left as an exercise). Hence by a union bound we may condition that this does not occur for any such $B \subset B_v^{k}$ for all $v  \in \mathbb{T}_d(m) $.       
\end{proof}

\begin{lem}
\label{l:nice4}
 Let $\la=\la_n, k, m$ and $ s$ be as above. Let $t=t(n,\la_{n},\delta_n) \gg s$. We say that $v \in \mathbb{T}_d(m)$ is $(k,s )$-\emph{nice} if $|\RR_t(\plant) \cap B_v^{k}| \ge k^{2}/16   $. Let $\mathrm{\widehat {Nice}}:=\{v \in \mathbb{T}_d(m) :v \text{ is $(k,s )$-nice} \}$. Then $\whp$ $|\mathrm{\widehat {Nice}}| \ge 1 $. 
\end{lem}

\emph{Proof of Theorem \ref{thm:GCfortori3d}:}
 Let $t \gg s $ and  $\mathrm{\widehat {Nice}}\subseteq \mathbb{T}_d(m) $ be as in Lemma \ref{l:nice4}.  For every $v \in \mathrm{\widehat {Nice}}$ we say that $v$ is \emph{fantastic} if  $\RR_t(\plant) \cap B_v^{k} $ is $(k,s)$-good, where as in Definition \ref{def:blockdynamics} $B \subseteq B_v^{k}$ is good if $|A_B(k,s)| \ge |B_v^{k}|/4$. By Lemma \ref{l:nice4} we may condition on $\mathrm{\widehat {Nice}} \neq \eset $. We pick an arbitrary $x=x_{0} \in \mathrm{\widehat {Nice}}  $. By construction and Proposition \ref{p:good3d}, the probability that $x$ is not fantastic is at most  $Ce^{-c \la k^2}=o(1) $. Hence we may condition on $x$ being fantastic. Let $\mathrm{GC}(x_0):=A_{\RR_t(\plant) \cap B_{x_{0}}^{k}  }(k,s) $. By construction we have that $|\mathrm{GC}(x_0)| \ge |B_{x_{0}}^{k}|/4$.     

Let  $\hat \ell:= \lceil C_1  ( \frac{1}{\la \delta_n^{2}}\log n )^{1/d}  \rceil $ be as in Proposition \ref{p:neatness3d}. By Proposition \ref{p:neatness3d} we may condition on the event that 
 $\mathrm{NEAT}=D$, where $D$ is some $(c \la \gd_n^2 \Ind{ \la \le C_2 \gd_n^{-2}}+c\Ind{ \la > C_2 \gd_n^{-2}} , \hat \ell )$-dense for some $c ,C_{2}>0$.

Imitating the proof of Theorem \ref{thm:GCfortori2d} it suffices to show that if $x+\xi=y \in \mathbb{T}_d( m) $ for some $\xi  \in \{\pm e_1,\ldots,\pm e_d \}  $ and $A \subseteq B_x^{k} $ is of size at least $|B_x^{k}|/4 $, then the probability $p$ that $\RR_{s}(\W_A^{\xi}) \cap \mathrm{neat}(y) = \eset  $ (under the aforementioned conditionings) satisfies $p \ll (n/k)^{-d} $.     

Pick some set $B \subset  \mathrm{neat}(y) $ of cardinality $ \lceil \delta_n^{-3} \rceil $ such that each pair of vertices in $B$ lie within distance at least $ \frac{1}{4} \delta_n^{3/d} k $ from one another. Since we conditioned on $\mathrm{NEAT}=D$, the set $B$ is non-random. For $b \in B$ let $J_b$ be the number of particles from $\W_A^{\xi} $ which reached $b$ in their length $s$ walks. Let $Q_b$ be the number of particles from $\W_A^{\xi} $ which reached $b$ in their length $s$ walk, which did not reach any other vertex in $B$ in their length $s$ walk. As in the proof of Lemma \ref{lem: spreadnew2} we have that for $b \in B$ both  $J_b$ and $Q_b$ have Poisson distributions and that $\mathbb{E}_{\la}[Q_b]=(1-o(1))\mathbb{E}_{\la}[J_b] \gtrsim \la s\gtrsim \gd_n^{2} \log n$, provided that $\gd_n$ is tends to 0 sufficiently slowly. By independence it follows that $- \log \PP_{\la}[ \sum_{b \in B} Q_b =0 ] \gtrsim \gd_n^{-1} \log n  $. \qed

\section{Expanders - Proof of Theorem \ref{thm: et1}}
\label{s:expanders}
In this section we study the case that $G$ is a $d$-regular expander. It is not difficult to extend the results to the case $G$ is an expander of maximal degree $d$. We note that the arguments presented in this section are inspired by  techniques from \cite[Theorem 5]{SN} and \cite[Theorem 3]{Hermon}. However, the analysis below includes some new ideas. In particular, the usage of a maximal inequality in the proof of Theorem \ref{thm:GCexp} (through the application of Lemma \ref{lem:maxineq}) is novel. Moreover, the analysis of the case $ \la \gg 1$ requires some new ideas.

\medskip

Consider the case that $\plant$, the planted walker at $\oo$, walks for $t=t_{|V|}$ steps, while the rest of the particles have lifespan $ M$ for some constant $M$. Recall that the set of vertices which are visited by this modified process  before it dies out is denoted by  $\mathcal{R}_{t,M}$. Similarly, consider the variation of the model in which there is no planted particle and initially the collection of particles occupying some set $A \subset V$ are activated. In this variant let the lifespan of all of the particles be $s$. Denote the set of vertices which are visited in this variant of the model  before the process dies out by  $\mathcal{R}_{s}^A$.

\begin{thm}
\label{thm:GCexp}
There exist absolute constants $c,C,C',M>0$ such that for every  $n \ge C' $,   $\la \ge  C n^{-1} \log n$ and   $\gamma \in (0,2] $, for every  regular $n$-vertex  $\gamma$-expander  $G=(V,E)$, we have
\begin{equation}
\label{e:GCexpt}
\forall A \subset V, \quad \PP_{\la}[|\mathcal{R}_{\lceil M \max \{  \la^{-1},1\} \gamma^{-1} \rceil}^{A}(G)|<n/4] \le Ce^{-c \la |A|}. \end{equation}  
\end{thm}

We now argue that Theorem \ref{thm: et1} when $\la \le 1$ follows from Theorem \ref{thm:GCexp} in conjunction with Lemma \ref{lem:rangeexp} and Corollary \ref{cor: prey}. 

\subsection{Proof of Theorem \ref{thm: et1} when $1/n \ll \la \le 1$ given Theorem \ref{thm:GCexp}}
\begin{proof}Let $\la \in [C n^{-1} \log n,1]$, where $C>0$ is as in Theorem \ref{thm:GCexp}. We first note that by  \eqref{e:submulrang} with $t=C \la^{-1} \gamma^{-1} \sqrt{ \log n}$,  $s=\lfloor \sqrt{\log n} \rfloor $, and $r= \la^{-1} \sqrt{\log n}$ and by Lemma \ref{lem:rangeexp} with the same choice for $t$ we have that \begin{equation}
\label{e:Rplantexp}
\PP_{\la}[|\RR_{\lceil C \la^{-1} \gamma^{-1} \log n \rceil }(\plant)|< \la^{-1} \sqrt{\log n} ] \le 4^{-s} \le e^{-c' \sqrt{\log n} }.
\end{equation}
For the remainder of the proof we fix some  $A \subseteq V$ of size at least $ \la^{-1} \sqrt{\log n}$ and
 condition on the event $\RR:= \RR_{\lceil C \la^{-1} \gamma^{-1} \log n \rceil }(\plant)=A $. 
We can partition the particles in $\W \setminus \{ \plant \}$ into two independent sets, each with density $\la/2$. We refer to the particles belonging to the first (resp.~second) set as type 1 (resp.~2) particles. We can apply Theorem  \ref{thm:GCexp} to the dynamics associated with the type 1 particles (as if there are no type 2 particles) with lifespan $\lceil M\la^{-1} \gamma^{-1} \rceil$. This dynamics is precisely the frog model with parameter $\la/2$ and the aforementioned lifetime, where initially $\W_A$ is activated. Denote by $B$ the collection of vertices visited by the type 1 dynamics before it dies out. By Theorem  \ref{thm:GCexp} (and our conditioning on $\RR=A$ for $|A| \ge \la^{-1} \sqrt{\log n} $) we have that $|B| \ge n/4 $ with probability at least $1-e^{-c_0 \sqrt{\log n} }$.  By Corollary \ref{cor: prey}, given that $|B| \ge n/4$, with probability at least $1-1/n$ we have that $V$ equals the union of the ranges of the walks of length  $\lceil
C \lambda^{-1}\gamma^{-1} \log 
n \rceil $  performed by the type 2 particles initially occupying $B$.

Finally, when $1/n \lll \la \le C n^{-1} \log n$ we first argue that $\whp$ the planted particle visits by time   $k=\lceil
C \lambda^{-1}\gamma^{-1} \log 
n \rceil $ at least $n/32 $ distinct vertices where now we use   \eqref{e:submulrang} with $t=\lceil C n\gamma^{-1} \rceil $,  $s=\lfloor k/t \rfloor $, and $r= n/32$ and Lemma \ref{lem:rangeexp} with the same choice for $t$. The proof is concluded using \eqref{e:lastihope}.  
    \end{proof}

\subsection{Proof of Theorem \ref{thm:GCexp}}
\label{s:GCexp}
\begin{proof}[Proof of Theorem~\ref{thm:GCexp}]  We use an exploration process due to Benjamini, Nachmias and Peres \cite{pclocal}. Let  $\gamma$ be the spectral gap of SRW on $G=(V,E)$.  Initially the collection of active particles is $\W_A$ for some $A \subset V $. Let $\la_0:=\min \{\la,16\}$ and \[t=t_{\la}:=\lceil L/(\la_{0} \gamma ) \rceil, \] where $L>0$ is some absolute constant to be determined shortly.  Recall that for a collection of particles $\W' \subseteq \W $ we  denote the union of the ranges of the length $\ell$ walks performed by the particles in $\W'$ by $\RR_{\ell}(\W')$ (with the convention that $\RR_{\ell}(\eset)=\eset $). Recall that  $\W_v $  is the collection of particles whose initial position is $v$.
\medskip

 Denote  $\kappa:= \lceil16 \la_{0}^{-1} \log (2^9) \rceil $. Note that for all $\la $ we have that
\begin{equation}
\label{e:poiskapppa}
\PP[\Pois(\la /16) \neq 0]\kappa  \ge 2.  
\end{equation}
As in Corollary \ref{cor:largeE}, for a set $B \subset V $ let
\[G_{B}:=\{a \in B: \Pr_{a}[|\{ X_{s}:s \in [t]\} \setminus B  | \ge \kappa] \ge 1/16   \}. \]
By Corollary \ref{cor:largeE}  for every $B$ of size at most $n/4$ we have that $|G_B| > \frac{3}{4}|B|$, provided that $L$ is taken to be sufficiently large. For every $B \subseteq V$ and $b \in B$ let \[\W_b^B:=\{w \in \W_b : |\RR_t(w) \setminus B| \ge \kappa  \}. \]
Observe that if $b \in G_B $, then $|\W_b^{B}| $ has a Poisson distribution with parameter at least $\la/16$ and hence by \eqref{e:poiskapppa}  $\PP_{\la}[|\W_b^{B}| \neq 0 ]\kappa  \ge 2 $.

\medskip

Initially set $\AA_0:=A  $ and $\mathcal{U}_0=\eset$. We now define inductively a collection of random sets of vertices $\AA_1,\AA_2,\ldots  $ and $\mathcal{U}_1,\mathcal{U}_2,\ldots $. For every $i \ge 0$ let $\mathcal{B}_i:=\AA_i \cup \mathcal{U}_i$.
Assume that we have already defined $\AA_0,\AA_1,\ldots,\AA_j $ and $\mathcal{U}_0, \mathcal{U}_1,\ldots,\mathcal{U}_j$ in the following manner:
\begin{itemize}
\item[(1)]  $|\mathcal{B}_j|<n/4$ and for all $1\le i<j$ we have $| \mathcal{B}_i| \ge \frac{4}{3}i $ and  $|\mathcal{U}_i|=i $.  Thus (using $G_{\mathcal{B}_i} \subseteq \mathcal{B}_i=\AA_i \cup \mathcal{U}_i  $ and $|G_{\mathcal{B}_i}| > \frac{3}{4}|\mathcal{B}_i| \ge i$)  it must be the case that $G_{\mathcal{B}_i} \cap \AA_i \neq \eset $.

\item[(2)] For each $1 \le i \le j $ at stage $i$ of the exploration process we expose $\RR_{t }(\W_{v_i}^{\mathcal{B}_{i-1}})$ for some vertex $v_i \in \AA_{i-1} \cap G_{\mathcal{B}_{i-1}}$. We then set \[ \mathcal{U}_{i}:=\mathcal{U}_{i-1} \cup \{v_i \} \quad \text{and} \quad \AA_{i} := \AA_{i-1} \cup\RR_{t }(\W_{v_i}^{\mathcal{B}_{i-1}})\setminus\mathcal{U}_{i}.\] 
\end{itemize}
The exploration process is terminated at the first stage $j$ at which either $|\mathcal{B}_j| \ge n/4 $ or  $|\mathcal{B}_j| \le \frac{4}{3}j$.  Let $J$ be the stage at which the exploration process is terminated. It is not hard to see that conditioned on $i<J$ we have that $|\mathcal{B}_{i+1} \setminus \mathcal{B}_{i}| $ stochastically dominates a binary random variable which equals $\kappa$ w.p.~$1-e^{-\la/16}$ and otherwise equals 0. 

Consider a sequence $\xi_1,\xi_2,\ldots $  of i.i.d.\ Bernoulli($1-e^{-\la/16}$) random variables. Let us first consider the case that $\la \le 16$. In this case, we have that $\kappa= \lceil16 \la^{-1} \log (2^9) \rceil $ and $1-e^{-\la/16} \ge \la /32 $. By the previous paragraph the probability that $\PP_{\la}[\mathcal{|B}_{ J }| < n/4]$ is at most the probability that for some $i \ge \frac{3}{4}|A| $ we  have that  $\kappa \sum_{j=1}^i \xi_i  \le \frac{4}{3  }i \le \kappa i(1-e^{-\la/16})(1-\frac{1}{3} ) $. By \eqref{LDber4} (with $p=1-e^{-\la/16}$, $k= \lceil \frac{3}{4}|A| \rceil $ and $\delta=\frac{1}{3}  $) this probability decays exponentially in $\la |A|$. 

We now consider the case that $\la \ge 16$. In this case $\kappa= \lceil16 \log (2^9) \rceil $.  Consider a sequence $\eta_1,\eta_2,\ldots $  of i.i.d.\ Bernoulli($e^{-\la/16}$) random variables.  The probability that $\PP_{\la}[\mathcal{|B}_{ J }| < n/4]$ is at most the probability that for some $i \ge \frac{3}{4}|A| $ we  have that  \[ \sum_{j=1}^i \eta_i  \ge i- \frac{4}{3 \kappa  }i  \ge ie^{-\la/16}( 4e^{\la/16}/5 ) .\] By \eqref{LDber3} (with $p=e^{-\la/16}$, $k= \lceil \frac{3}{4}|A| \rceil $ and $\delta=4 e^{\la/16}/5-1  $) this probability decays exponentially in $\la |A|$.      
\end{proof}
\subsection{Proof of Theorem \ref{thm: et1} when $\la \ge 1$}
\begin{proof}
Let $1 \le \la \le \frac{1}{2} \log n$. Denote $t:=\lceil C \gamma^{-1}  \la^{-1} \log n  \rceil $. Let $H$ be the event that $|\RR_{t }(\plant)|< \la^{-1} \sqrt{\log n}$.  
By \eqref{e:Rplantexp} we have that $\PP[H ] \le e^{-c' \sqrt{\log n} } $. Thus by Theorem \ref{thm:GCexp}  
 \[\PP_{\frac{\la}{2}}\left[|\RR_t| < \frac{n}{4} \right] \le \PP[H ]+ \PP_{\frac{\la}{2}}\left[|\RR_t| < \frac{n}{4} \mid H^{c} \right] \le 2 e^{- \hat c \sqrt{\log n} },   \]
As in the proof of the case $\la \le 1$, it suffices to consider the case that the particle density is $  \la/2  $ and initially there is some set $A \subset V$ of size $\lceil n/4 \rceil$ which is activated (with no planted particles).  We partition the particles into $\beta = \lfloor \la  \rfloor $ independent sets, $\W^1,\ldots,\W^{\beta}$ each of density at least $\half$. For all $v \in V$ and $i \in [\beta]$ we denote the particles from $\W^i$ which initially occupy $v$ by $\W_v^i$. Then $(\W_v^i)_{v \in V,i \in [\beta]}$ stochastically dominate i.i.d.~$\Pois(1/2)$. For a set $B \subseteq V$ and $i \in [\beta]$ let $\W_{B}^i:=\cup_{b \in B}\W_b^i $ be the collection of particles from the $i$th set initially occupying $B$. Let $A_0=A$. We define inductively for all $i \in [\beta -1]$ \[A_{i+1}:=A_i \cup \RR_{2t}(\W_{A_{i}}^{i+1}).  \]  
That is, at stage $i$ we use the particles from the $(i+1)$th set which initially occupy the currently exposed set in order to reveal additional vertices. Our goal is to estimate the probability that $A_{\beta} \neq V$. Our strategy is to show that
\begin{equation}
\label{e:EAi}
\E[|A_{i+1}^{c}| \mid |A_i| ] \le n^{-4\la^{-1} }|A_i^c|, \quad \text{for all }i \in [\beta -1]. \end{equation}
We first explain how the proof is concluded using \eqref{e:EAi}. By \eqref{e:EAi} we have
\[\E[|A_{\beta}^c|] \le |A^c|(n^{-4\la^{-1} })^{\beta} \le n \cdot n^{-2}=n^{-1}. \]
Before we begin to prove \eqref{e:EAi} we need some preliminaries.

\medskip

Let $\mu$ be a distribution on $V$. Recall that the $\ell_2$ distance of $\mu$ from uniform distribution $\pi$ is defined as \[ \|\mu - \pi \|_{2,\pi}:=\left(\sum_{v \in V}\pi(v)\left(\frac{ \mu(v)}{\pi(v)} -1 \right)^2\right)^{1/2}= \left(n \sum_{v \in V}\left(\mu(v) -n^{-1} \right)^2\right)^{1/2}. \]
Denote the distribution of the lazy SRW at time $s$ (resp.~of the entire lazy SRW) started from initial distribution $\mu$ by $\Pr_{\mu}^s $ (resp.~$\Pr_{\mu}$). By the Poincar\'e inequality we have that \[\|\Pr_\mu^s - \pi \|_{2,\pi}^2  \le (1-\gamma/2)^{2s} \|\mu - \pi \|_{2,\pi}^2, \quad   \text{for all }s\ge 0.\]
Recall that $\pi_B$ is the uniform distribution on $B$. An easy calculation shows that for every $B \subset V $ we have that  $\|\pi_{B} - \pi \|_{2,\pi}^2=\frac{\pi(B^c)}{\pi(B)}$. Hence if $s \ge t,$  
\[ \|\Pr_{\pi_{B}}^s - \pi \|_{2,\pi}^2  \le (1-\gamma/2)^{2s} \|\pi_{B} - \pi \|_{2,\pi}^2 \le e^{-\gamma s} \frac{\pi(B^c)}{\pi(B)}  \le n^{-C/ \la }\frac{\pi(B^c)}{\pi(B)}. \] 
Denote $\mu_{B}:=\frac{1}{t}\sum_{s=t+1}^{2t}\Pr_{\pi_{B}}^s $. By convexity also $\|\mu_{B} - \pi \|_{2,\pi}^2 \le n^{-C/ \la }\frac{\pi(B^c)}{\pi(B)} $. Consider the set \[J_{B}:=\{v \in V : \mu_B(b) \le 1/2 \}. \]
Note that $\pi(J_B)/4 \le \|\mu_{B} - \pi \|_{2,\pi}^2 $. Thus $\pi(J_B) \le 4 n^{-C/ \la }\frac{\pi(B^c)}{\pi(B)} $.

\medskip

We are now in the position to prove \eqref{e:EAi}. First recall that we may assume the particles are performing lazy SRW. Let $0 \le i < \beta$. Since $\pi(A_{0}) \ge 1/4 $ we have that $\pi(J_{A_{i}}) \le 16  n^{-C /\la } \pi(A_i^c)$. Fix some $x \notin J_{A_i}$. Since $\mu_{A_i}(x)>1/2 $ the expected number of visits to $x$ (with multiplicities) by a particle from $\W_{A_i}^{i+1}$ (from time $t+1$ to time $2t$) is at least $\frac{t}{2} \pi(A_i) \ge t/8$. 

By \eqref{e:Green} 
$ \sum_{i=0}^{2t}P_{L}^i(x,x) \le 2\gamma^{-1}+2t/n \le 3 \gamma^{-1}$, provided that $n$ is sufficiently large. Again by \eqref{e:Green} for each $x \notin J_{A_i}$ the number of particles from $\W_{A_i}^{i+1} $ which visit $x$ in the first $2t$ steps of their walks has a Poisson distribution with parameter at least $\frac{t/8}{3 \gamma^{-1}} \ge 8 \la^{-1} \log n $ (provided that $C$ is taken to be sufficiently large). It follows that the expected number of $x \in A_i^c \setminus J_{A_i} $ that do not belong to $A_{i+1}$ is at most $|A_i^c| e^{-8 \la^{-1} \log n}=|A_i^c|n^{-8 \la^{-1} }$. Finally, using  $\pi(J_{A_{i}}) \le 16  n^{-C \la^{-1} } \pi(A_i^c)$, we get that 
 \begin{equation*}
 \begin{split}
 \E[|A_{i+1}^{c}| \mid |A_i| ] & \le n^{-8\la^{-1} }|A_i^c|+n \pi(J_{A_i}) \\ & \le n^{-8\la^{-1} }|A_i^c|  +16  n^{-C \la^{-1} }|A_i^c|   \le n^{-4\la^{-1} }|A_i^c| ,
\end{split}
\end{equation*} provided that $C$ is sufficiently large.
\end{proof}
\appendix
\section{Appendix A: Proofs of some remarks}
\subsection{Sketch of proof of the assertion of Remark \ref{rem:lalog}}
\label{s:lalog}

Let $G=(V,E)$ be a connected $d$-regular graph. Let $\la \ge (1+\delta) d \log |V|$ for some constant $\delta>0$. We argue that $\PP_{\la}[\SS(G)>1] \le |V|^{-\gd}$. To see this consider an arbitrary spanning tree of $G$, rooted at $\oo$, whose edges are oriented away from $\oo$. Let  $v \in V \setminus \{\oo\} $. Denote its parent by $u$. If $u$ is activated before the process dies out, then the probability that $v$ is not activated (by some particle from $u$) is at most $|V|^{-(1+\delta)}$ (by Poisson thinning).

Now consider the case $G=\mathbb{T}_d(n)$.
We argue that if $\la \ge (2d+\delta)\log n$ for some $\delta>0$, then $\PP_{\la}[ \SS(\mathbb{T}_d(n)) > 1   ] \le Cn^{-\delta/2}$. To see this, observe that (by Poisson thinning) for every $v$ of distance at least 2 from $\oo$, the number of particles that move to $v$ in their first step, which initially occupy some neighbor of $v$ which is closer to $\oo$, has a Poisson distribution with parameter $d \cdot (\frac{\la}{2d})=\la/2$.

\subsection{Sketch of proof of \eqref{e:H(d,n)}}
\label{s:H(d,n)}
The term $n^2$ corresponds to the (expected) cover time of $\SS(H_{d,n}) $ by $\plant$ (up to a constant factor). This is also roughly the time required for $\plant$ to visit at least half of the copies of $J_d$. We leave the details as an exercise (hint: the walk typically spends $ \asymp d^2$ time units at each copy of $J_d$, and the time it takes SRW on the $\lceil n/d \rceil $-cycle to visit half of the vertices is typically $\asymp n^2/d^2 $).

We now briefly explain the remaining terms on the r.h.s.~of \eqref{e:H(d,n)}.   Note that if there are at least two edges, connecting distinct copies of $J_d$, which were not crossed by a single particle, then deterministically some vertices were not visited. 

  The term $ ds $ is obtained from the estimate that the number of particles which crossed each edge connecting two copies of $J_d$ for lifespan $t$ and particle density $\la$ is stochastically dominated by the Poisson distribution with parameter $\la t/d$ (so one needs to take $t \asymp ds  $ to ensure the expected number of such uncrossed edges is not large). 

The term $s^2$ comes from the fact that for lifespan $t$ and particle density $\la$,  for each copy of $J_d$,  the  number of particles initially not occupying it which visit it is stochastically dominated by the Poisson distribution with parameter $c_{0} \la \sqrt{t}$ (cf.\ the proof of Theorem  \ref{thm: cycle}). Again, to ensure that the expected number of such unvisited copies of $J_d$ is not large one needs to take $t \asymp s^2  $. The argument can be made precise via a second moment calculation similar to the one from the proof of Theorem \ref{thm: cycle}.
\section{Appendix B - Large deviation estimates for sums of Bernoulli random variables}
\begin{fact}
\label{f:LDber}
Let $\xi_1,\xi_2,\ldots $ be i.i.d.\ Bernoulli random variables of mean $p \in (0,1)$. Let $S_n:=\sum_{j=1}^n \xi_i $. Then\begin{equation}
\label{LDber1}
\forall \delta \ge 0, \quad \Pr[S_n \ge n p(1+ \delta )  ] \le \frac{(1+p\delta )^n}{(1+\delta)^{(1+\delta)np}} \le \exp \left(- \frac{np \delta \log (1+\delta)}{4} \right).
\end{equation}
\begin{equation}
\label{LDber2}
\forall \delta \in (0,1), \quad \Pr[S_n \le n p(1- \delta )  ] \le \frac{(1-p\delta )^n}{(1-\delta)^{(1-\delta)np}} \le \exp\left(- \frac{n p\delta^2}{4} \right).
\end{equation}
\begin{equation}
\label{LDber3}
\forall \delta \ge 1, k \in \N \quad \Pr[\sup_{n:n \ge k} S_n /n\ge  p(1+ \delta )  ] \le \frac{\exp(- \frac{kp \delta \log (1+\delta)}{4} )}{1-\exp(- \frac{p \delta \log (1+\delta)}{4} )} .
\end{equation}
\begin{equation}
\label{LDber4}
\forall \delta \in (0,1),k \in \N \quad \Pr[\sup_{n:n \ge k}S_n/n \le  p(1- \delta )  ] \le 8\delta^{-2} \exp\left(- \frac{k p\delta^2}{4} \right).
\end{equation}
\end{fact}
\begin{proof} We first prove \eqref{LDber1}. Let $t = \log (1+\delta) $. Then $\mathbb{E}[e^{tS_n}]=(pe^t+(1-p))^n=(1+p\delta )^n $. Thus \[\Pr[S_n \ge n p(1+ \delta )  ]=\Pr[e^{tS_n}\ge e^{t n p(1+ \delta) }  ] \le e^{-t n p(1+ \delta) }\mathbb{E}[e^{tS_n}]=\frac{(1+p\delta )^n}{(1+\delta)^{(1+\delta)np}} .\] It is not hard to verify that $ \frac{(1+p\delta )^n}{(1+\delta)^{(1+\delta)np}} \le \exp(- \frac{np \delta \log (1+\delta)}{4} )$. 

We now prove \eqref{LDber2}.   Let $r = \log (1-\delta) $. Then $\mathbb{E}[e^{rS_n}]=(pe^r+(1-p))^n=(1-p\delta )^n $. Thus $\Pr[S_n \le n p(1- \delta )  ]=\Pr[e^{rS_n}\ge e^{r n p(1- \delta) }  ] \le e^{-r n p(1- \delta) }\mathbb{E}[e^{tS_n}]=\frac{(1-p\delta )^n}{(1-\delta)^{(1-\delta)np}} $. With some additional algebra it is not hard to verify that  $ \frac{(1-p\delta )^n}{(1-\delta)^{(1-\delta)np}} \le \exp(- \frac{np \delta ^{2}}{4} )$.

We now prove \eqref{LDber3}. Let $\eps \in (0,\delta/2)$. Then by \eqref{LDber1}
\begin{equation*}
\begin{split} \Pr\left[\sup_{n:n \ge k} S_n /n\ge  p(1+ \delta )  \right] & \le \sum_{n=k}^{\infty} \Pr[S_n \ge n p(1+ \delta -\eps )  ] \\ & \le \frac{\exp(- \frac{kp \delta \log (1+\delta - \eps)}{4} )}{1-\exp(- \frac{p \delta \log (1+\delta -\eps)}{4} )}.
  \end{split}
 \end{equation*}
Sending $\eps$ to 0 concludes the proof of \eqref{LDber3}. We now prove \eqref{LDber4}. Let $\eps \in (0,\delta/2)$.  Let $J:=\{n \ge k: S_n \le p(1- \delta+ \eps )\}$. It is easy to see that $\mathbb{E}[J \mid J \ge 1 ] \ge 1/p $. Thus by \eqref{LDber2}
\[\Pr[J \ge 1 ]= \frac{\mathbb{E}[J ] }{\mathbb{E}[J \mid J \ge 1 ] } \le p \sum_{n =k}^{\infty}\Pr[S_n \le n p(1- \delta + \eps )  ] \le \frac{ p \exp[- \frac{k p(\delta-\eps)^2}{4} ]}{1-\exp[- \frac{p(\delta-\eps)^2}{4} ]}.    \]
Sending $\eps$ to 0 and noting that $1-e^{-p \delta^{2}/4 } \ge p \delta^{2}/8 $ concludes the proof.
\end{proof}

\section*{Acknowledgements}
The authors would also like to thank the anonymous referees for suggesting substantial improvements to the presentation.

\bibliographystyle{plain}
\bibliography{Frog}

\end{document}